\documentclass[11pt]{article}
\usepackage{a4wide,amsmath,amssymb,latexsym,amsthm,mathtools}
\usepackage{booktabs,array,bbm,xcolor,enumitem,xargs,subcaption}
\usepackage{authblk,xr}
\usepackage[utf8]{inputenc}
\usepackage[authoryear]{natbib}
\usepackage{graphics}
\usepackage{setspace, wrapfig}

\RequirePackage[colorlinks=true,citecolor=blue!50!black,linkcolor=blue!50!black,urlcolor=blue!50!black,pdfauthor=author]{hyperref}

\renewcommand{\L}{\mathbb L}
\newcommand{\N}{\mathbb N}
\newcommand{\Z}{\mathbb Z}
\newcommand{\R}{\mathbb R}

\newcommand{\cov}{\mathrm{Cov}}
\newcommand{\E}{\mathbb E}
\renewcommand{\P}{\mathbb P}
\newcommand{\var}{\mathrm{Var}}
\newcommand{\1}{\mathbbm 1}
\newcommand{\rmd}{\mathrm d}
\newcommand{\rme}{\mathrm e}
\newcommand{\rmi}{\mathrm i}

\newcommand{\bGamma}{\mathbf{\Gamma}}
\newcommand{\bPhi}{\mathbf{\Phi}}

\newcommand{\bSigma}{\mathbf{\Sigma}}
\newcommand{\bTheta}{\mathbf{\Theta}}
\newcommand{\bOmega}{\mathbf{\Omega}}

\newcommand{\bLambda}{\mathbf{\Lambda}}

\newcommand{\bepsilon}{\mbox{\boldmath$\varepsilon$}}
\newcommand{\bd}{\mathbf{d}}
\newcommand{\brho}{\mbox{\boldmath$\rho$}}
\newcommand{\bphi}{\mbox{\boldmath$\phi$}}
\newcommand{\bA}{\mathbf{A}}
\newcommand{\bB}{\mathbf{B}}

\newcommand{\bD}{\mathbf{D}}

\newcommand{\bG}{\mathbf{G}}

\newcommand{\bI}{\mathbf{I}}

\newcommand{\bM}{\mathbf{M}}

\newcommand{\bW}{\mathbf{W}}
\newcommand{\bX}{\mathbf{X}}

\newcommand{\bpm}{\left[\begin{matrix}}
\newcommand{\epm}{\end{matrix}\right]}
\renewcommand{\dim}{p}

\newcommand{\atan}{\mbox{$\mathrm{atan}$}}
\newcommand{\sign}{\mathrm{sign}}

\renewcommand{\leq}{\leqslant}
\renewcommand{\geq}{\geqslant}
\renewcommand{\epsilon}{\varepsilon}
\renewcommand{\hat}{\widehat}

\renewcommand{\tilde}{\widetilde}

\newcommand{\diag}[1]{\mathrm{Diag}\left(#1\right)}
\newcommand{\abs}[1]{\left\lvert#1\right\rvert}
\newcommand{\norm}[1]{\left\lVert#1\right\rVert}
\newcommand{\bff}{\mbox{${\boldsymbol{f}}$}}
\newcommand{\BigO}{\mbox{O}}

\DeclareMathOperator{\Vect}{vec}
\newcommand{\vect}[1]{\Vect\!\left( #1 \right)}

\def\bdexp{\mbox{\footnotesize $\bd$}}

\def\bGexp{\mbox{\footnotesize $\bG$}}

\def\GIGd{\bG\centerdot\mathcal I^{\bdexp}\centerdot \bG}
\def\GIGG{\bG\centerdot\mathcal I^{\bGexp}\centerdot \bG}

\makeatletter
\newcommand{\pushright}[1]{\ifmeasuring@#1\else\omit\hfill$\displaystyle#1$\fi\ignorespaces}
\newcommand{\pushleft}[1]{\ifmeasuring@#1\else\omit$\displaystyle#1$\hfill\fi\ignorespaces}
\makeatother

\makeatletter
\newcommand{\setlabel}[3]{%
  \phantomsection
  #1\def\@currentlabel{\unexpanded{#2}}\label{#3}%
}
\makeatother

\DeclareMathOperator*{\argmin}{argmin}

\theoremstyle{plain}
\newtheorem{theorem}{Theorem}
\newtheorem{proposition}[theorem]{Proposition}
\newtheorem{definition}[theorem]{Definition}
\newtheorem{corollary}[theorem]{Corollary}
\newtheorem{lemma}[theorem]{Lemma}
\theoremstyle{remark}

\graphicspath{{.}{../revision/figures/}}

\begin{document}

\mathtoolsset{showonlyrefs}

\newcolumntype{d}[1]{D{.}{.}{#1} }

\title{Local Whittle estimation with (quasi-)analytic wavelets}

\author{Sophie Achard} 
\affil{Univ. Grenoble Alpes, CNRS, Inria, Grenoble INP, LJK, France}

\author{Ir\`ene Gannaz}
\affil{Univ. Grenoble Alpes, CNRS, Grenoble INP\footnote{Institute of Engineering Univ. Grenoble Alpes}, G-SCOP, 38000 Grenoble, France}

\date{August 2023}

\maketitle

\begin{abstract}
In the general setting of long-memory multivariate time series, the long-memory characteristics are defined by two components. The long-memory parameters describe the autocorrelation of each time series. And the long-run covariance measures the coupling between time series, with general phase parameters. It is of interest to estimate the long-memory, long-run covariance and general phase parameters of time series generated by this wide class of models although they are not necessarily Gaussian nor stationary. This estimation is thus not directly possible using real wavelets decomposition or Fourier analysis. Our purpose is to define an inference approach based on a representation using quasi-analytic wavelets. We first show that the covariance of the wavelet coefficients provides an adequate estimator of the covariance structure including the phase term. Consistent estimators based on a local Whittle approximation are then proposed. Simulations highlight a satisfactory behavior of the estimation on finite samples on linear time series and on multivariate fractional Brownian motions. An application on a real neuroscience dataset is presented, where long-memory and brain connectivity are inferred.
\end{abstract}


\vspace{\baselineskip}

\setlength{\parindent}{0pt}
\setlength{\parskip}{\baselineskip}
\onehalfspacing

\section{Introduction}
\label{sec:model}

Multivariate processes are often observed nowadays thanks to the recordings of multiple sensors simultaneously. Many examples can be cited such as hydrology \citep{WhitcherJensen00}, finance \citep{Gencay} or neuroscience \citep{AchardGannaz}. When in addition the time series have the property of long-memory, the definition of the model is complicated and several definitions can be proposed. Some approaches proposed a simple definition, where the covariance matrix is real \citep{Lobato99, Shimotsu07}. However, this simple model is not able to address any multivariate models. For example, in \citep{Lobato97} two models were introduced, FIVARMA and VARFIMA, from this approach. Long-memory models with a complex covariance matrix give a solution to overcome this problem \citep{KechagiasPipiras20,Baek2020}. Following these modelings, the long-memory model studied in this paper admits a complex long-run covariance matrix, where a phase-term is added to the covariance structure.

\parbox{\textwidth}{Let ${\bX=\{\bX(t),\,t\in\Z\}}$ denote a multivariate long-memory dependence process $\bX(t)=\bpm X_1(t)&\dots&X_{\dim}(t)\epm^T$, $t\in\Z$, $p\in\N$, $p\geq 1$, with long memory parameters $\bd=(d_1,d_2,\ldots,d_p)$, $\bd\in(-0.5,+\infty)^p$. The exponent $T$ is the transpose operator. We will denote by $\L$ the backward lag operator, ${(1-\L)\bX(t)=\bX(t)-\bX(t-1)}$. The $k$-th difference operator, ${(1-\L)^k}$, $k\in\N$, is defined by $k$ recursive applications of ${(1-\L)}$. For $\bD=\lfloor \bd+1/2\rfloor$, we assume that the multivariate process $\diag{{(1-\mathbb{L})^{D_{\ell}}},\ell=1,\dots,p}\bX$ is covariance stationary with a spectral density matrix given by}
\begin{enumerate}[label=(M-\arabic*)]
\item\label{ass:zero-frequency} 
$
\bff^{\bD}(\lambda) = \Bigl(\diag{{\lambda}^{-d^*_1},\dots,{\lambda}^{-d^*_{\dim}}}\,\bTheta\,\diag{{\lambda}^{-d^*_1},\dots,{\lambda}^{-d^*_{\dim}}}\Bigr)\circ \bff^S(\lambda), \quad\text{for all }\lambda >0,
$
\end{enumerate}
where $\circ$ denotes the Hadamard product,
and $d_m^*=d_m-D_m\in(-0.5,0.5)$ for all $m$. The process ${(1-\L)^{D_{m}}X_m}$ is said to have long-memory if $d_m^*\in(0,0.5)$, and to be anti-persistent if $d_m^*\in(-0.5,0)$ (see for instance \cite{Lobato99,Shimotsu07}). For simplicity of notation, we will use the term long-memory parameters $\bd$ throughout the paper.  

The function $\bff(\cdot)$ is defined by
\[ 
\bff(\lambda) = (\bLambda(\lambda)\,\bTheta\,\bLambda(\lambda))\circ \bff^S(\lambda), \quad\text{for all }\lambda >0,
 \]
 with $\bLambda(\lambda)=\diag{\lambda^{-d_1},\dots,\lambda^{-d_{\dim}}}.$ Under condition \ref{ass:zero-frequency}, the function $\bff(\cdot)$ is called the generalized spectral density of the multivariate process ${\bX=}\{\bX(t),\,t\in\Z\}$.

The function $\bff^S(\cdot)$ represents the short-range memory of $\bff(\cdot)$. In order to get identifiability, it is necessary to assume $\bff^S(0)=\mathbf{1}$. The following assumption is also needed to control the regularity. 
\begin{enumerate}[label=(M-\arabic*),start=2]
\item\label{ass:beta} There exists $C_f>0$ and
$\beta>0$ such that
$
  \sup_{0<\lambda<\pi}\;\sup_{\ell, m=1,\dots,N}\frac{\abs{f_{\ell,m}^S(\lambda)-1}}{\lambda^\beta}\leq
  C_f \;.
$
\end{enumerate}

In particular, our definition agrees with the one given in \cite{KechagiasPipiras} if $D_m=0$ for all {$m=1,\dots,p$}. Definition \ref{ass:zero-frequency} includes both stationary and non-stationary time series. It has the advantage of including multivariate fractional Brownian motion \citep{mFBM13}.  

\parbox{\textwidth}{The major interest of this model is the introduction of the matrix $\bTheta$. This provides a generalization of multivariate long-memory models used in \cite{Lobato97,Shimotsu07,AchardGannaz}. Indeed, the matrix  $\bTheta$ can be written as,
\[
\Theta_{\ell,m}=\Omega_{\ell,m}\rme^{\rmi\phi_{\ell,m}}\;,\;\ell,m=1,\dots,p,
\]
with $\bOmega=(\Omega_{\ell,m})_{\ell,m=1,\dots,\dim}$ a real symmetric non-negative semi-definite matrix and $\bPhi=(\phi_{\ell,m})_{\ell,m=1,\dots,\dim}$ an anti-symmetric matrix. 
Let the bar above denote the conjugate operator. The matrix $\bTheta$ satisfies $\bTheta^T=\overline{\bTheta}$ since $\bff^T(\cdot)=\overline{\bff(\cdot)}$.  
We will use $\norm{\bOmega}$ to denote the infinity norm, that is, $\norm{\bOmega}=\max_{\ell,m=1,\dots,\dim} {\Omega_{\ell,m}}$.
In \cite{Lobato97,Shimotsu07,AchardGannaz}, the phase term was defined by $\phi_{\ell,m}=\pi(d_\ell-d_m)/2$.}

In a univariate setting, the main parameter of interest is the long-memory parameter or equivalently the Hurst parameter. In this particular case, three main families of Fourier-based estimation have already been proposed: the average periodogram estimation \citep{Robinson94a}, the log periodogram regression \citep{Geweke, Robinson95a} and semiparametric estimation based on Whittle approximation \citep{kunsch:1987,Robinson95b}. Estimation with a wavelet representation of time series was proposed in \cite{AbryVeitch98} with a log-scalogram approach similar to log-periodogram estimation and in \cite{Moulines08Whittle} with a wavelet-based local Whittle estimation.

In a multivariate setting, estimation procedures have also been proposed using either Fourier or wavelets. For a general phase term, \cite{SelaHurvich2012} proposed an estimation based on the average periodogram and \cite{Robinson08} and \cite{Baek2020} developed a Fourier-based local Whittle estimation. For a fixed phase term, $\phi_{\ell,m}=\pi(d_\ell-d_m)/2$, estimation of both the covariance structure and the long-memory was proposed by \cite{Lobato99}, \cite{Shimotsu07} and \cite{nielsen2011fully}, with a Fourier-based local Whittle estimation, and by \cite{AchardGannaz} with a similar procedure based on a real wavelet representation.

The objective of this work is to propose an estimation procedure in the general framework described above, with a general phase, based on a wavelet representation of the processes rather than a Fourier representation. Our model includes among other the co-integration case \citep{Robinson08,Nielsen11,Baek2020}. Observe that the Fourier-based local Whittle procedure proposed in \cite{Baek2020} is very closed to the one developed here.

Introducing wavelets is motivated by their flexibility for real data applications. In particular, it allows to consider non-stationary processes thanks to an implicit differentiation. The introduction of a general phase term challenges the choice of the wavelet filters. Due to condition \ref{ass:zero-frequency}, we need to consider complex filters for identifying the imaginary part of $\bTheta$. Indeed, as illustrated in \cite{spie_analytic}, with real wavelet filters it is not possible to recover both the real and the imaginary part of the matrix $\bTheta$. Complex wavelet filters, with quasi-analytic properties, are described in Section~\ref{sec:filters}. The main properties of the filters are displayed and an approximation of the covariance of the wavelet coefficients is derived in Section \ref{sec:cov}. Section~\ref{sec:whittle} recalls the definition of the wavelet local Whittle estimators. We prove their consistency and their convergence rate, as well as their asymptotic distribution. Section~\ref{sec:simu} reports some simulation results, on ARFIMA linear models and on multivariate fractional Brownian motions. Section~\ref{sec:real} provides an empirical application on neuroscience data. The detailed proofs are provided in Appendix.

\section{Transform of the multivariate process}

\label{sec:filters}

We first define the filters used to transform the multivariate time series ${\bX=}\{\bX(t),\,t\in\Z\}$.

Let $(h^{(L)}(\cdot),h^{(H)}(\cdot))$ and  $(g^{(L)}(\cdot),g^{(H)}(\cdot))$ denote two pairs of respectively low-pass and high-pass filters. 
Let $(\varphi_h(\cdot), \psi_h(\cdot))$ be respectively the father and mother
wavelets associated to $(h^{(L)}(\cdot),h^{(H)}(\cdot))$. They can be defined through their Fourier transforms as
\begin{equation}\label{eqn:phipsi_h}
\hat\varphi_h(\lambda)=2^{-1/2}\prod_{j=1}^\infty \left[2^{-1/2}\hat h^{(L)}(2^{-j}\lambda)\right] \text{  ~~and~~  }
\hat\psi_h(\lambda)=2^{-1}\hat h^{(H)}(\lambda/2)\, \hat\varphi_h(\lambda/2).
\end{equation}
Let us define similarly $(\varphi_g(\cdot), \psi_g(\cdot))$ the father and the mother wavelets associated with
the wavelet filters $g^{(L)}(\cdot)$ and $g^{(H)}(\cdot)$. Their Fourier transforms are
\begin{equation}\label{eqn:phipsi_g}
\hat\varphi_g(\lambda)=2^{-1/2}\prod_{j=1}^\infty \left[2^{-1/2}\hat g^{(L)}(2^{-j}\lambda)\right] \text{  ~~and~~  }
\hat\psi_g(\lambda)=2^{-1}\hat g^{(H)}(\lambda/2)\, \hat\varphi_g(\lambda/2).
\end{equation}
The complex father and mother wavelets $(\varphi(\cdot),\psi(\cdot))$ are then defined by 
\begin{equation}\label{eqn:phipsi_total}
\hat\varphi(\lambda)=\hat\varphi_h(\lambda)+\rmi\,\hat\varphi_g(\lambda) \text{  ~~and~~  }
\hat\psi(\lambda)=\hat\psi_h(\lambda)+\rmi\,\hat\psi_g(\lambda).
\end{equation}

{Wavelet $\psi(\cdot)$ is said to be analytic if its Fourier transform is
only supported on the positive frequency semi-axis. In particular, it is sufficient to show that the pair $(\psi_g(\cdot),\psi_h(\cdot))$ is a Hilbert pair, that is,
$
\hat\psi_g(\lambda)=-\rmi \; \sign(\lambda)\hat\psi_h(\lambda)
$, for all $\lambda\in\R$, where $\sign(\lambda)$ denotes the sign function taking values $-1,0$ and 1 for
$\lambda<0$, $\lambda=0$ and $\lambda>0$, respectively.}

{From Paley-Wiener Theorem, analytic filters with finite support do not exist. Selesnick's common factor filters propose compact wavelet filters with a relaxation of the strict analytic condition.
}

\subsection{Selesnick's common factor filters}

 We choose here to consider the quasi-analytic filters introduced by \cite{Thiran,Selesnick-thiran}. The \emph{common-factor} wavelets, defined by \cite{Selesnick-thiran}, have a compact support and are quasi-analytic. They are parameterized by a degree $L$ quantifying the approximation of the analytic property of the derived complex wavelet.  We refer the reader to \cite{Selesnick2001, Selesnick-thiran, common_factor} for a fuller  description of the construction of the wavelets and of their properties.

Let $\widehat{d}_L(\lambda)$, $\lambda\in\R$, be defined by
\begin{equation}
\widehat{d}_L(\lambda)=\rme^{\rmi\,\lambda(-L/2+1/4)}\left[\cos(\lambda/4)^{2L+1}+\rmi\,(-1)^{L+1} \sin(\lambda/4)^{2L+1}\right].
\end{equation}
Next, filters $(\hat h^{(L)}$, $\hat h^{(H)})$, and  $(\hat g^{(L)}, \hat g^{(H)})$ are defined by
\begin{gather}
\label{eqn:hat_h}\widehat{h}^{(L)}(\lambda)=2^{-M+1/2}
\left(1+e^{-i\lambda}\right)^M\hat q_{L,M}(\lambda)\,\widehat{d}_L(\lambda)
\text{   ~~and~~   }
\widehat{h}^{(H)}(\lambda)=\overline{\widehat{h}^{(L)}(\lambda+\pi)}e^{-i\lambda}\,,\\
\label{eqn:hat_g}
\widehat{g}^{(L)}(\lambda)=2^{-M+1/2}
\left(1+e^{-i\lambda}\right)^M\hat q_{L,M}(\lambda)\,\overline{\widehat{d}_L(\lambda)} e^{-i\lambda L}
\text{   ~~and~~   }
\widehat{g}^{(H)}(\lambda)=\overline{\widehat{g}^{(L)}(\lambda+\pi)}e^{-i\lambda},
\end{gather}
with $\hat q_{L,M}(\lambda)$ a real polynomial of $(e^{-i\lambda})$ such that $\hat q_{L,M}(0)=1$. Observe that the normalization of the filters is different from \cite{common_factor}.

Common-factor wavelets are introduced with $\hat q_{L,M}{(.)}$ such that the wavelet decomposition satisfies the perfect reconstruction condition. This condition is classically used for deriving wavelet bases
$2^{1/2}\psi_{g\,j,k}{(.)}=2^{-1/2}2^{j/2}\psi_g(2^j\cdot-k)$ and
$2^{1/2}\psi_{h\,j,k}{(.)}=2^{-1/2}2^{j/2}\psi_h(2^j\cdot-k)$, $j,k\in\Z$, which are 
orthonormal bases of $L^2(\R)$. In that case, $\hat q_{L,M}$ is defined as a solution of 
\begin{equation}\label{eqn:PR}
 \abs{\hat q_{L,M}(\lambda)}^2s(\lambda)+\abs{\hat q_{L,M}(\lambda+\pi)}^2s(\lambda+\pi)=1 \;,
\end{equation}
where $s(\lambda)=\frac{2^{4L-1}}{(2L+1)^2}2^{-M}(1+\cos(\lambda))^M \abs{\hat d_{L}(\lambda)}^2$. The existence of $\hat q_{L,M}(\cdot)$ is proved in \cite{common_factor}. However, to the best of our knowledge, under perfect reconstruction, no explicit expression of $\hat q_{L,M}$ is easy to obtain. Since perfect reconstruction is not necessary for deriving estimation procedures, we can assume that $\hat q_{L,M}{(.)}$ is a constant equal to 1. 

\begin{definition}[Common-Factor Wavelets (CFW)]
Let $M$, $L$ be strictly positive integers.
 Let $(\varphi{(.)}, \psi{(.)})$ be a family of Common-Factor wavelets defined by equations \eqref{eqn:phipsi_h}, \eqref{eqn:phipsi_g}, \eqref{eqn:phipsi_total}, and \eqref{eqn:hat_h}, \eqref{eqn:hat_g}. If the filter $\hat q_{L,M}{(.)}$ satisfies perfect reconstruction condition \eqref{eqn:PR}, the pair $(\varphi{(.)}, \psi{(.)})$ will be denoted by CFW-PR(M,L). If $\hat q_{L,M}{(.)}$ is a constant polynomial equal to 1, $(\varphi{(.)}, \psi{(.)})$ will be denoted by CFW-C(M,L) filters.
\end{definition}

{The two main characteristics are the compact support and the quasi-analiticity. The compact support property for CFW-C(M,L) filters is given below.}

\begin{proposition}\label{prop:support}
{Let $M$, $L$ be strictly positive integers.  The functions $\varphi(.)$, and $\psi(.)$ for CFW-C(M,L) have supports of respective length $M+2L+1$ and $M+L+1/2$.}
\end{proposition}
{The proof is given in Appendix.}

{Concerning the functions $\varphi(.)$, and $\psi(.)$ for CFW-PR(M,L), in practice, they have supports of respective length $2M+3L$ and $2M+2L+1/2$. Yet, there is no theoretical proof that these lengths are indeed achieved for all $(L,M)$. See Section 4 of \cite{common_factor}.}

Let us now recall the main result concerning the analytic approximation established in \cite{common_factor}.

\begin{theorem}[\cite{common_factor}]\label{th:main-result-analyticity}
For all $\lambda\in\R$, for all $\hat q_{L,M}{(.)}$ real polynomial of $(e^{-i\lambda})$,
\begin{gather}\label{eq:uL}
\hat\psi(\lambda)=\hat \psi_h(\lambda)+\rmi\,\hat \psi_g(\lambda)=\left(1-\rme^{\rmi \eta_L(\lambda)}\right)\hat \psi_h(\lambda)\;,\\
\label{eq:alpha-def}
\text{with~~  }  a_L(\lambda)= 2(-1)^L\,\atan\left(\tan^{2L+1}(\lambda/4)\right)
 \text{~~and~~} \eta_L(\lambda)=-a_L(\lambda/2+\pi)+\sum_{j=1}^{\infty} a_L(2^{-j-1}\lambda).
\end{gather} 
  Additionally, for all $\lambda\in\R$,
  \[
  \left|\hat\psi_h(\lambda)+\rmi\,\hat\psi_g(\lambda) -
    2\1_{\R_+}(\lambda)\,\hat\psi_{h}(\lambda)\right|= U_L(\lambda)
  \left|\hat\psi_h(\lambda)\right|\;,
  \]
  where $U_L(\cdot)$ is a $\R\to[0,2]$ function satisfying, for
  all $\lambda\in\R$,
    \begin{equation}
    \label{eq:final-analycity-result}
U_L(\lambda) \leq
    2\sqrt{2}
\left(\log_2\left(\frac{\max(4\pi,\abs{\lambda})}{2\pi}\right)+2\right)\,\left(1-\frac{\delta(\lambda,4\pi\Z)}{\max(4\pi,\abs{\lambda})}\right)^{2L+1}
 \;.
\end{equation}
and, for all $\lambda\in\R$ and $A\subset\R$, $\delta(\lambda,A)$ denotes the distance of $\lambda$ to $A$ defined by
$\delta(\lambda,A)=\inf\left\{\abs{\lambda-x},~x\in A\right\}.$
\end{theorem}
In equation~\eqref{eq:alpha-def}, we adopt the convention that $\atan(\pm\infty)=\pm\pi/2$ so
that $\alpha_L(\cdot)$ is well defined on $\R$. 

Theorem \ref{th:main-result-analyticity} quantifies the quality of the analytic approximation. Observe that the function $U_L(\cdot)$ only depends on the parameter $L$. The higher $L$, the better the analytic approximation. {However, the higher $L$, the larger the wavelets support.}

\section{Moments approximations of the wavelet coefficients}
\label{sec:cov}

Let $\{\bW_{j,k},\,j\geq 0,\,k\in\Z\}$ denote the wavelet coefficients of the process $\bX$ associated to the wavelet pair $(\varphi{(.)},\psi{(.)})$. At a given resolution $j\geq 0$, for $k\in\Z$, we define the dilated and translated functions $\psi_{j,k}(\cdot)=2^{-j/2}\psi(2^{-j}\cdot -k)$. The wavelet coefficients of the process $\bX$ are defined by \[
\bW_{j,k}=\int_\R \tilde{\bX}(t)\psi_{j,k}(t)dt\quad j\geq 0, k\in\Z,
\]
where $\tilde{\bX}(t)=\sum_{k\in\Z}\bX(k)\varphi(t-k).$
Given any $j\geq 0$ and any $k\in\Z$, $\bW_{j,k}$ is a $p$-dimensional vector $\bW_{j,k}=\begin{pmatrix}
W_{j,k}(1) & W_{j,k}(2) & \dots & W_{j,k}(p) \end{pmatrix}^T$ where $W_{j,k}(a)= \int_\R \tilde{X_a}(t)\psi_{j,k}(t)\rmd t$, $a=1,\dots,p$.  Throughout the paper, we adopt the convention that large values of the scale index $j$ correspond to coarse scales (low frequencies).

We will consider the behavior of $\cov(\bW_{j,k})$, defined as follows
\begin{equation}
\cov(\bW_{j,k})=\E\left[\bW_{j,k}\overline{\bW_{j,k}}^T\right]=\int_{-\pi}^\pi 
\bff(\lambda)\abs{\hat\tau_j(\lambda)}^2\;\rmd \lambda\,,
\end{equation}
with $\hat\tau_j(\lambda)=\int_{-\infty}^\infty \sum_{\ell\in\Z}\varphi(t+\ell)\,\rme^{-\rmi\,\lambda\,\ell}2^{-j/2} \psi(2^{-j}t)\rmd t.$

In practice, a finite number of observations of the process $\bX$ are available, $\bX(1), \bX(2), \dots \bX(N)$. As the wavelets have a compact support, only a finite number of coefficients are non-zero at each scale $j$. More precisely, for every $j\geq 0$, let $n_j$ denote the number of coefficients $\bW_{j,k}$ evaluated from the observations. Note that only the coefficients evaluated without boundary effects are taken into account (see the definition of $n_j$ in Lemma \ref{lem:nj}). For all $k< 0$ or $k> n_j$, the coefficients $\bW_{j,k}$ are set to zero. In the following, we will assume that $M$ is fixed and finite, and that $L$ may go to infinity. Hence, the length of the wavelets support is equivalent to $L$ when $N$ goes to infinity. If additionally $2^j N^{-1}L\to 0$, then $n_j$ is equivalent to $2^{-j}N$. In that case, the behavior of $n_j$ is similar to the framework of \cite{Moulines08Whittle} and \cite{AchardGannaz}.

\begin{lemma} \label{lem:nj}
{Let $(\varphi{(.)},\psi{(.)})$ be a CFW-C(M,L) wavelet pair, with $M, L \geq 1$.
Let $\{\hat \bW_{j,k},\,j\geq 0,\,k\in\Z\}$ denote the wavelet coefficients evaluated from $\bX(1), \bX(2), \dots \bX(N)$ by $\hat \bW_{j,k}=\int_\R \hat{\bX}(t)\psi_{j,k}(t)\,\rmd t$, $j\geq 0$, $k\in\Z,$
where $\hat{\bX}(t)=\sum_{k=1}^N\bX(k)\varphi(t-k).$
Then, at each scale $j$, the number of coefficients $n_j$ such that $\hat\bW_{j,k}=\bW_{j,k}$ is \[n_j=\max\{0,\lfloor 2^{-j}(N-2L-M-1)-L-M-1/2\rfloor\}.\]
Suppose that $N^{-1}L\to 0$ when $N$ goes to infinity. Then, for all $j$ such that $2^jN^{-1}L\to 0$ when $N$ goes to infinity, $n_j\,2^{j}N^{-1} \to 1$ when $N$ goes to infinity.}
\end{lemma}

\subsection{Motivations}

\label{sec:example} 
 
In this section, results obtained in \cite{spie_analytic} are summarized.
We begin with a bivariate ARFI{\-}MA(0,$\bd$,0) process defined as 
 \begin{equation}\label{eqn:example}
 X_\ell(k) =  (1-\mathbb{L})^{-d_\ell} u_\ell(k),\quad\ell=1,2,\;k\in\Z\,,
 \end{equation}
where $\mathbb{L}$ is a lag operator and $\begin{pmatrix}
u_1(k) \\ u_2(k)
\end{pmatrix}$ {\it i.i.d.} with distribution $\mathcal N\left(\begin{pmatrix}
0\\ 0
\end{pmatrix},\bOmega\right)$, where $\bOmega=\begin{pmatrix}
1 & 0.8\\ 0.8 & 1
\end{pmatrix}.$
The spectral density of $(X_1,X_2)$ satisfies \ref{ass:zero-frequency} with $\Theta_{\ell,m}=\Omega_{\ell,m}\rme^{\rmi\,\phi_{\ell,m}}$, $\phi_{\ell,m}=\pi(d_1-d_2)/2$.
Let $\bd$ be equal to $(0.2, 1.2)$. The phase is equal to $\pi/2$ and, hence, $\Theta_{\ell,m}$ is imaginary.  Let us now illustrate the impossibility using real wavelets decomposition to infer $\Theta_{\ell,m}$.

We simulate $\bX(1), \dots, \bX(2^J)$, with $J=12$. For each scale $j\geq 0$, we evaluate the sample wavelet covariances as $\hat\bSigma(j)=\frac{1}{n_j}\sum_{k=0}^{n_j-1} W_{j,k}(1)\overline{W_{j,k}(2)} - \Bigl(\frac{1}{n_j}\sum_{k=0}^{n_j-1} W_{j,k}(1)\Bigr)\Bigl(\frac{1}{n_j}\sum_{k=0}^{n_j-1}\overline{W_{j,k}(2)}\Bigr)$, and the wavelet sample correlations as $\hat\Sigma_{1,2}(j)/\sqrt{\hat\Sigma_{1,1}(j)\hat\Sigma_{2,2}(j)}$. 

\begin{figure}[!ht]
\includegraphics[width=\textwidth]{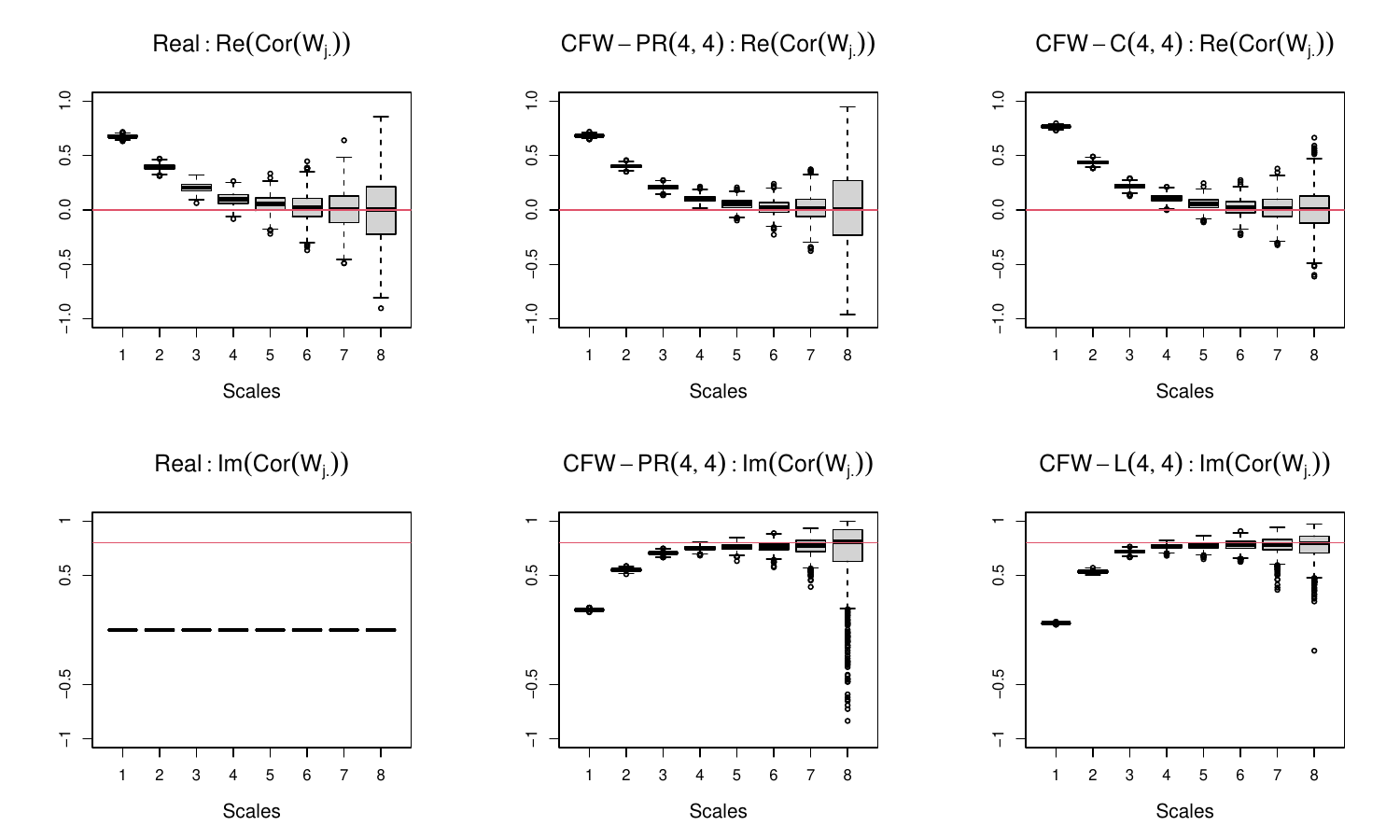}
\caption{Boxplots of sample correlations between wavelet coefficients at different scales for the bivariate process defined in \eqref{eqn:example}. First row gives the real part of the correlations and second row gives the imaginary part. Each column corresponds to different wavelet filters, from left to right: Daubechies' real wavelets with M=4, CFW-PR(4,4) and CFW-C(4,4). Horizontal red lines correspond to the real value, that is, $\Omega_{1,2}\cos(\phi_{1,2})/\sqrt{\Omega_{1,1}\Omega_{2,2}}$ for the real part and $\Omega_{1,2}\sin(\phi_{1,2})/\sqrt{\Omega_{1,1}\Omega_{2,2}}$ for the imaginary part.}
\label{fig:example}

\end{figure}

Figure \ref{fig:example} shows the behavior of sample wavelet correlations with respect to scale $j$ over 100 realizations of $(\bX(1), \dots, \bX(2^J))$. First observe that for real wavelets (left column), the wavelet sample covariance $\hat\Sigma_{1,2}(j)$ tends to 0 when the scale $j$ increases. This confirms the impossibility to identify $\Theta_{1,2}$. In addition, the plots displayed in the middle and right columns illustrate that the imaginary part of the sample wavelet coefficient correlations does not vanish for CFW-PR(M,L) and CFW-C(M,L) filters. The average sample correlation seems to converge to $\Omega_{1,2}\rme^{\rmi\phi_{1,2}}/\sqrt{\Omega_{1,1}\Omega_{2,2}}$ as the frequency decreases.

 \subsection{Theoretical results}

 We will now develop the theory of the behavior of $\cov(\bW_{j,k})$. This result consists in the extension of Proposition~3 of \cite{AchardGannaz} to quasi-analytic wavelets. The results are obtained hereafter only for CFW-C(M,L) filters. Indeed, the results are more difficult to obtain for CFW-PR(M,L) filters because no explicit expression of $\hat q_{L,M}$ satisfying \eqref{eqn:PR} is available. 
 
Our basic assumption on the regularity of the spectral density is the following.  
\begin{enumerate}[label=(C-\alph*),start=1]
\item\label{ass:parameters} 
$-M+\beta+1< 2\,d_\ell< M\quad\text{for all}\quad \ell=1,\dots,\dim$, $M\geq 2$. 
\end{enumerate}
Parameter $M$ is the number of vanishing moments and it also corresponds to the regularity of CFW-C(M,L) filters. Parameters $(d_\ell)_{\ell=1,\dots\dim}$ and $\beta$ characterize the dependence in the spectral domain \ref{ass:zero-frequency}-\ref{ass:beta}.

Let us first prove the following approximation using the regularity of the filters.
\begin{proposition}
\label{prop:covY2}
Let $\bX$ be a $\dim$-multivariate long range dependent process with long memory parameters $d_1,\dots,d_{\dim}$ with generalized spectral density $\bff(\cdot)$ satisfying~\ref{ass:zero-frequency} with short-range behavior \ref{ass:beta}. Consider $\{W_{j,k}(\ell),\, (j,k)\in\Z,\, \ell=1,\dots,p\}$ the wavelet coefficients obtained with CFW-C(M,L) filters.
Suppose that \ref{ass:parameters} hold.
Then we have, for all $j\geq0$, $k\in\Z$,
\[
\abs{\cov(W_{j,k}(\ell),W_{j,k}(m)) -
2^{j(d_\ell+d_m)}\Omega_{\ell,m}\int_{-\infty}^\infty 
\rme^{\sign(\lambda)\,\phi_{\ell,m}}\,\abs{\lambda}^{-d_\ell-d_m}
\abs{\hat\psi(\lambda)}^2\;\rmd \lambda}
\leq C_1\,2^{j\,(d_\ell+d_m-\beta)}\;,
\]
where 
 $C_1$ is a constant only depending on $M, L$ and $C_f,\beta, \norm{\bOmega}, \{d_\ell,\ell=1,\dots,\dim\}$.
\end{proposition}
The proof is given in Appendix.

The result follows from the fact that CFW-C(M,L) satisfy the assumptions (W1)--(W4) described in \cite{Moulines08Whittle} and \cite{AchardGannaz} (see Appendix). Note that it does not depend on the quasi-analytic property.

The use of the Proposition \ref{prop:covY2} in inference needs the evaluation of the integral depending of $\lvert\hat\psi(\lambda)\rvert^2$. With real wavelets, the approximation is given in Proposition~3 of \cite{AchardGannaz}. Since $\lvert\hat\psi(\lambda)\rvert^2$ is a real and symmetric function, the imaginary part of the integral is null. Consequently, a cosine term with the phase appears in the approximation of the covariance. That is, we would obtain in this framework an approximation of the form \begin{multline}\label{eqn:real}
\abs{{2^{-j\,(d_\ell+d_m)}}\cov(W_{j,k}(\ell),W_{j,k}(m)) -
2\,\Omega_{\ell,m}\cos(\phi_{\ell,m})\,\int_{0}^\infty \abs{\lambda}^{-\delta}\;\abs{\hat \psi(\lambda)}^2\;\rmd \lambda }\\
\leq C\,\norm{\bOmega}\,{2^{-j\,\beta}}.
\end{multline}

It is straightforward to check that parameters $\{\Omega_{\ell,m}, \phi_{\ell,m}\}$ are not identifiable. Estimation can be derived in the case of a parametric phase, typically $\phi_{\ell,m}=\frac{\pi}{2}(d_\ell-d_m)$ (see \cite{AchardGannaz}).

In the case of quasi-analytic wavelets, the imaginary part no longer vanishes. The control of quasi-analyticity, given by Theorem~\ref{th:main-result-analyticity}, leads to the following result.

\begin{proposition}
\label{prop:covY3}
Let $\bX$ be a $\dim$-multivariate long range dependent process with long memory parameters $d_1,\dots,d_{\dim}$ with generalized spectral density {$\bff(\cdot)$} satisfying~\ref{ass:zero-frequency}--\ref{ass:beta}. 

Consider $\bigl\{W_{j,k}(\ell),\, (j,k)\in\Z,\, \ell=1,\dots,p\bigr\}$ the wavelet coefficients obtained with CFW-C(M,L) filters.
Suppose that \ref{ass:parameters} hold and that $L$ goes to infinity, with  $L2^{-2j}\to 0$ when $j$ goes to infinity.

Then,  for all $(\ell,m)\in\{1,\dots,p\}^2$,
\begin{multline}
\label{eqn::approx_covbis}
\abs{2^{-j(d_\ell+d_m)}\,\cov(W_{j,k}(\ell),W_{j,k}(m)) -
4\,\Theta_{\ell,m}\,\int_{0}^\infty \abs{\lambda}^{-d_\ell-d_m}\;\abs{\hat \psi_h(\lambda)}^2\;\rmd \lambda}\\
\leq C_2\,\bigl(2^{-j\beta}+L2^{-2j}+L^{-M-1}\bigr)\;,
\end{multline}
where $C_2$ is a constant only depending on $M$ and $C_f,\beta, \norm{\bOmega}, \{d_\ell,\ell=1,\dots,\dim\}$.
\end{proposition}
The proof is given in Appendix.

Convergence~\eqref{eqn::approx_covbis} can be written as follows:  when $2^{-j\beta}+L2^{-2j}+L^{-1}\to0$, for all $(\ell,m)\in\{1,\dots,p\}^2$,
\begin{gather}
\label{eqn:approx}
{\lim_{j\to\infty}2^{-j\,(d_\ell+d_m)}\cov(W_{j,k}(\ell),W_{j,k}(m)) =
G_{\ell,m},} \\
~\text{with}~ G_{\ell,m}  = \Theta_{\ell,m}\,K(d_\ell+d_m) \text{~~and~~}
K(\delta)   = 4 \int_{0}^\infty \abs{\lambda}^{-\delta}\;\abs{\hat \psi_h(\lambda)}^2\;\rmd \lambda\,.
\end{gather}
Common-factor wavelets, as stated by Proposition~\ref{prop:covY3}, have the ability of recovering simultaneously the magnitude and the phase. Observe that with real wavelets the upper bound in \eqref{eqn:real} is $2^{-j\beta}$, up to a multiplicative constant. With complex wavelets, the rate depends on $L$, and the parameter will need to be calibrated accordingly.

The specificity of CFW-C(M,L) filters is that the quality of the analytic approximation is based only on parameter $L$, as written in Proposition~\ref{prop:covY3}. Nevertheless, if we want to have an approximation with the same quality as that obtained with real wavelets, the choice of $L$ is more constrained. This trade-off is due to the fact that the greater $L$, the better analyticity approximation, but the larger the length of the wavelets support. In practice, due to numerical instability, choosing high values (\emph{i.e.} $\geq 8$) is not manageable. As shown by the simulations in Section~\ref{sec:simu}, however, the results are of good quality even with a smaller value of $L$.

\subsection{Quality of approximation}

\label{sec:approx}

To empirically assess the accuracy of the approximation, let us compare the empirical covariances of the example of Section \ref{sec:example} to the approximation of Proposition~\ref{prop:covY2}. Figure~\ref{fig:boxplots_example} displays the sample covariance of the wavelet coefficients, respectively with real Daubechies filters with $M=4$, CFW-PR(4,4) and CFW-C(4,4) filters. As for Figure~\ref{fig:example}, $N=2^{12}$ observations were considered. Observe that the covariance term is complex, and only the magnitude is represented in Figure~\ref{fig:boxplots_example}.

Figure~\ref{fig:boxplots_example} shows the difference between our theoretical findings given in Proposition~\ref{prop:covY2} and the simulations for both CFW-PR(M,L) and CFW-C(M,L). To better evaluate the quality of the approximation with CFW-C(M,L) filters, the same figure without the first scale is provided in Figure~\ref{fig:boxplots_example_zoom}. It shows that indeed the approximation improves when the scale $j$ increases. Nevertheless, the difference between the results obtained with the simulations at first scales (corresponding to the highest frequencies) and the approximation given in Proposition~\ref{prop:covY2} is higher with CFW-C(4,4) filters in comparison with Daubechies and CFW-PR(4,4) filters. Therefore, the lowest scale used in estimation may be higher with CFW-C(M,L) filters. This choice may reduce the bias but increase the variance.

\begin{figure}[p]
\centering
\includegraphics[width=0.9\textwidth]{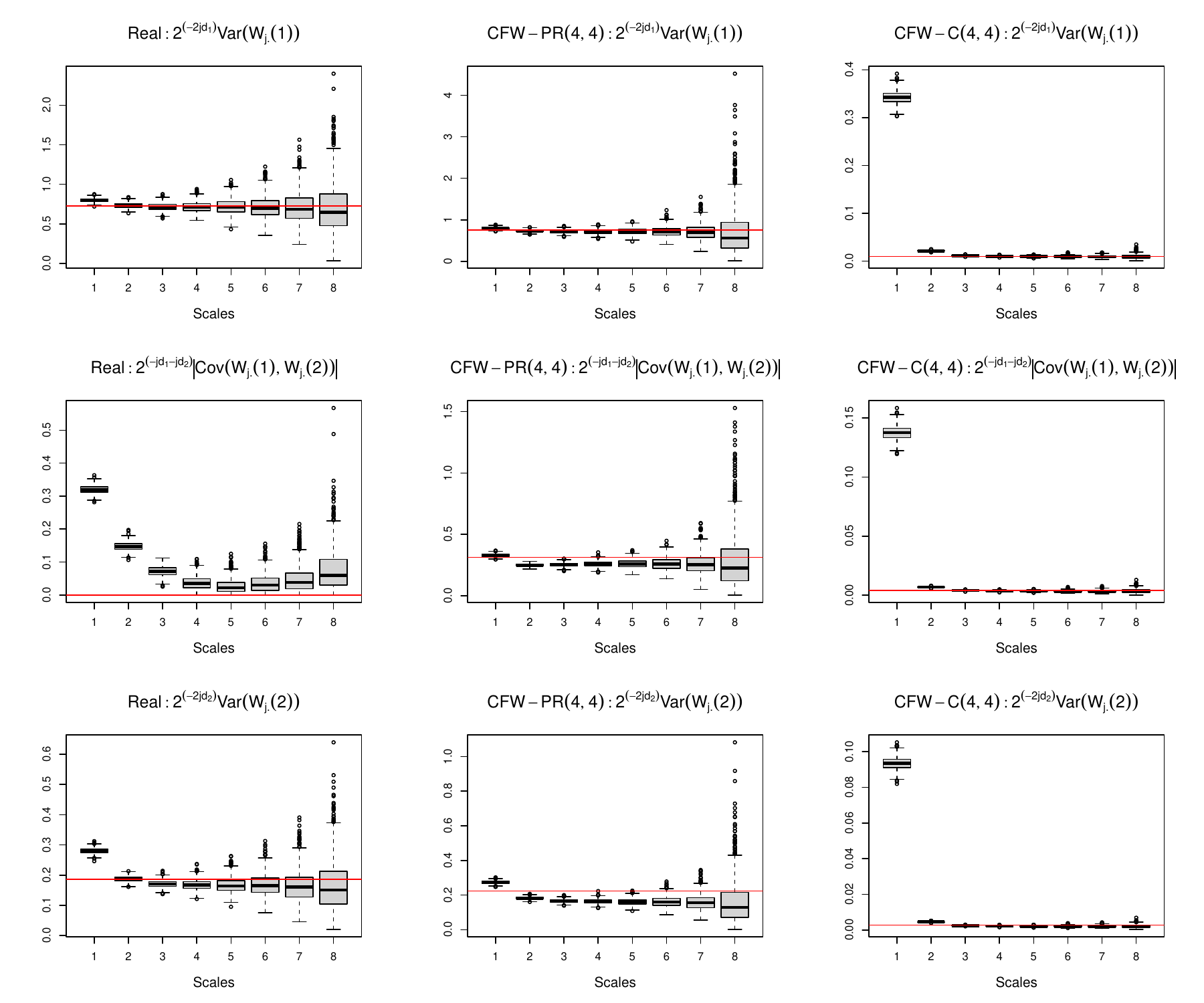}
\caption{Boxplots of normalized sample covariances between wavelet coefficients at different scales for the bivariate process defined in \eqref{eqn:example}. {Let $\mathbf{V}_{j}=\begin{pmatrix}
2^{-2\,j\,d_1}\var{\{W_{j,k}(1),\,k\in\Z\}} & 2^{-j(d_1+d_2)}\cov{\{(W_{j,k}(1),W_{j,k}(2)),\,\,k\in\Z\}} \\
2^{-j(d_1+d_2)}\cov{\{(W_{j,k}(1),W_{j,k}(2)),\,\,k\in\Z\}} & 2^{-2\,j\,d_2}\var{\{W_{j,k}(2),\,k\in\Z\}}
\end{pmatrix}$. The first row gives the sample version of $V_{j\,11}$, the second row gives the sample version of $\abs{V_{j\,12}}$ and the third row gives the sample version of $V_{j\,22}$.} Each column corresponds to a different wavelet filters, form left to right: Daubechies' real wavelets with M=4, CFW-PR(4,4) and CFW-C(4,4).  Horizontal red lines correspond to the approximation given by Proposition~\ref{prop:covY2}.}
\label{fig:boxplots_example}
\end{figure}

\begin{figure}[!ht]
\centering
\includegraphics[width=0.9\textwidth]{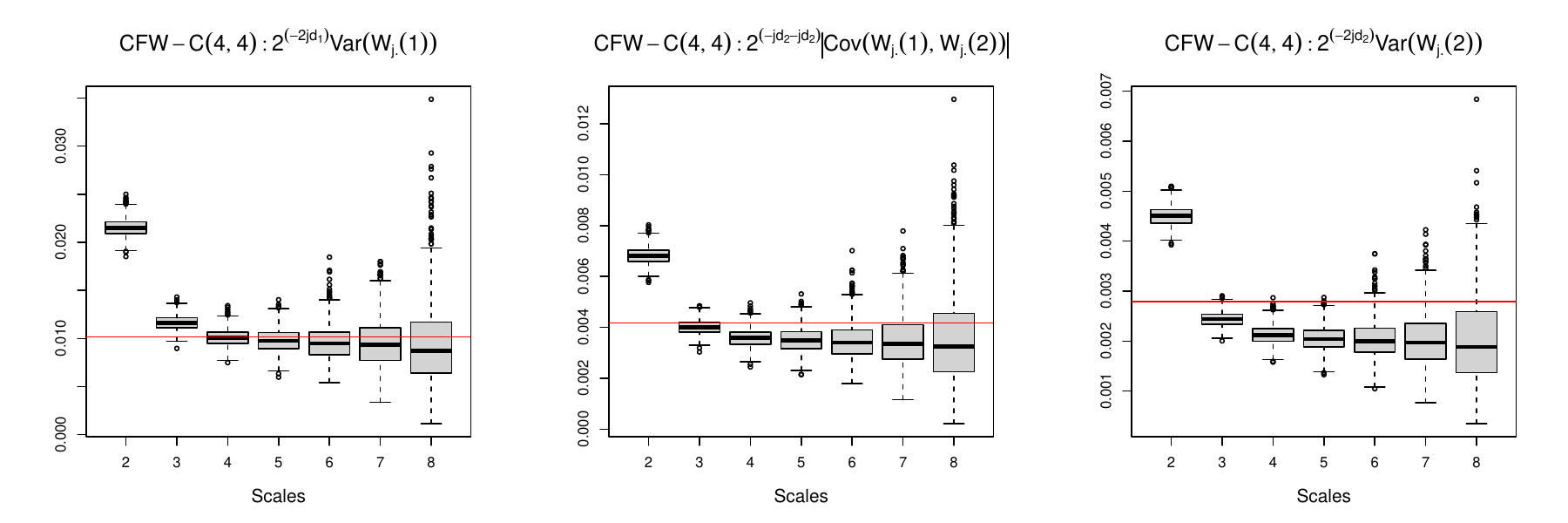}
\caption{Boxplots of normalized sample covariances between CFW-C(4,4) coefficients at different scales for the bivariate process defined in \eqref{eqn:example}. {Plots are the same as the right column of Figure~\ref{fig:boxplots_example} but without the first wavelet scale.} From left to right, panels correspond respectively to the variance of the first component, the magnitude of the covariance and the variance of the second component. Horizontal red lines correspond to the approximation given by Proposition~\ref{prop:covY2}.}
\label{fig:boxplots_example_zoom}

\end{figure}

\section{Estimation}

\label{sec:whittle}

Let $j_0$ and $j_1$, $j_1 \geq j_0 \geq 1$ be respectively the lower and the upper resolution levels used in the estimation procedure. The estimation is based on the vectors of wavelet coefficients $\{\bW_{j,k},\, j_0\leq j\leq j_1,\, k\in\Z\}$. The total number of non-zero coefficients used for estimation is then $n=\sum_{j=j_0}^{j_1} n_j$. Without restriction of generality, we can assume that $L=o(N)$.


\subsection{Estimation procedure}

Based on approximation~\eqref{eqn:approx}, the objective function $\mathcal L(\cdot)$ is defined by the wavelet Whittle approximation of the negative log-likelihood (see \cite{AchardGannaz}) 
\begin{equation} \label{eqn:whittle}
\mathcal L(\bG,\bd) = \frac{1}{n}\sum_{j=j_0}^{j_1} \left[ n_j \log\det\left(\bGamma_j(\bd)\,\bG\,\bGamma_j(\bd)\right)+\sum_{k=0}^{n_j} \overline{\bW_{j,k}}^T\left(\bGamma_j(\bd)\,\bG\,\bGamma_j(\bd)\right)^{-1}\bW_{j,k}\right]
\end{equation}
where $\bGamma_j(\bd)$ is the diagonal matrix with diagonal entries $2^{-j\,d_1},\dots,2^{-j\,d_{\dim}}$, and $\bG$ is the matrix with elements $G_{\ell,m}=\Theta_{\ell,m}K(d_\ell+d_m)$, $1\leq, \ell,m\leq p$. 
We can rewrite $\mathcal L(\cdot)$ as \begin{equation}
\label{eqn:critere}
\mathcal L(\bG,\bd) = \frac{1}{n}\sum_{j=j_0}^{j_1} \left[n_j  \log\det\left(\bGamma_j(\bd)\,\bG\,\bGamma_j(\bd)\right)+\text{trace}\left(\left(\bGamma_j(\bd)\,\bG\,
\bGamma_j(\bd)\right)^{-1}\bI(j)\right)\right],
\end{equation}
where $\bI(j)=\sum_{k=0}^{n_j} \bW_{j,k} \overline{\bW_{j,k}}^T$ denotes the (non-normalized) empirical scalogram at scale $j$.

Note that when $\bG$ is a positive definite Hermitian matrix, for all $j\geq 0$ and for all $\bd\in(-0.5,\infty)^p$, $\det(\bGamma_j(\bd)\,\bG\, \bGamma_j(\bd))$ is real and strictly positive and $\text{trace}\left((\bGamma_j(\bd)\,\bG\,\bGamma_j(\bd))^{-1}\bI_j\right)$ is real. The objective function $\mathcal L(\bG,\bd)$ is hence well-defined for $\bG$ in the set of Hermitian matrices and for all $\bd\in\R^\dim$, and takes its values in $\R$.

Differentiating expression~\eqref{eqn:critere} with respect to the matrix $\bG$ yields
\[
\frac{\partial \mathcal L}{\partial \bG}(\bG,\bd)=\frac{1}{n}\sum_{j=j_0}^{j_1}\left[n_j \bG^{-1}-\bG^{-1}\bGamma_j(\bd)^{-1}
\bI(j)\bGamma_j(\bd)^{-1} \bG^{-1}\right]^T.
\]
Some keys for complex matrix differentiation can be found in \cite{matrixdiff}. Hence, the minimum for fixed $\bd$ is attained at \begin{equation}\label{eqn:G}
\hat \bG(\bd) =\frac{1}{n} \sum_{j=j_0}^{j_1} \bGamma_j(\bd)^{-1}
\bI(j)\bGamma_j(\bd)^{-1}.
\end{equation}
In \cite{Shimotsu07} the resulting objective function only depends of $\bd$ since the phases are parametric whereas in \cite{Baek2020} the authors consider a general form of phases. In both \cite{Shimotsu07} and \cite{Baek2020}, with a Fourier-based approach, a real matrix $\bG(\bd)$ and complex valued matrices $\bGamma_j(\bd)$, including 
the phases $\bigl(\phi_{\ell,m}\bigr)_{\ell,m=,,\dots, \dim}$, are considered. $\bG(\bd)$ and $\bGamma_j(\bd)$ are estimated in a second step, together with parameter $\bd$. They minimize the objective function obtained when replacing $\bG$ by $\hat \bG(\bd)$ in \eqref{eqn:critere}. However, our procedure makes it possible to estimate the magnitude of the correlation even when the phase is equal to $\pi/2$, with imaginary terms in $\bG$.

Replacing $\bG$ by $\hat \bG(\bd)$, the objective function is defined by 
\[
R(\bd):=\mathcal L(\hat \bG(\bd),\bd)-\dim= \log\det(\hat \bG(\bd)) - \frac{1}{n}\sum_{j=j_0}^{j_1}  n_j \log\left(\det\bigl(\bGamma_j(\bd)\bGamma_j(\bd)\bigr)\right).
\]
Since $\bGamma_j(\bd)=\diag{2^{-j\bd}}$, we obtain
\begin{equation}
\label{eqn:R}
R(\bd)= \log\det(\hat \bG(\bd)) + 2\log(2)\left(\frac{1}{n}\sum_{j=j_0}^{j_1} j n_j\right)\left(\sum_{\ell=1}^p d_\ell\right).
\end{equation}
The vector of the long-memory parameters $\bd$ is estimated by $\hat \bd=\argmin_{\bd} R(\bd)$.

In a second step of estimation we define $\hat \bG(\hat \bd)$, estimator of $\bG$. And we recover an estimation of $\bTheta$ by \[
\hat\Theta_{\ell,m}=\hat G_{\ell,m}(\hat\bd)\,/\,K(\hat d_\ell+\hat d_m)\,.
\]

\subsection{Asymptotic convergence}

\label{sec:cvgce}

Following \cite{Moulines08Whittle} and \cite{AchardGannaz}, we introduce an additional condition on the variance of the scalogram {$\{\bI(j),\;j\geq 0\}$}. Examples of linear processes satisfying this condition can be found in Proposition 4 of \cite{AchardGannaz} for real wavelets. With complex wavelets, to obtain the convergence results, we need to define the parameter $L$ depending on $N$. We omit the dependence in the notation except when it is necessary. Therefore, the wavelet bases depend on $N$ as the parameter $L$ depends on $N$. Hence the wavelet scalogram $\{\bI(j), j\geq 0\}$ depends on $N$, via the number of observations used in the calculation of the coefficients and via $L$.

\setlabel{{\bf Condition (C)}}{Condition (C)}{ass:conditionC}

$L$ is a sequence of $N$, $L=L(N)$, such that,
\[
\text{for all } \ell,m=1,\ldots,p,\quad \sup_{N} \sup_{j\geq 0}\, \frac{\abs{\mbox{Var}\left({I_{\ell,m}(j)}\right)}}{n_j \, 2^{2j(d_\ell+d_m)}}  \,<\,\infty\,.
\]

Let $\bd^0$, $\bG^0$ and $\bTheta^0$ denote the true values of the parameters. 
The consistency of the estimators can be established as in \cite{AchardGannaz}.



\begin{theorem}
\label{prop:convergence}
\label{thm:cvgce}
Suppose that \ref{ass:parameters} and assumptions of Proposition~\ref{prop:covY3} hold.
Assume that \ref{ass:conditionC}  is satisfied. Denote $j_N=\max\{j,n_j\geq 1\}$. 

 Let $j_0$ and $j_1$ satisfy $\log(N)^2(2^{-j_0\beta}+ N^{-1/2} 2^{j_0/2})\to 0$ and $j_0 < j_1 \leq j_N$. 
 
\parbox{\textwidth}{Consider CFW-C(M,L) filters with $M\geq 2$ and {$2^{-2j_0}L+N^{-1}2^{j_1}L+\log(N)^3\,L^{-M-1} \to 0$}.
}

Then , $\forall(\ell,m)\in\{1,\ldots,p\}^2$,
\begin{align*}
\hat \bd-\bd^0&=\BigO_\P( L2^{-2j_0}+\log(N)L^{-M-1} + 2^{-j_0\beta}+ N^{-1/2} 2^{j_0/2}),\\
\hat G_{\ell,m}(\hat \bd)-G_{\ell,m}(\bd^0) &= \BigO_\P(\log(N)(L2^{-2j_0}+\log(N)L^{-M-1} +2^{-j_0\beta}+ N^{-1/2} 2^{j_0/2})),\\
\hat \Theta_{\ell,m}-\Theta_{\ell,m}^0&=\BigO_\P(\log(N)(L2^{-2j_0}+\log(N)L^{-M-1} +2^{-j_0\beta}+ N^{-1/2} 2^{j_0/2})).
\end{align*}

Taking $2^{j_0}=N^{1/(1+2\beta)}$ and $L=N^{\frac{\beta}{1+2\beta}\frac{1}{M}}$,
\[
\hat \bd-\bd^0=\BigO_{\P}(N^{-\beta/(1+2\beta)}).
\]
\end{theorem}
Elements of proof are given in Appendix.

The convergence rate $\hat \bd-\bd^0=\BigO_{\P}(N^{-\beta/(1+2\beta)})$ is optimal in minimax sense \citep{giraitis:robinson:samarov:1997}.

The condition on the scales used in the estimation is $\log(N)^2(2^{-j_0\beta}+ N^{-1/2} 2^{j_0/2})\to 0$. As explained in \cite{multiwave}, this indicates that highest frequencies should be removed. The number of scales to remove depends on the short-range dependence via $\beta$. In practice, $j_0$ can be chosen applying a bootstrap procedure on the time series, see \cite{multiwave}.

Next, the parameters $M$ and $L$ in CFW-C(M,L) filters are subject to the conditions \ref{ass:parameters} and $2^{-2j_0}L+N^{-1}2^{j_1}L+\log(N)^3\,L^{-M-1} \to 0$. Condition \ref{ass:parameters} only depends on $M$. It is very similar to the one given in \cite{AchardGannaz} with real filters. It imposes that the number of vanishing moments $M$ is high enough. 

The parameter $L$ quantifies the quality of the analytic approximation of CFW-C(M,L) filters. The assumption $2^{j_1}N^{-1}L\to 0$ results from Lemma \ref{lem:nj}. This is a technical assumption  allowing $n_j$ to be equivalent to $2^{-j}N$ as $N$ goes to infinity. This facilitates the translation of the proofs from real wavelets to complex wavelets. 
 This assumption deals with the highest scale $j_1$. It can be formulated on $j_0$, $N^{-1}2^{j_0}L\to 0$, when $j_1=j_0+\Delta$, with $\Delta <\infty$. The condition $\log(N)^3\,L^{-M-1} \to 0$ guarantees that $L$ is high enough for the analytic approximation to be satisfactory. Alternatively, $L$ should not be too high, and condition $2^{-2j_0}\,L\to 0$ ensures that the size of the wavelets support remains reasonable. As discussed in Section \ref{sec:simu}, in practice, the choice of $L$ is not critical, but this condition influences the choice of $j_0$. It must be higher than the usual choice for real filters. This also appears in the discussion in Section \ref{sec:approx}, where it can be seen that the behavior of the wavelet coefficients at first scales differs from other scales.

{\textit{Remark.} \cite{Baek2020} observe that in case of co-integration, the corresponding magnitude $\Omega_{a,b}$ is equal to 0. Hence, the phase parameter $\varphi_{a,b}$ is not identifiable. To counter this problem, \cite{Baek2020} propose to use another parametrization, where $\Theta_{a,b}$ is decomposed into its real and its imaginary parts. Our procedure estimates the complex matrix $\bTheta$, which is always identifiable, so this discussion is unnecessary here.}

\subsection{Asymptotic normality}

A useful result in estimation is asymptotic normality. For real wavelet-based local Whittle estimation, in a multivariate context, it has been studied by \cite{normality}. The proof of the latter can be extended to common-factor wavelets.

Let us introduce an additional assumption on the process $\bX$.
\begin{enumerate}[label=(M-3)]
\item \label{ass:linear}
There exists a sequence $\{\bA^{(\mathbf{D})}(u)\}_{u\in\Z}$ in $\R^{p\times p}$ such that $\sum_{u\in\Z} \max_{a,b=1,\dots,p}|A_{a,b}^{(\mathbf{D})}(u)|^2<\infty$ and 
\[
\forall t\in\Z,\,~\bigl({(1-\L)^{D_{a}}}X_a(t)\bigr)_{a=1,\dots,p}=\sum_{u\in\Z} \bA^{(\mathbf{D})}({t+u}) \bepsilon(u)
\]
with $\bepsilon(t)$ weak white noise process, in $\R^p$. Let $\mathcal F_{t-1}$ denote the $\sigma$-field of events generated by $\{\bepsilon(s), \,s\leq t-1\}$. Assume that $\bepsilon$ satisfies $\E[\bepsilon(t)|\mathcal F_{t-1}]=0$, $\E[\epsilon_a(t)\epsilon_b(t)|\mathcal F_{t-1}]=\1_{a=b}$ and $\E[\epsilon_a(t)\epsilon_b(t)\epsilon_c(t)\epsilon_d(t)|\mathcal F_{t-1}]=\mu_{a,b,c,d}$ with $|\mu_{a,b,c,d}|\leq \mu_\infty<\infty$, for all $a,b,c,d=1,\ldots,p$.
For all $(a,b)\in\{1,\dots,p\}^2$, for all $\lambda\in\R$, the sequence $(2^{-j\,d_a}\lvert A^{(\mathbf{D})\ast}_{a,b}(2^{-j}\lambda)\rvert)_{j\geq 0}$ is convergent as $j$ goes to infinity.
\end{enumerate}

The asymptotic normality of the estimator of the long-memory parameters is established by our following theorem. For $\bM\in\R^{p\times p}$, $\vect{\bM}$  denotes the operation that transforms a matrix $\bM\in\R^{p\times p}$ into a vector of $\R^{p^2}$.

\begin{theorem}
\label{thm:normality}
Suppose that conditions of Theorem~\ref{thm:cvgce} are satisfied and that assumption \ref{ass:linear} hold.
Let $j_0 < j_1 \leq j_N$ with $j_N=\max\{j,n_j\geq 1\}$ such that \[
j_1-j_0\to \Delta\in\{1,\dots,\infty\},\; \log(N)^2(N 2^{-j_0(1+2\beta)}+ N^{-1/2} 2^{j_0/2})\to 0. 
\] Define $n=\sum_{j=j_0}^{j_1} n_j$.

Consider CFW-C(M,L) filters with $M\geq 2$ and $${N^{-1}2^{j_1}L+}\log(N)^3\,N^{1/2} 2^{-j_0/2}(L2^{-2j_0}+ L^{-M-1})\to 0.$$
Then, \begin{itemize}[topsep=-10pt]
\item $\sqrt{n}(\hat\bd -\bd^0)$ converges in distribution to a centered Gaussian distribution with a variance $\mathbf{V}^{(\bdexp)}(\Delta)$ defined in Appendix, equation \eqref{eqn:vard}.

\item $\vect{\sqrt{n}\left(\hat \bG(\hat\bd)-\bG^0\right)}$ converges in distribution to a centered Gaussian distribution with a variance $\mathbf{V}^{(\bGexp)}(\Delta)$ defined in Appendix, equation \eqref{eqn:varG}.
\end{itemize}
\end{theorem}
The proof is very similar to the one of \cite{normality}. Some points are detailed in Appendix.

The highest scale is $j_1=j_0+\Delta$. The theorem distinguishes the cases $\Delta < \infty$ and  $\Delta = \infty$. Note that when $\Delta<\infty$, the condition $N^{-1/2} 2^{j_1/2}L\to 0$ is equivalent to $N^{-1/2} 2^{j_0/2}L\to 0$. Hence, the condition on $L$ writes as $\log(N)^3\,N^{1/2} 2^{-j_0/2}(L2^{-2j_0}+ L^{-M-1})\to 0$.

{\it Remark.} Observe that with the condition on $j_0$, the minimax rate is not achieved. But we can take $2^{j_0}=N^{a_0}$ with ${1/(1+2\beta)}<{a_0}<1$. If $L$ is defined as $L=N^{b_0}$, then it suffices that $\frac{\beta}{1+2\beta}.\frac{1}{M+1}<b_0<\frac{1}{2}\min\bigl\{1-a_0, 5a_0-1\bigr\}$. For example, we can take $2^{j_0}=N^{(1+\beta)/(1+2\beta)}$ and $L=N^{c_0.\beta/(1+2\beta)}$ with $\frac{1}{2(M+1)}<c_0<\frac{1}{2}$.

As detailed in \cite{normality}, these results allow to build hypothesis tests on the long-memory parameters and on the long-run covariance. 

\section{Simulation study}

\label{sec:simu}

In this section, we verify the accuracy of the covariance approximation given in Proposition~\ref{prop:covY3} and the consistency of the parameters estimates provided in Proposition \ref{prop:convergence} on simulated data. We consider 1000 Monte-Carlo simulations of bivariate long-memory processes $\bX$ observed at $\bX(1), \dots, \bX(N)$ with $N=2^{12}$. For each process, we compute the wavelet coefficients using CFW-PR(4,4) and CFW-C(4,4) filters.

We compare the quality of estimation of parameters $\bd$ to the one given by real wavelets, namely Daubechies' wavelets with 4 vanishing moments. A comparison with a Fourier-based local Whittle procedure is also provided, when the simulated processes are stationary.

Observe that comparisons of Fourier-based and wavelet-based estimation procedures in univariate settings have been done previously in \cite{nielsen2005finite} and \cite{fay:moulines:roueff:taqqu:2008survey}. In multivariate settings, comparisons have been done in \cite{multiwave} with a parametric phase. With a general phase, \cite{Baek2020} proposed Monte Carlo simulations using Fourier-based approach. The specificity here is to extend simulations of the former to non-stationary processes.

The estimated parameters are $\bd=(d_1,d_2)$, the magnitude of the long-run covariance ${\bOmega}$, the phase $\phi=\phi_{1,2}$ and the long-run correlation $\rho=\frac{{\Omega_{1,2}}}{\sqrt{{\Omega_{1,1}}{\Omega_{2,2}}}}$. For each parameter, we will evaluate the quality of estimation by the bias, the standard deviation (std) and the Root Mean Squared Error, RMSE = $\sqrt{\text{bias}^2+\text{std}^2}$.

Two models are considered: models admitting a linear representation called ARFIMA, and multivariate fractional Brownian motions (mFBM).

\subsection{ARFIMA models}

\label{sec:arfima}

We first provide an estimation example on linear time series. Let $\xi$ be a $p$-dimensional white noise with $\E[\xi(t)\mid \mathcal F_{t-1}]=0$ and $\E[\xi(t)\xi(t)^T\mid \mathcal F_{t-1}]=\bSigma$ with $\bSigma$ positive definite, where $\mathcal F_{t-1}$ is the $\sigma$-field generated by $\{\xi(s),\,s<t\}$. The spectral density of $\xi$, denoted $\bff_\xi(.)$,  satisfies $\bff_\xi(\lambda)=\bSigma$, for all $\lambda\in\R$.

Let $(\bA_k)_{k\in\N}$ be a sequence in $\R^{p\times p}$ with $\bA_0$ the identity matrix and $\sum_{k=0}^\infty \|\bA_k\|^2<\infty$. Let $\bA(\cdot)$ be the discrete Fourier transform of the sequence, $\bA(\lambda)=\sum_{k=0}^{\infty} \bA_k e^{i k\lambda}$. We assume that $|\bA(\L)|$ has all its roots outside the unit circle so that $\bA(\cdot)^{-1}$ is defined and smooth on $\mathbb{R}$. We also define $(B_k)_{k\in\N}$ to be a sequence in $\R^{p\times p}$ with $B_0$ the identity matrix and $\sum_{k=0}^\infty \|B_k\|^2<\infty$. Let $\bB(\cdot)$ be the discrete Fourier transform of the sequence, $\bB(\lambda)=\sum_{k=0}^{\infty} \bB_k e^{i k\lambda}$.

Let us define the process $\bX=\{\bX(t),\,t\in\Z\}$ by
\begin{equation}
\label{eqn:arfima}
\bA(\L)\,\diag{(1-\L)^{\bd}}\,\bX(t)=\bB(\L)\xi(t),\,t\in\Z.
\end{equation} 
The spectral density of $\bX$ satisfies
\[
\bff(\lambda)=(1-e^{-i\lambda})^{-\bd}\bA(e^{-i\lambda})^{-1}\bB(e^{-i\lambda})\bff_\xi(\lambda)\bB(e^{i\lambda})^T {\bA(e^{i\lambda})^T}^{-1}(1-e^{i\lambda})^{-\bd}.
\]
In particular
\[
f_{\ell,m}(\lambda)\sim_{\lambda\to 0^+} G_{\ell,m}e^{-i\pi/2(d_\ell-d_m)}\lambda^{-(d_\ell+d_m)}\;,
\]
with $\bG=\bA(0)^{-1}\bB(0)\bff_\xi(\lambda)\bB(0)^T{\bA(0)^T}^{-1}=\bA(0)^{-1}\bB(0)\bSigma\bB(0)^T{\bA(0)^T}^{-1}$ a real valued matrix. Condition \ref{ass:beta} is satisfied with $\beta =\min_\ell(d_\ell)$. In this case $\bff(0^+)= \bff(0^-)$.

This corresponds to Model A of \cite{Lobato97}. Note that this model satisfies the definition of LRD processes of \cite{KechagiasPipiras}. In \cite{KechagiasPipiras}, the process $\xi$ is not necessarily a white-noise process, and verifies $\bff_\xi(\lambda)\sim_{\lambda\to 0^+} \bSigma$, which is indeed true for a white-noise process.

Following \eqref{eqn:arfima}, we have simulated $\bX(1),\dots \bX(N)$ in \eqref{eqn:arfima} with $N=2^{12}$, null $\bA_k$ and $\bB_k$ for $k\geq 0$. That is, there is no short-range terms in the model. We consider three sets of values for $\bd$, $\bd\in\{(0.2,0.2),\;(0.2,0.4),$ $\;(0.2,0.8) \}$. Matrix $\bSigma$ is equal to $\begin{pmatrix}1& \rho \\
\rho & 1\end{pmatrix}$, with $\rho=0.8$. The phase is equal to $\pi(d_1-d_2)/2$ which is respectively equal to $0, \pi/10, 3\pi/10$. Simulations were done using R package {\it multiwave} \citep{multiwave-package}.

{\it Remark.} The objective of this simulation part is to compare estimations based respectively on real and complex wavelet filters. Hence, only simulations with null $\bA_k$ and $\bB_k$ are considered. We refer to \cite{multiwave} for results with non null AR and MA parts, with real wavelet filters.

Figure~\ref{fig:boxplots_arfima} displays the boxplots of the correlations between the wavelet coefficients obtained by CFW-PR(4,4) filter at different scales. It illustrates that the approximation of Proposition~\ref{prop:covY3} is valid, especially for high scales (lowest frequencies), even if it has not been established theoretically for such filters. The figure shows that the approximation of Proposition \ref{prop:covY3} is slightly more accurate for the real part of wavelet correlations than for the imaginary part.

\begin{figure}[!ht]
\begin{subfigure}[t]{0.32\textwidth}
\caption{$\bd=(0.2, 0.2)$}
\includegraphics[width=\textwidth]{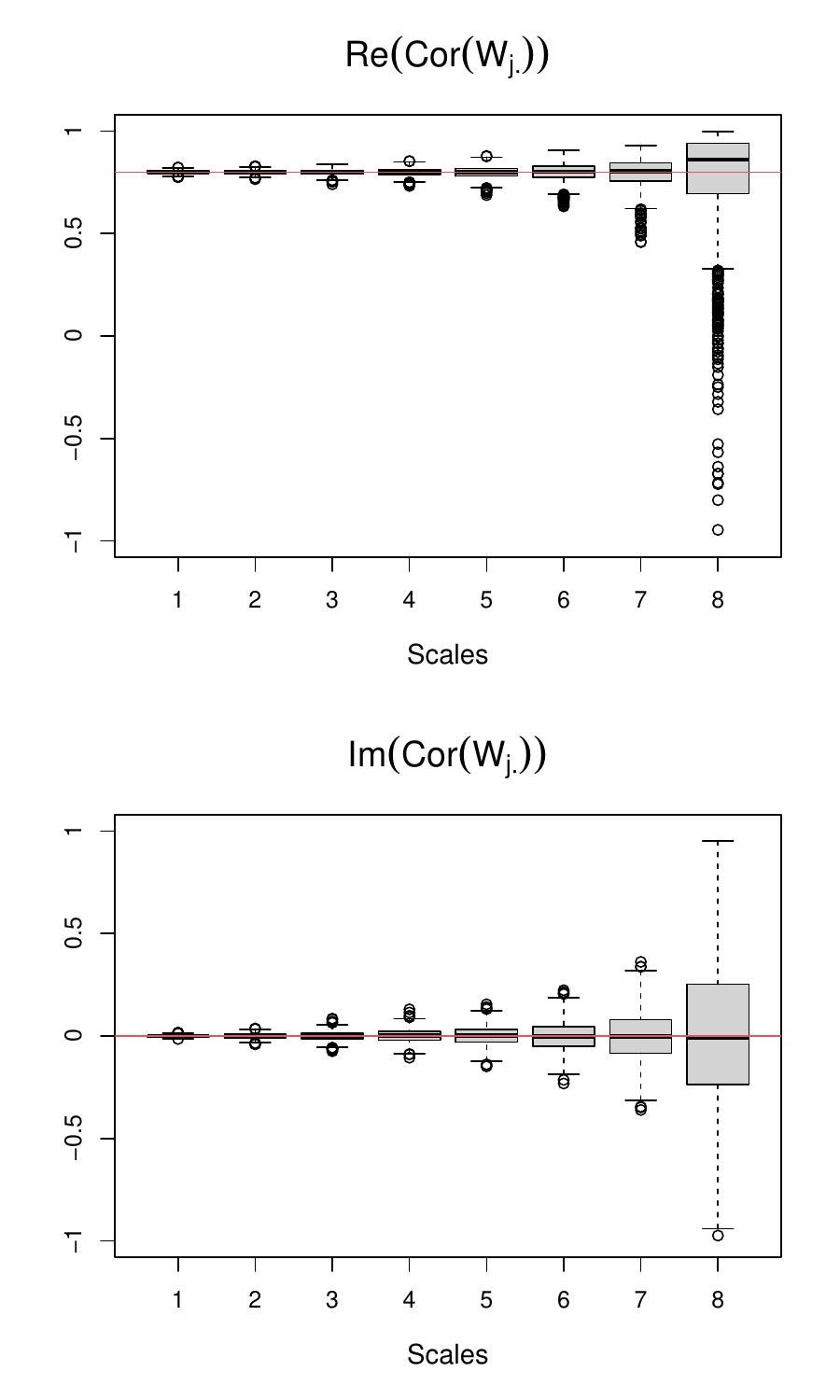}
\end{subfigure}
\begin{subfigure}[t]{0.32\textwidth}
\caption{$\bd=(0.2, 0.4)$}
\includegraphics[width=\textwidth]{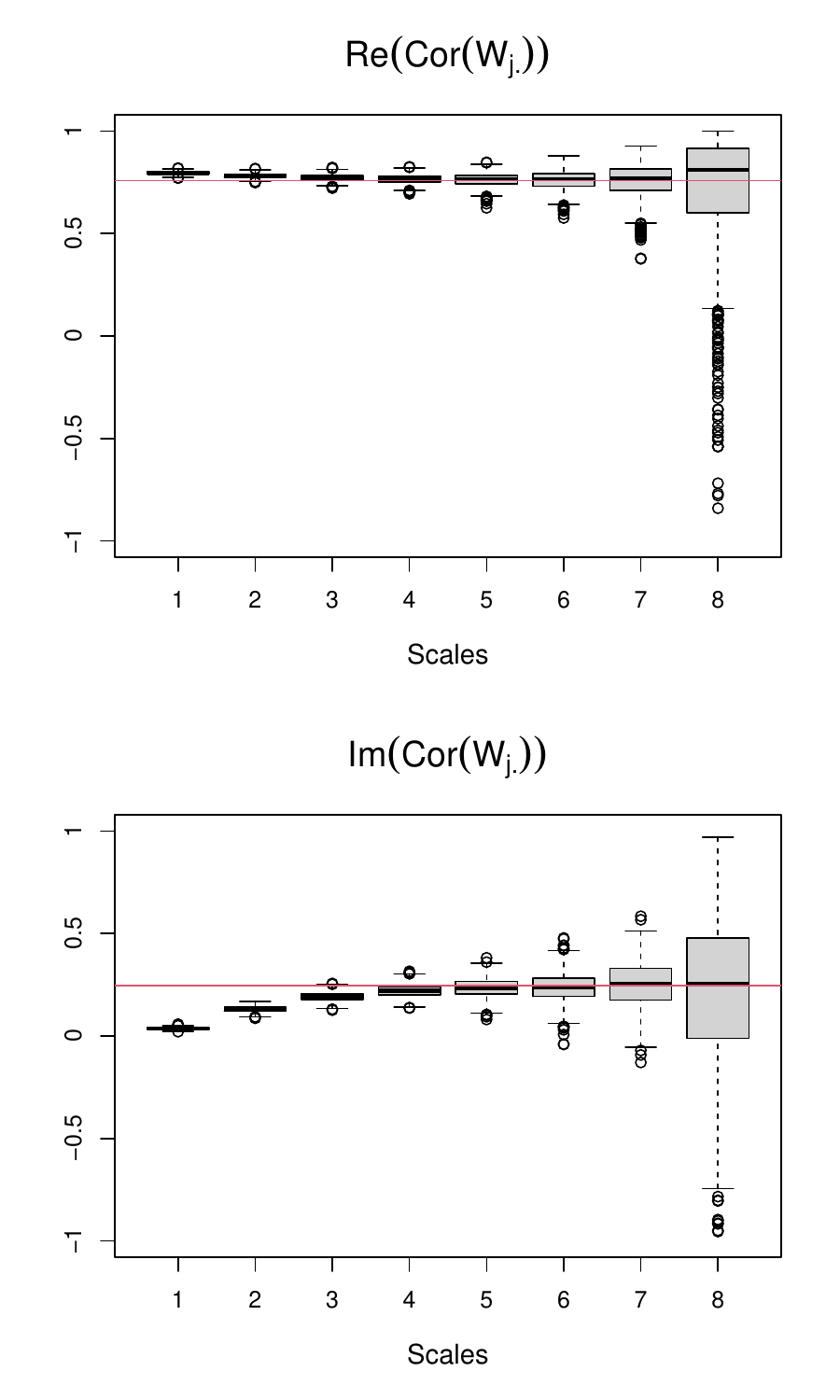}
\end{subfigure}
\begin{subfigure}[t]{0.32\textwidth}
\caption{$\bd=(0.2, 0.8)$}
\includegraphics[width=\textwidth]{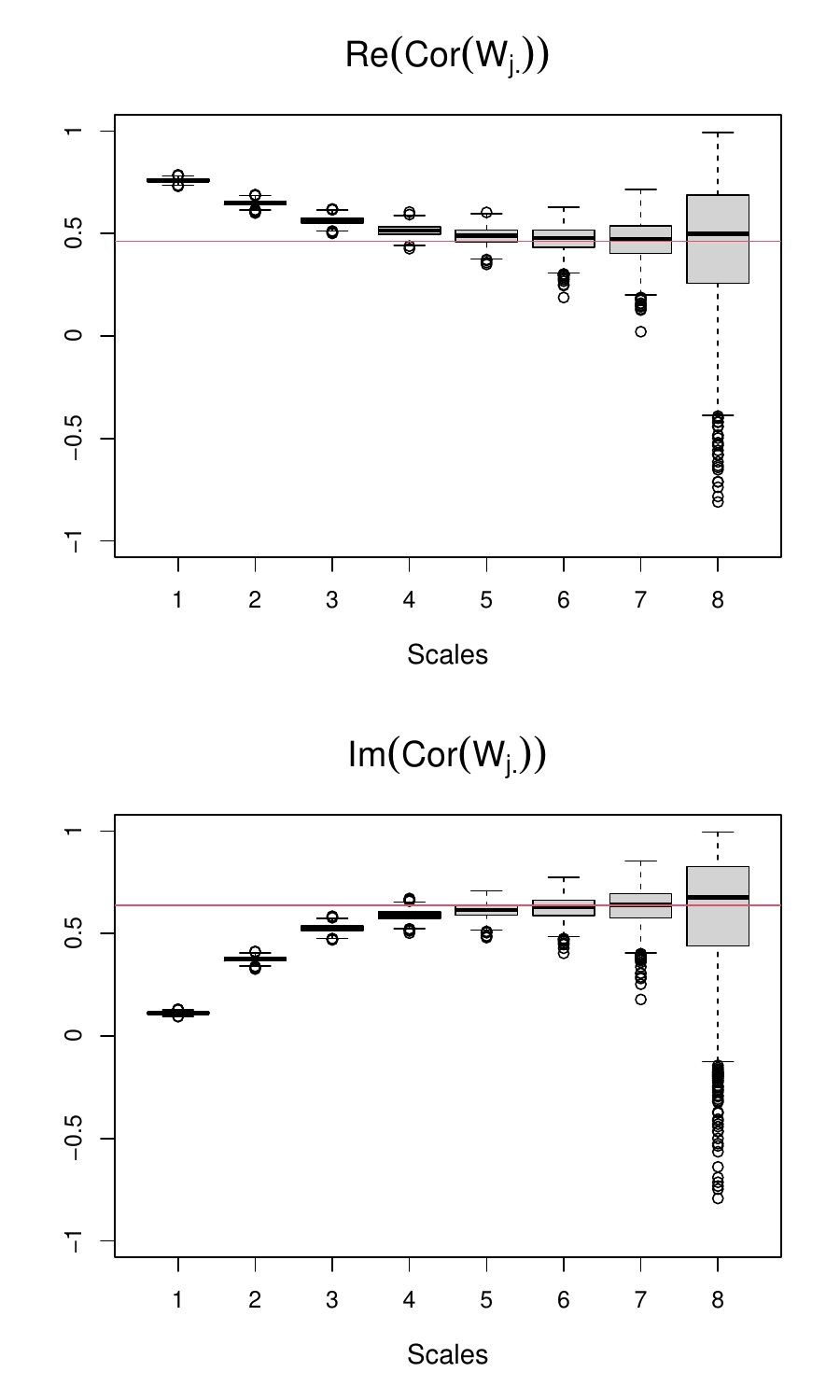}
\end{subfigure}
\caption{Boxplots of correlation between CFW-PR(4,4) coefficients at different scales for ARFIMA processes. First row gives the real part of the correlations and second row gives the imaginary part. Each column corresponds to a given value of parameter $\bd$. Horizontal red lines correspond to the approximation given by Proposition~\ref{prop:covY3}, $\rho\cos(\phi)r_K$ for the real part and  $\rho\sin(\phi)r_K$ for the imaginary part, with $r_K=K(d_1+d_2)/\sqrt{K(2\,d_1)K(2\,d_2)}$.}
\label{fig:boxplots_arfima}

\end{figure}

The results for the estimation of the long-memory parameters $\bd$ are displayed respectively in Table~\ref{tab:arfima} for CFW-PR(4,4) filter and in Table~\ref{tab:arfima_C} for CFW-C(4,4) filter. Based on \cite{multiwave}, the only hyperparameter to choose is the minimal scale $j_0$. The example in Section~\ref{sec:approx} illustrates that the behaviors of Daubechies' wavelets and CFW-PR(M,L) are very similar with respect to scales. Thus, for both of them, we consider $j_0=1$ when the two components of the time series are stationary, and $j_0=2$ when a component is not stationary. These choices are motivated by previous studies \citep{multiwave}.

 Table~\ref{tab:arfima} shows that the estimation of $\bd$ with CFW-PR(M,L) filters is good and similar to the Daubechies' real wavelet-based estimation.

When $\bd=(0.2,0.2)$ or $\bd=(0.2,0.4)$, the processes are stationary, it is then possible to estimate the parameters of the model using a Fourier-based procedure. Following \cite{Baek2020}, we implemented a Fourier-Based local Whittle estimation by keeping our parametrization. In particular,  the complex matrix $\bTheta$ is not decomposed in real and imaginary parts or in magnitude and phase in the objective function. Let $\lambda_j = 2\pi j/N$, $j=1, \dots, m$, be the Fourier frequencies used in estimation, $m\in\N$. Define $\bLambda_j^F(\bd)=\diag{\lambda_j^{\bd} \rme^{-\rmi\lambda_j\,\bd/2}}$.
The Fourier-based estimators $(\hat\bd^{MFW},\hat\bTheta^{MFW})$ are minimizers of the objective function 
\[ 
\frac{1}{m}\sum_{j=1}^{m} \left[ \log\det\Bigl(\bLambda_j^F(\bd)\bTheta
\bLambda_j^F(\bd)^\ast\Bigr)+\bW^{F}(\lambda_j)^{\ast}\Bigl(\bLambda_j^F(\bd)\bOmega(\bd) \bLambda_j^F(\bd)^{\ast} \Bigr)^{-1}\bW^F(\lambda_j)\right].
\]
Note that this objective function differs from \cite{Baek2020} since we chose a different parametrization (see Section \ref{sec:cvgce}).
We consider $m=N^{0.65}$, as suggested by \cite{Lobato99,Shimotsu07,Nielsen11}.

  Concerning the estimation with CFW-C(M,L), it has been shown in Section \ref{sec:approx} that the approximation of Proposition \ref{prop:covY2} is not accurate for the first scales. Therefore, the procedure is executed with $j_0=4$ in each case, to reduce the bias of the long-run covariance estimation. Since fewer scales are available in the procedure, the variance increases and the quality of the estimation is lower than the one based on CFW-PR(M,L). Table~\ref{tab:arfima_C} illustrates that the quality of the estimation with CFW-C(M,L) filters presents a good accuracy, however lower than that with CFW-PR(M,L) filters.

\begin{table}[!ht]
\centering
\begin{tabular}{lccccc}
\toprule
 $\bd$ &  bias & std & RMSE & PR/Real & {PR/Fourier}\\
 \midrule \midrule
 0.2 & -0.0066 & 0.0165 & 0.0178 & 1.0375 & {1.0000} \\ 
 0.2 & -0.0069 & 0.0156 & 0.0171 & 1.0228 & {1.0108}\\   \midrule
 0.2 & -0.0080 & 0.0165 & 0.0183 & 1.1537 & {1.0000}\\ 
 0.4 & -0.0138 & 0.0159 & 0.0211 & 1.1888 & {0.9606}\\ \midrule
 0.2 & -0.0094 & 0.0255 & 0.0272 & 0.9072 & .\\ 
 0.8 & -0.0145 & 0.0270 & 0.0307 & 1.3319 & . \\ \bottomrule
\end{tabular}
\caption{Results for the estimation of long-memory parameters $\bd$ with CFW-PR(4,4) filters using ARFIMA processes. $j_0=1$ for $\bd\in\{(0.2,0.2), (0.2,0.4)\}$ and $j_0=2$ for $\bd\in\{(0.2,0.8)\}$. PR/Real denotes the ratio between the RMSE given by CFW-PR(4,4) filter and the RMSE given by Daubechies' real filter. {PR/Fourier denotes the ratio between the RMSE given by CFW-PR(4,4) filter and the RMSE given by a Fourier-based local Whittle procedure.}}
\label{tab:arfima}
\end{table}

\begin{table}[!ht]
\centering
\begin{tabular}{lcccc}
\\
\toprule
 $\bd$ &  bias & std & RMSE & C/PR \\
 \midrule \midrule
 0.2 & -0.0152 & 0.0401 & 0.0429 & 2.4155 \\ 
 0.2 & -0.0145 & 0.0392 & 0.0418 & 2.4504 \\   \midrule
 0.2 & -0.0158 & 0.0394 & 0.0425 & 2.3229 \\ 
 0.4 & -0.0150 & 0.0385 & 0.0413 & 1.9626 \\  \midrule
 0.2 & -0.0166 & 0.0397 & 0.0430 & 1.5837 \\ 
 0.8 & -0.0163 & 0.0390 & 0.0422 & 1.3762 \\    \bottomrule
\end{tabular}
\caption{Results for the estimation of long-memory parameters $\bd$ with CFW-C(4,4) filters using ARFIMA processes. $j_0=4$. C/PR denotes the ratio between the RMSE given by CFW-C(4,4) filter and the RMSE given by CFW-PR(4,4) filter in Table \ref{tab:arfima}.} 
\label{tab:arfima_C}
\end{table}

Table~\ref{tab:arfimaG} and Table~\ref{tab:arfimaG_C} give the results for the estimation of the covariance structure, respectively for CFW-PR(4,4) filter and for CFW-C(4,4) filter. The quality of the estimation using CFW-PR(4,4) improves significantly the results obtained with real wavelets. This is not the case for CFW-C(4,4) filters where the quality of estimation deteriorates.

With CFW-PR(M,L) filters, the results for the phase parameter $\phi$ are less satisfactory. A bias term can be observed when the phase $\phi$ increases. This term is for example of order $\pi/10$ when estimating the phase of $3\pi/10$, corresponding to the case $\bd=(0.2, 0.8)$. {However, our results for the stationary cases are comparable or significantly better than those obtained by the Fourier-based Whittle estimation. Similar qualities were observed in \cite{Baek2020}.} Estimating the phase is challenging. Interestingly, the quality of the estimation based on CFW-C(M,L) filters is much stable with respect to the phase values. Indeed, even when the phase value increases, the bias remains constant.

\begin{table}[!ht]
\centering
\begin{tabular}{lcccccc}
\toprule
 $\bd$ & & bias & std & RMSE & PR/Real & {PR/Fourier}\\ 
 \midrule \midrule
 ( 0.2 , 0.2 ) & $\Omega_{1,1}$ & 0.0087 & 0.0227 & 0.0243 & 0.5089 & {1.6669}\\ 
 	 & $\Omega_{1,2}$ &  0.0076 & 0.0204 & 0.0218 & 0.5509 & {1.6621}\\ 
 	 & $\Omega_{2,2}$ &  0.0092 & 0.0229 & 0.0247 & 0.5184 & {1.6796} \\ 
  	 & correlation    &    4e-04 & 0.0057 & 0.0057 & 0.9702 & {0.6987} \\ 
 	  & phase    &    1e-04 & 0.0081 & 0.0081 & .& {1.0869}\\ \midrule
 ( 0.2 , 0.4 ) & $\Omega_{1,1}$ & 0.0082 & 0.0229 & 0.0243 & 0.4919  & {1.7049}\\ 
 	 & $\Omega_{1,2}$ &  0.0028 & 0.0206 & 0.0208 & 0.4883 & {1.4759}\\ 
 	 & $\Omega_{2,2}$ &  0.0146 & 0.0231 & 0.0273 & 0.3142 & {1.3126}\\ 
  	 & correlation    &    -0.0063 & 0.0058 & 0.0085 & 0.6079  & {0.6950}\\ 
 	  & phase    &    0.1879 & 0.0083 & 0.1880 & .& {0.1677}\\ \midrule
 ( 0.2 , 0.8 ) & $\Omega_{1,1}$ & -0.0184 & 0.0365 & 0.0409 & 0.8883 &. \\ 
 	 & $\Omega_{1,2}$ &  -0.0860 & 0.0292 & 0.0908 & 0.4163 &.\\ 
 	 & $\Omega_{2,2}$ &  -0.1145 & 0.0352 & 0.1198 & 2.1877 &.\\ 
  	 & correlation    &    -0.0342 & 0.0083 & 0.0351 & 0.1788 &.\\ 
 	  & phase    &    0.2790 & 0.017 & 0.2795 & . &.\\ 
\bottomrule
\end{tabular}
\caption{Results for the estimation of matrices $\bTheta$ with CFW-PR(4,4) filters on ARFIMA processes. $j_0=1$ for $\bd\in\{(0.2,0.2), (0.2,0.4)\}$ and $j_0=2$ for $\bd\in\{(0.2,0.8)\}$. {PR/Fourier denotes the ratio between the RMSE given by CFW-PR(4,4) filter and the RMSE given by a Fourier-based local Whittle procedure.}}
\label{tab:arfimaG}

\end{table}

\begin{table}[!ht]
\centering
\begin{tabular}{lccccc}
\\
\toprule
 $\bd$ &  & bias & std & RMSE & C/PR \\ 
 \midrule \midrule
 ( 0.2 , 0.2 ) & $\Omega_{1,1}$ & 0.0488 & 0.175 & 0.1817 & 7.5165 \\ 
 	 & $\Omega_{1,2}$ &  0.0367 & 0.1252 & 0.1305 & 6.0138 \\ 
 	 & $\Omega_{2,2}$ &  0.0451 & 0.1685 & 0.1744 & 7.1674 \\ 
 	  & correlation    &     0.0012 & 0.0173 & 0.0173 & 2.9994 \\ 
 	  & phase    &     8e-04 & 0.0367 & 0.0367 & 4.5335\\  \midrule
 ( 0.2 , 0.4 ) & $\Omega_{1,1}$ & 0.0469 & 0.1726 & 0.1789 & 7.3547 \\ 
 	 & $\Omega_{1,2}$ &  0.0124 & 0.1198 & 0.1205 & 5.7810 \\ 
 	 & $\Omega_{2,2}$ &  -0.0086 & 0.1628 & 0.1630 & 5.9735 \\ 
 	  & correlation    &     -3e-04 & 0.0172 & 0.0172 & 2.0196 \\ 
 	  & phase    &     0.0225 & 0.0364 & 0.0428 & 0.2278\\  \midrule
 ( 0.2 , 0.8 ) & $\Omega_{1,1}$ & 0.0476 & 0.1715 & 0.178 & 4.3507 \\ 
 	 & $\Omega_{1,2}$ &  -0.0266 & 0.1156 & 0.1186 & 1.3064 \\ 
 	 & $\Omega_{2,2}$ &  -0.0973 & 0.1499 & 0.1787 & 1.4916 \\ 
 	  & correlation    &     -0.0026 & 0.0175 & 0.0177 & 0.5036 \\ 
 	  & phase    &     0.0645 & 0.0357 & 0.0737 & 0.2637 \\ 
 \bottomrule
\end{tabular}
\caption{Results for the estimation of matrices $\bTheta$ with CFW-C(4,4) filters on ARFIMA processes. $j_0=4$. C/PR denotes the ratio between the RMSE given by CFW-C(4,4) filter and the RMSE given by CFW-PR(4,4) filter in Table \ref{tab:arfimaG}.}
\label{tab:arfimaG_C}
\end{table}

It can be observed that the estimation with CFW-C(4,4) filters has a lower bias and a higher variance than the estimation with CFW-PR(4,4) filters. The higher $j_0$, the lower the bias, but the higher the variance. For the estimation of $\bd$, the main difficulty is to control the variance according to the choice of $j_0$. Since the bias for the phase estimation is critical, a higher $j_0$ for CFW-PR(4,4) filters can be considered. Moreover, as illustrated in Figure~\ref{fig:boxplots_arfima}, the quality of the approximation of the imaginary part of the correlation is not accurate for the three highest scales. Thus, to ensure a good approximation and consequently a small bias for the correlation and the phase, it seems more appropriate to remove the three first scales and to consider $j_0=4$ also for CFW-PR(M,L), even if this increases the variance.

Table \ref{tab:arfima2} and Table \ref{tab:arfimaG2} display the results  considering $j_0=4$, with CFW-PR(M,L), respectively for the long-run dependence parameter $\bd$ and for the long-run correlation and the phase. The results for the long-run covariance are omitted, for simplicity. Table \ref{tab:arfima2} shows that even if the bias of $\hat\bd$ increases, it has a similar order of magnitude. Interestingly, Table \ref{tab:arfimaG2} highlights that the quality of the estimations with CFW-C(4,4) and CFW-PR(4,4) filters are then very similar. In particular, the bias of the phase and of the correlation decrease, as shown in Table~\ref{tab:arfimaG}. With $j_0=4$, the RMSE is significantly lower than the one of the Fourier-based local Whittle estimator. This approach, hence, improves significantly the estimation of the phase compared to the other procedures. Note that it is nevertheless sensitive to the choice of $j_0$.

\begin{table}[!ht]
\centering
\begin{tabular}{lcccccc}
\toprule
 $\bd$ &  bias & std & RMSE & PR/Real & {PR/Fourier} & {C/PR}\\
 \midrule \midrule
 0.2 & -0.0020 & 0.0412 & 0.0412 & 1.1056 & {1.0000}&{1.0068}\\ 
 0.2 &   -0.0037 & 0.0419 & 0.0421 & 1.1177 & {1.5829}&{1.0121}\\   \midrule
 0.2 & -0.0029 & 0.041 & 0.0411 & 1.0477 &{1.0000 }&{ 0.9815}\\ 
 0.4 &   -0.0035 & 0.0418 & 0.0419 & 1.0560&{1.5435 }&{0.9682} \\   \midrule
 0.2 & -0.0054 & 0.0409 & 0.0412 & 0.9472 & .&{1.0173}\\ 
 0.8 &   -0.0078 & 0.0423 & 0.0430 & 0.9049 & .&{0.9855} \\  \bottomrule
\end{tabular}
\caption{Results for the estimation of long-memory parameters $\bd$ with CFW-PR(4,4) filters on ARFIMA processes. For real filter, $j_0=1$ for $\bd\in\{(0.2,0.2), (0.2,0.4)\}$ and $j_0=2$ for $\bd\in\{(0.2,0.8)\}$. For CFW-PR(4,4) filter, $j_0=4$. PR/Real denotes the ratio between the RMSE given by CFW-PR(4,4) filter and the RMSE given by Daubechies' real filter. PR/Fourier denotes the ratio between the RMSE given by CFW-PR(4,4) filter and the RMSE given by a Fourier-based local Whittle procedure. C/PR denotes the ratio between the RMSE given by CFW-C(4,4) filter in Table \ref{tab:arfima_C} and the RMSE given by CFW-PR(4,4) filter.}
\label{tab:arfima2}
\end{table}

\begin{table}[!ht]
\centering
\begin{tabular}{lcccccc}
\\
\toprule
 $\bd$ &  & bias & std & RMSE & {PR/Fourier} & {C/PR}\\ 
 \midrule \midrule
 ( 0.2 , 0.2 )    	 & correlation    &    8e-04 & 0.0166 & 0.0166 & {0.9944} & {1.0477}\\ 
 	  & phase    &    0 & 0.0340 & 0.0340 & {1.5912} &{1.0346}\\   \midrule
 ( 0.2 , 0.4 )  
 	  	 & correlation    &    2e-04 & 0.0164 & 0.0164 &{ 0.9745} &{1.0362}\\ 
 	  & phase    &    0.0241 & 0.0345 & 0.0421 & {0.1320} &{0.9842} \\ 
  \midrule
 ( 0.2 , 0.8 )
  	 & correlation    &    -0.0033 & 0.0169 & 0.0172 & . &{1.0291}\\ 
 	  & phase    &    0.0654 & 0.0341 & 0.0738 & . &{0.9995}\\ 
 \bottomrule
\end{tabular}
\caption{Results for the estimation of matrices $\bTheta$ with CFW-PR(4,4) filters on ARFIMA processes with $j_0=4$. {PR/Fourier denotes the ratio between the RMSE given by CFW-PR(4,4) filter and the RMSE given by a Fourier-based local Whittle procedure.  C/PR denotes the ratio between the RMSE given by CFW-C(4,4) filter in Table \ref{tab:arfimaG_C} and the RMSE given by CFW-PR(4,4) filter.}}
\label{tab:arfimaG2}
\end{table}

\subsection{Multivariate fractional Brownian motions}

We now consider a multivariate fractional Brownian motion (mFBM). Since mFBM are not stationary, Fourier-based estimation is not available (without a differentiation). A specificity of mFBM is that it does not have a linear representation, even if it can be seen as the limit process of a linear representation, see \cite{mfbm-linear}.

The $p$-multivariate fractional Brownian motion ${\bX=\{\bX(t),\,t\in\R\}}$ of long-memory parameter $\bd$, for any $\bd\in(0.5,1.5)^p$ is a process satisfying the three following properties:
\begin{itemize}
\item $\bX(t)$ is Gaussian for any $t\in\R$;
\item $\bX$ is self-similar with parameter $\bd-1/2$, {\it i.e.} for every $t\in\R$ and $a>0$, $(X_{1}(at),\ldots, X_{p}(at))$ has the same distribution as $(a^{d_1-1/2}X_{1}(t),\ldots, a^{d_p-1/2}X_{p}(t))$;
\item the increments are stationary.
\end{itemize}
Another usual parametrization is the one with Hurst parameters, equal to $\bd-1/2$.

We introduce the following quantities, for $1\leq \ell, m\leq p$: \begin{align*}
\sigma_\ell&=\E[X_\ell(1)^2]^{1/2}\\
r_{\ell,m}&=r_{m,\ell}= \mbox{Cor}(X_\ell(1),X_m(-1))\\
\eta_{\ell,m} &= -\eta_{m,\ell}=(\mbox{Cor}(X_\ell(1),X_m(-1))-\mbox{Cor}(X_\ell(-1),X_m(1))) /c_{\ell,m} \\ & \qquad\text{~with~} c_{\ell,m}=\begin{cases}2(1-2^{d_\ell+d_m-1})
&\text{~~if~}d_\ell+d_m\neq 1,\\
2\log(2) &\text{~~if~} d_\ell+d_m=1,\end{cases}
\end{align*}
where $\mbox{Cor}(X_1,X_2)$ denotes the Pearson correlation between variables $X_1$ and $X_2$.
The quantities $(\eta_{\ell,m})_{\ell,m=1,\dots,p}$ measure the dissymmetry of the process. A mFBM is time reversible if the distribution of $\bX(-t)$ is equal to the distribution of $\bX(t)$ for every~$t$. \cite{DidierPipiras} established that zero-mean multivariate Gaussian
stationary processes $\bX$ is equivalent to $\E[\bX_\ell(t)\bX_m(s)] = \E[\bX_\ell(s)\bX_m(t)]$ for all $(s,t)$, which corresponds to the definition of time reversibility used in \cite{KechagiasPipiras}. A mFBM is time-reversible if and only if $\eta_{\ell,m}=0$ for all $(\ell,m)$.

\cite{mFBM13} characterize the spectral behaviour of the increments of a mFBM.  If $f_{\ell,m}^{(1,1)}{(.)}$ denotes the cross-spectral density of {$\{(1-\L) X_{\ell}(t), (1-\L)X_{m}(t)), t\in\Z\}$}, then $$f_{\ell,m}^{(1,1)}(\lambda)=2\,\Omega_{\ell,m}\frac{1-\cos(\lambda)}{|\lambda|^{d_\ell+d_m}}e^{i\phi_{\ell,m}},$$ with \begin{align*}\Omega_{\ell,m}&=\begin{cases}
{\sigma_\ell\sigma_m}\Gamma(d_\ell+d_m)\left(r_{\ell,m}^2\cos^2(\frac{\pi}{2}(d_\ell+d_m))+\eta_{\ell,m}^2\sin^2(\frac{\pi}{2}(d_\ell+d_m)) \right)^{1/2} & \text{~~if~}d_\ell+d_m\neq 2\\
{\sigma_\ell\sigma_m}\Gamma(d_\ell+d_m)\left( r_{\ell,m}^2+\eta_{\ell,m}^2\frac{\pi^2}{4} \right)^{1/2} & \text{~~if~}d_\ell+d_m= 2
\end{cases}\\
\phi_{\ell,m}&=\begin{cases}
\mbox{atan}\left(\frac{\eta_{\ell,m}}{r_{\ell,m}}\tan(\frac{\pi}{2}(d_\ell+d_m))\right)&\text{~~if~}d_\ell+d_m\neq 2\\
\mbox{atan}\left(\frac{\eta_{\ell,m}}{r_{\ell,m}}\frac{\pi}{2}\right) & \text{~~if~}d_\ell+d_m= 2.
\end{cases}
\end{align*}
Let $\bTheta$ be given by $\bTheta=(\Omega_{\ell,m}e^{i\phi_{\ell,m}})_{\ell,m=1,\dots,p}$. When $\lambda$ tends to $0$, the spectral density $f_{\ell,m}^{(1,1)}(\lambda)$ is equivalent to $\Theta_{\ell,m}|\lambda|^{-(d_\ell+d_m-2)}$.
Thus, assumption~\ref{ass:zero-frequency} holds. Assumption \ref{ass:beta} is satisfied for any $0<\beta<2$. We can verify easily that time-reversibility is still equivalent to $\phi_{\ell,m}=0$ in this setting.

Note that the set of parameters $\{d_\ell, \sigma_\ell, r_{\ell,m}, \eta_{\ell,m}, \ell,m=1,\dots,p\}$ is not identifiable. Indeed, for $0<a<1$, $\{d_\ell, \sigma_\ell, r_{\ell,m}, \eta_{\ell,m}, \ell,m=1,\dots,p\}$ and $\{d_\ell, \sqrt{a}\,\sigma_\ell, r_{\ell,m}/a, \eta_{\ell,m}/a, \ell,m=1,\dots,p\}$ lead to the same expressions of $f_{\ell,m}^{(1,1)}(\cdot)$. It thus seems reasonable to parameterize the fractional Brownian motion by $\{d_\ell, \Theta_{\ell,m},\; \ell,m=1,\dots,p\}$.

We consider two mFBM, both with parameters $\sigma_1=\sigma_2=1$ and $\bd = (1, 1.2)$. 
\begin{description}
\item[Case 1.] $\eta_{1,2}= 0.9$, $r_{1,2}=0.6$.~\\ The phase $\phi_{1,2}$ is approximately equal to $\pi/7$ and $\bOmega\simeq\begin{pmatrix} 1.000 & 0.699 \\
 0.699. &1.005\end{pmatrix}
$, giving a long-run correlation $\rho\simeq0.70$.

\item[Case 2.] $\eta_{1,2}= -0.6$, $r_{1,2}=0.2$.~\\  The phase $\phi_{1,2}$ is approximately equal to $-\pi/4$ and $\bOmega\simeq\begin{pmatrix} 1.000 & 0.293 \\
 0.293& 1.005\end{pmatrix}$, giving a long-run correlation $\rho\simeq0.29$.
\end{description}
Simulations were done using R functions provided by J-F Coeurjolly at \url{https://sites.google.com/site/homepagejfc/software}.

Figure~\ref{fig:boxplot_mfbm} represents the boxplots of CFW-PR(4,4) wavelet correlations at different scales in Case 1 and in Case 2. The good behavior of the approximation is observed except for the highest frequencies. Identical observations are obtained for CFW-C(4,4) filters (figure not provided).  

We now consider the local Whittle estimation of the parameters. Based on the discussion of Section~\ref{sec:arfima}, and on Figure \ref{fig:boxplot_mfbm}, we fix $j_0=4$. Table~\ref{tab:d_mfbm} and Table~\ref{tab:d_mfbm_C} highlight the good behavior of the estimation of long-memory parameters $\bd$, respectively for CFW-PR(4,4) and CFW-C(4,4) filters. Again, considering $j_0=4$ for both filters, the estimation procedures are equivalent for the two common-factor wavelets. Compared to the real wavelet-based estimation (with $j_0=2$ as suggested by \cite{multiwave}), the RMSE increases. This is mainly due to the choice of the hyperparameter $j_0$.

\begin{figure}[!ht]
\begin{subfigure}[t]{0.45\textwidth}
\caption{Case 1}
\includegraphics[width=\textwidth]{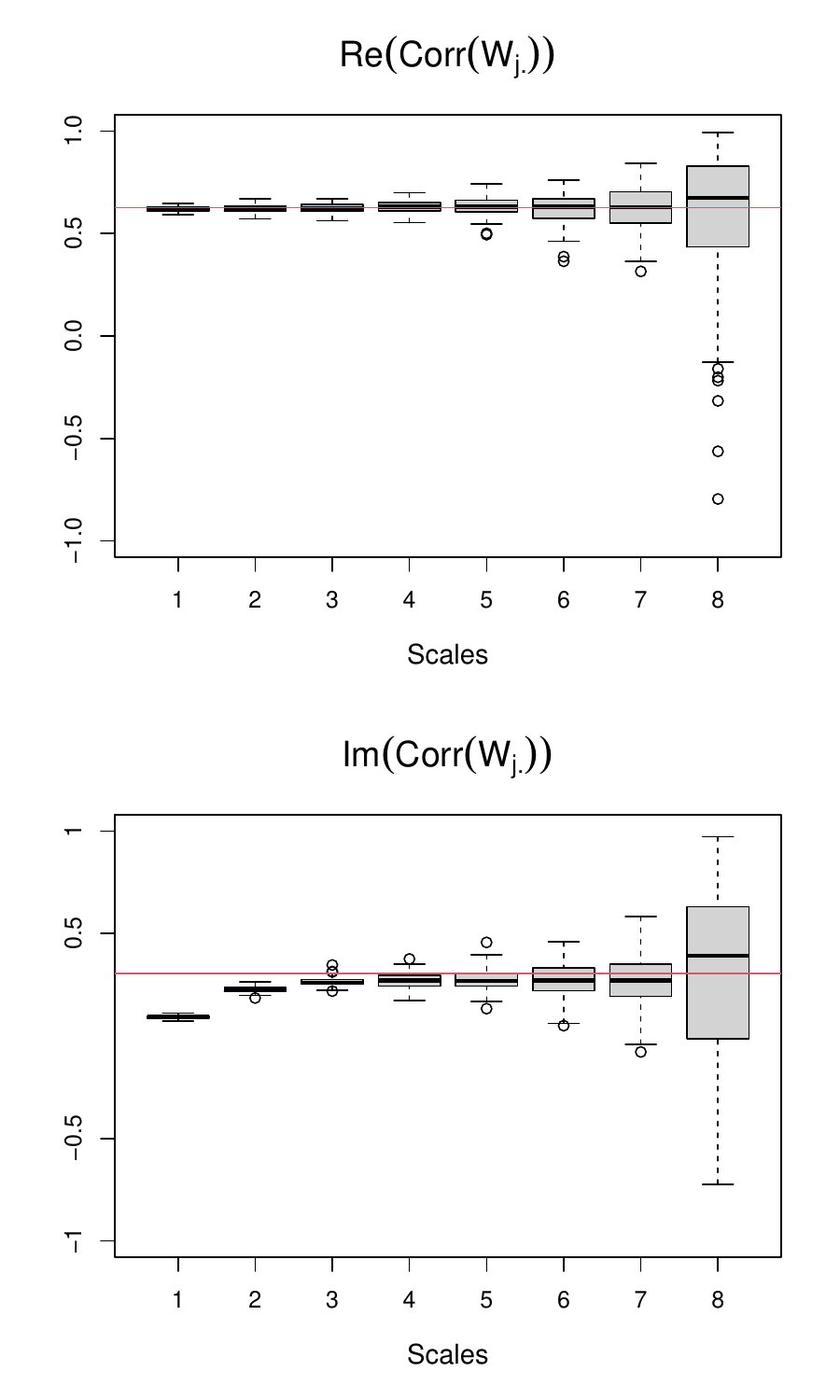}
\end{subfigure}
\begin{subfigure}[t]{0.45\textwidth}
\caption{Case 2}
\includegraphics[width=\textwidth]{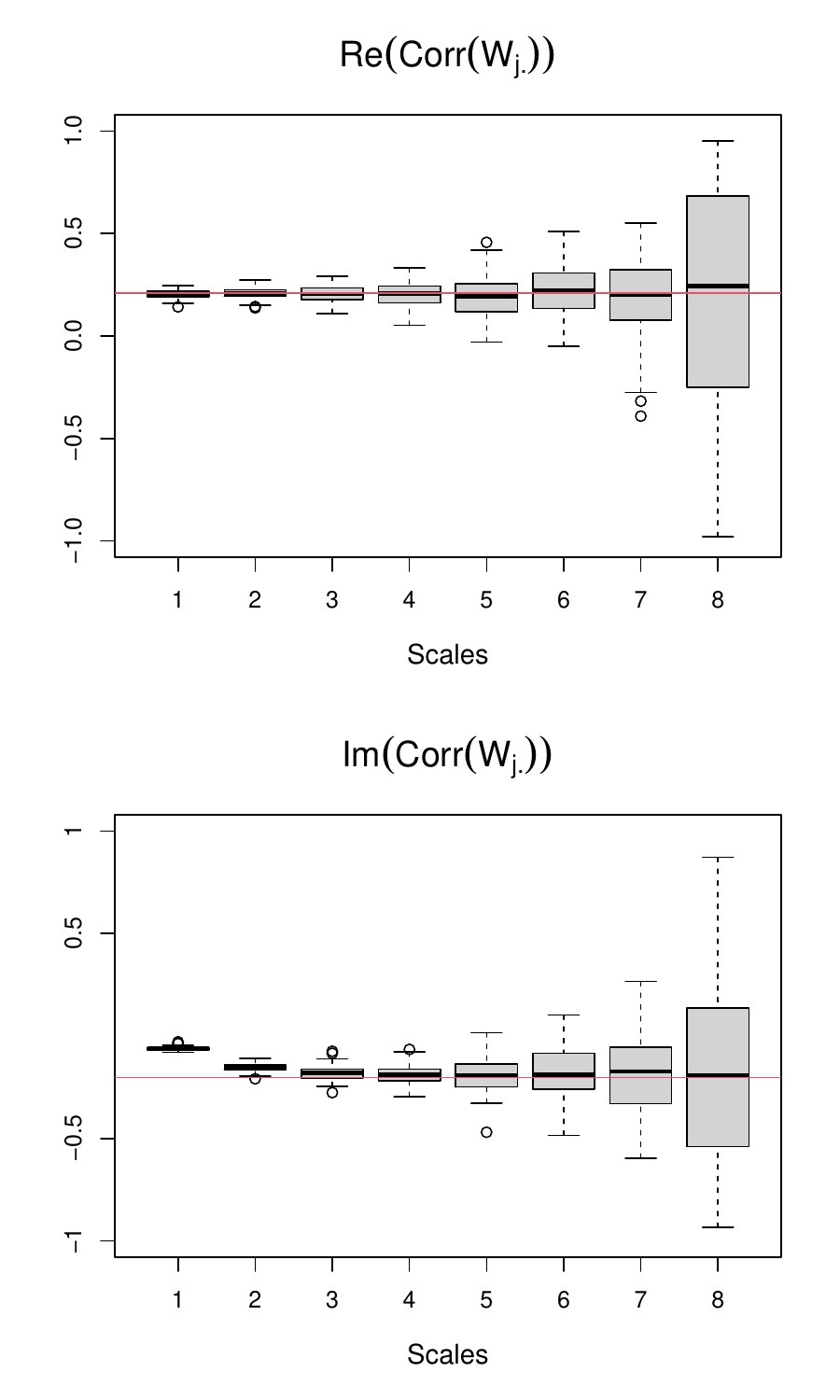}
\end{subfigure}
\caption{Boxplots of correlation between CFW-PR(4,4) coefficients at different scales for the simulated mFBM in Case 1 (left column--(a)) and in Case 2 (right column--(b)). First row gives the real part of the correlations and second row gives the imaginary part. Horizontal red lines correspond to the approximation given by Proposition~\ref{prop:covY3}, that is, $\rho\cos(\phi)r_K$ for the real part and  $\rho\sin(\phi)r_K$ for the imaginary part, with $r_K=K(d_1+d_2)/\sqrt{K(2\,d_1)K(2\,d_2)}$.}
\label{fig:boxplot_mfbm}
\end{figure}

\begin{table}[!ht]
\centering
\begin{tabular}{llcccc}
\\
\toprule
 &$\bd$ &  bias & std & RMSE & ratio PR/Real \\
 \midrule \midrule
Case 1 &  1  & -0.0065 & 0.0464 & 0.0469 & 1.5577 \\ 
& 1.2 &   -0.0059 & 0.0475 & 0.0478 & 2.0886 \\   \midrule
Case 2 & 1 & -0.0051 & 0.0510 & 0.0513 & 1.6812 \\ 
& 1.2 &   -0.0035 & 0.0515 & 0.0516 & 1.9884 \\ 
  \bottomrule
\end{tabular}
\caption{Results for the estimation of long-memory parameters $\bd$ with CFW-PR(4,4) filter on mFBMs. Hyperparameter $j_0$ satisfies $j_0=4$ for CFW-PR(4,4) and $j_0=2$ for real filters. PR/Real denotes the ratio between the RMSE given by CFW-PR(4,4) filter and the RMSE given by Daubechies' real filter.}
\label{tab:d_mfbm}
\end{table}

\begin{table}[!ht]
\centering
\begin{tabular}{llcccc}
\\
\toprule
& $\bd$ &  bias & std & RMSE & ratio C/PR \\
 \midrule \midrule
Case 1 & 1 & -0.0155 & 0.0409 & 0.0437 & 0.9323 \\ 
& 1.2 &   -0.0133 & 0.0402 & 0.0423 & 0.8849 \\ \midrule
Case 2 &   1 & -0.0177 & 0.0448 & 0.0482 & 0.9397 \\ 
& 1.2 &   -0.0116 & 0.0473 & 0.0487 & 0.9441 \\    \bottomrule
\end{tabular}
\caption{Results for the estimation of long-memory parameters $\bd$ with CFW-C(4,4) filter on mFBMs. Hyperparameter $j_0$ satisfies $j_0=4$. C/PR denotes the ratio between the RMSE given by CFW-C(4,4) filter and the RMSE given by CFW-PR(4,4) filter.}
\label{tab:d_mfbm_C}
\end{table}

Table~\ref{tab:G_mfbm} and Table~\ref{tab:G_mfbm_C} give the results obtained for the estimation of the covariance structure, that is, ${\bOmega}$, $\rho$ and $\phi$. It is not possible to compare our results with alternative non parametric procedures because real wavelet-based procedure estimates the real part of the long-run covariance or of the correlation, and Fourier-based estimations are not valid for non-stationary time series. 

The results of CFW-PR(4,4) and CFW-C(4,4) are similar. We observe a high bias and a high standard deviation for the estimation of ${\bOmega}$. On the other hand, we observe a good quality for the estimation of $\rho$ and of $\phi$.

\begin{table}[!ht]
\centering
\begin{tabular}{lccccc}
\\
\toprule
 &&   bias & std & RMSE  \\ 
 \midrule
Case 1 &  $\Omega_{1,1}$ & -0.1999 & 0.1544 & 0.2526\\ 
 	 & $\Omega_{1,2}$ &  -0.1592 & 0.0925 & 0.1841 \\ 
 	 & $\Omega_{2,2}$ &  -0.2471 & 0.1504 & 0.2892  \\ 
  	 & correlation    &    0.0971 & 0.0243 & 0.1001  \\ 
 	  & phase    &    0.0039 & 0.0526 & 0.0528  \\\midrule
Case 2 &$\Omega_{1,1}$ & -0.2048 & 0.1654 & 0.2633  \\ 
 	 & $\Omega_{1,2}$ &  -0.0647 & 0.0501 & 0.0818 \\ 
 	 & $\Omega_{2,2}$ &  -0.2508 & 0.1588 & 0.2969 \\ 
  	 & correlation    &    0.0967 & 0.0434 & 0.1059 \\ 
 	  & phase    &    -0.0087 & 0.1549 & 0.1551   \\ 
 \bottomrule
\end{tabular}
\caption{ Results for the estimation of matrices $\bTheta$ with CFW-PR(4,4) filter on mFBMs. Hyperparameter $j_0$ satisfies $j_0=4$.}
\label{tab:G_mfbm}
\end{table}

\begin{table}[!ht]
\centering
\begin{tabular}{lccccc}
\\
\toprule
& &   bias & std & RMSE & ratio C/PR \\ 
 \midrule
Case 1 &  $\Omega_{1,1}$ & -0.1548 & 0.1501 & 0.2156 & 0.8537 \\ 
 	 & $\Omega_{1,2}$ &  -0.1305 & 0.0865 & 0.1565 & 0.8501 \\ 
 	 & $\Omega_{2,2}$ &  -0.2088 & 0.1365 & 0.2495 & 0.8625 \\ 
 	  & correlation    &     0.0965 & 0.0245 & 0.0995 & 0.9940 \\ 
 	  & phase    &     0.0043 & 0.0493 & 0.0495 & 0.9376 \\ 
 \midrule
   Case 2  &  $\Omega_{1,1}$ & -0.1427 & 0.1632 & 0.2168 & 0.8235 \\ 
 	 & $\Omega_{1,2}$ &  -0.052 & 0.0527 & 0.074 & 0.9043 \\ 
 	 & $\Omega_{2,2}$ &  -0.2107 & 0.1596 & 0.2643 & 0.8905 \\ 
 	  & correlation    &     0.0936 & 0.045 & 0.1039 & 0.9805 \\ 
 	  & phase    &     -0.0061 & 0.156 & 0.1561 & 1.0064 \\ 
\bottomrule
\end{tabular}
\caption{ Results for the estimation of matrices $\bTheta$ with CFW-C(4,4) filter on mFBMs. Hyperparameter $j_0$ satisfies $j_0=4$. C/PR denotes the ratio between the RMSE given by CFW-C(4,4) filter and the RMSE given by CFW-PR(4,4) filter.}
\label{tab:G_mfbm_C}
\end{table}

To conclude, no major difference are observed between CFW-PR and CFW-C filters. As theoretical results are also available for CFW-C filters, it seems preferable to use them in practice.

\section{Application on a neuroscience dataset}

\label{sec:real}

We have applied our framework on fMRI data acquired on rats.  We consider functional Magnetic Resonance images (fMRI) of dead and live rats. Our aim is to estimate the brain connectivity, that is, the significant correlations between brain regions where fMRI signals are recorded. For this data set, we know that for dead rats the recordings are just noise, as no legitimate functional activity should be detected. Thus, the estimated graphs should be empty. We also expect non-empty graphs for live rats under anesthetic, as brain activity keeps on during anesthesia. The dataset is freely available at \href{https://zenodo.org/record/2452871}{https://10.5281/zenodo.2452871} \citep{becq_10.1088/1741-2552/ab9fec,guillaume2020functional}.

\subsection{Description of the dataset}

Functional Magnetic Resonance Images (fMRI) were acquired for dead and live rats (the full description is available in  \cite{guillaume2020functional}). 25 rats were scanned and identified in 4 different groups: DEAD, ETO\_L, ISO\_W and MED\_L. The first group contain dead rats and the three last groups correspond to different anesthetics. The duration of the scan was 30 minutes with a time repetition of 0.5 second so that $N=3600$ time points were available at the end of experience. After preprocessing as explained in \cite{guillaume2020functional}, $p=51$ time series for each rat were extracted. Each time series captures the functioning of a given region of the rat brain based on an anatomical atlas.

For each rat, we compute the estimators of 
\begin{itemize}[topsep=-10pt,itemsep=-5pt]
\item the vector of long-memory parameters, $\hat\bd$, 
\item the magnitude of the correlations, $\hat\brho=\{\hat\rho_{\ell,m},\, 1\leq \ell<m\leq p\}$ with $\hat\rho_{\ell,m}=\frac{{\hat\Omega_{\ell,m}}}{\sqrt{\hat\Omega_{\ell,\ell}\hat\Omega_{m,m}}}$,
\item the phases, $\hat\bphi=\{\hat\phi_{\ell,m},\, 1\leq \ell<m\leq p\}$.
\end{itemize}
Estimation was done with CFW-PR(4,4) filters. Densities of the estimators are represented on the figures using \texttt{R} default kernel-based estimation.

\subsection{Results and group comparisons}

Figure \ref{fig:fmri.d} shows the empirical distribution of the estimated empirical estimators $\hat \bd$. As expected, the long-memory parameters for dead rats are close to zero. The distributions are centered around zero, with a Gaussian-like shape. For rats under anesthetics, the densities are not centered around zero and the variance between brain regions is higher than what is observed for dead rats. Long-memories for rats under anesthetic ISO\_W are higher than under other anesthetics.

\begin{figure}[!ht]
\centering
{\includegraphics[width=\textwidth]{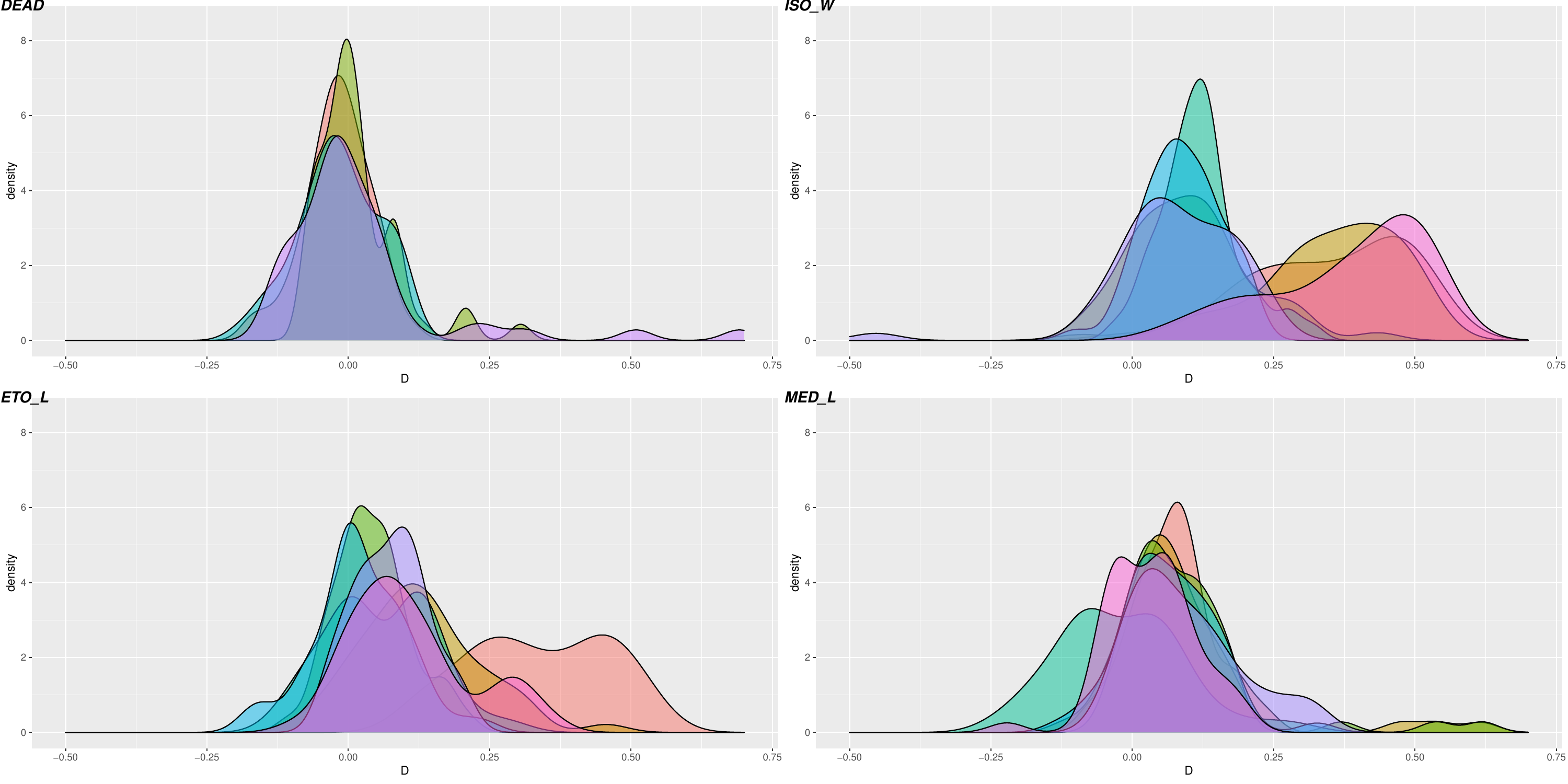}}
\caption{Plot of the empirical distribution of the long memory parameters $\hat\bd$ obtained for the 4 groups of rats. Each color corresponds to a rat.}
\label{fig:fmri.d}
\end{figure}

The distributions of the magnitudes and the phases of the estimated correlations, $\rho$ and $\phi$, for each rats, are shown respectively in Figure~\ref{fig:fmri.corr} and Figure~\ref{fig:fmri.phase}. First, as expected, the magnitudes obtained for the dead rats seem significantly different from those of the live rats. For dead rats, distributions have a small support, that is, only 9 on the 5100 values (0.18\%) satisfy $\hat \rho>0.3$. Note also that no major differences are observed between the rats. Next, ISO\_W and ETO\_L present quite similar distributions, with possibly high magnitudes. By contrast, the correlations for MED\_L anesthetic  are lower. These results tend to show that MED\_L anesthetic is more potent than the other anesthetics, leading to fewer connections between brain regions.

 First of all, as expected, the quantities obtained for the dead rats appear significantly different from those of the living rats. For dead rats, the distributions have a small support, that is, only 9 values out of 5100 (0.18\%) satisfy $\hat \rho>0.3$. Also note that no major differences are observed between the rats. Then, ISO\_W and ETO\_L show quite similar distributions, with possibly high magnitudes. On the other hand, the correlations for the anesthetic MED\_L are weaker. These results tend to show that the anesthetic MED\_L is more potent than other anesthetics, resulting in fewer connections between brain regions.

The phase parameter can be interpreted as an asymmetry of the coupling at large lags among the components of the signals for each brain region (a null phase is equivalent to time-reversibility). The distributions displayed in Figure~\ref{fig:fmri.phase} correspond to the empirical densities of the upper triangular matrices of phases, $\{\phi_{\ell,m},\,1\leq\ell<m\leq p\}$. This explains why the distributions are not symmetric.

\begin{figure}[!ht]
\centering
{\includegraphics[width=\textwidth]{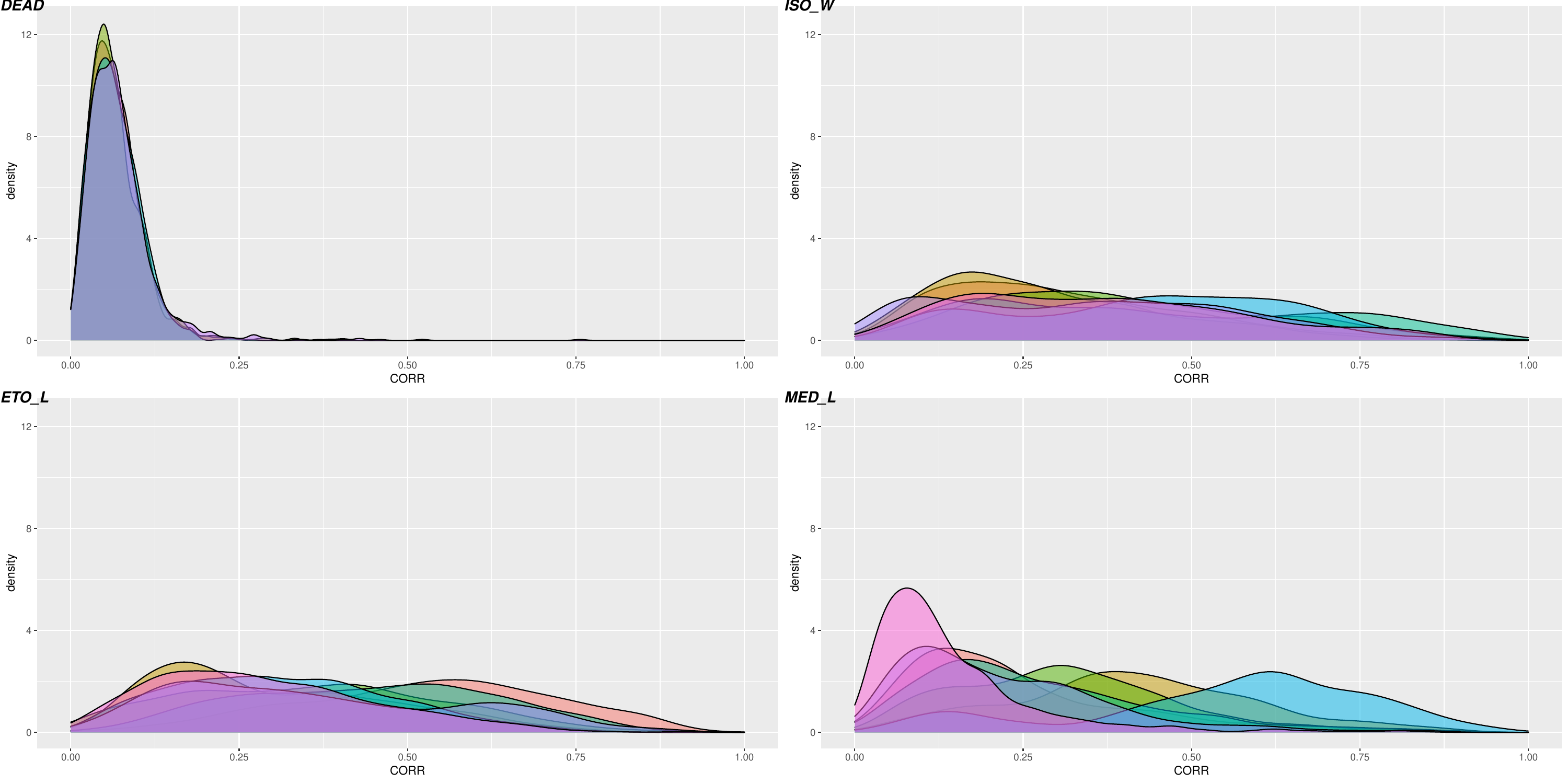}}
\caption{Plot of the empirical distribution of the correlation magnitudes $\hat\brho$ obtained for the 4 groups of rats. Each color corresponds to a rat.}
\label{fig:fmri.corr}
\end{figure}

For dead rats, we observe mainly uniform distributions. For live rats,  Figure~\ref{fig:fmri.phase} shows that the distributions have heavy tails. The tails are heavier for MED\_L anesthetic than for other anesthetics. This can be explained by the fact that the phase is non-informative when the magnitude is close to zero. As indicated previously, this problem of identifiability occurs for example in the case of fractional co-integration. Following \cite{Baek2020} another parametrization could be proposed to overcome it. The parametrization chosen here, nevertheless, seems more appropriate since the magnitude is crucial in this real data application.

To illustrate this fact, Figure~\ref{fig:fmri.corr_phase} shows the distributions of the estimated phases $\phi$ corresponding to magnitudes satisfying $\rho>0.3$ (this choice is motivated by the observation on the support of dead rats' correlations above). The distributions then have smaller tails. It can be observed that the supports of the phases are larger for live rats than for dead rats. Next, the 95\%-quantiles of absolute values (\emph{i.e.} $q$ such that 95\% of absolute values of phases are lower than $q$) are respectively 2.95, 1.90, 1.89, 1.61 for dead rats, ISO\_W, ETO\_L and MED\_L. It seems that ISO\_W has a higher support, meaning that shifts appear in the connections between brain regions, with respect to other anesthetics. Yet, we have not tested whether the difference is significant.

\begin{figure}[p]
\centering
{\includegraphics[width=\textwidth]{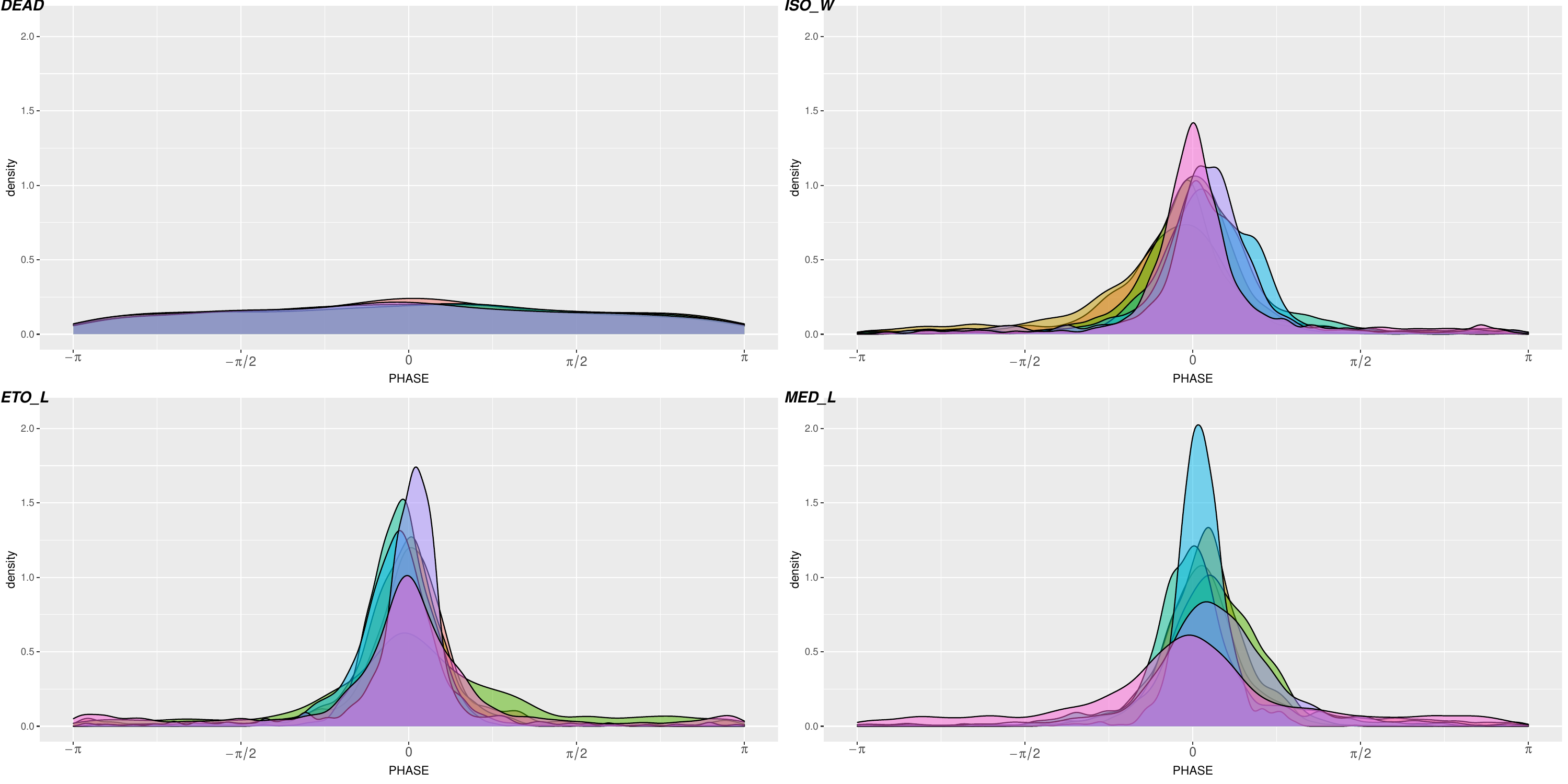}}
\caption{Plot of the empirical distribution of the phases $\hat\bphi$ obtained for the 4 groups of rats without thresholding the correlations.  Each color corresponds to a rat.}
\label{fig:fmri.phase}

{\includegraphics[width=\textwidth]{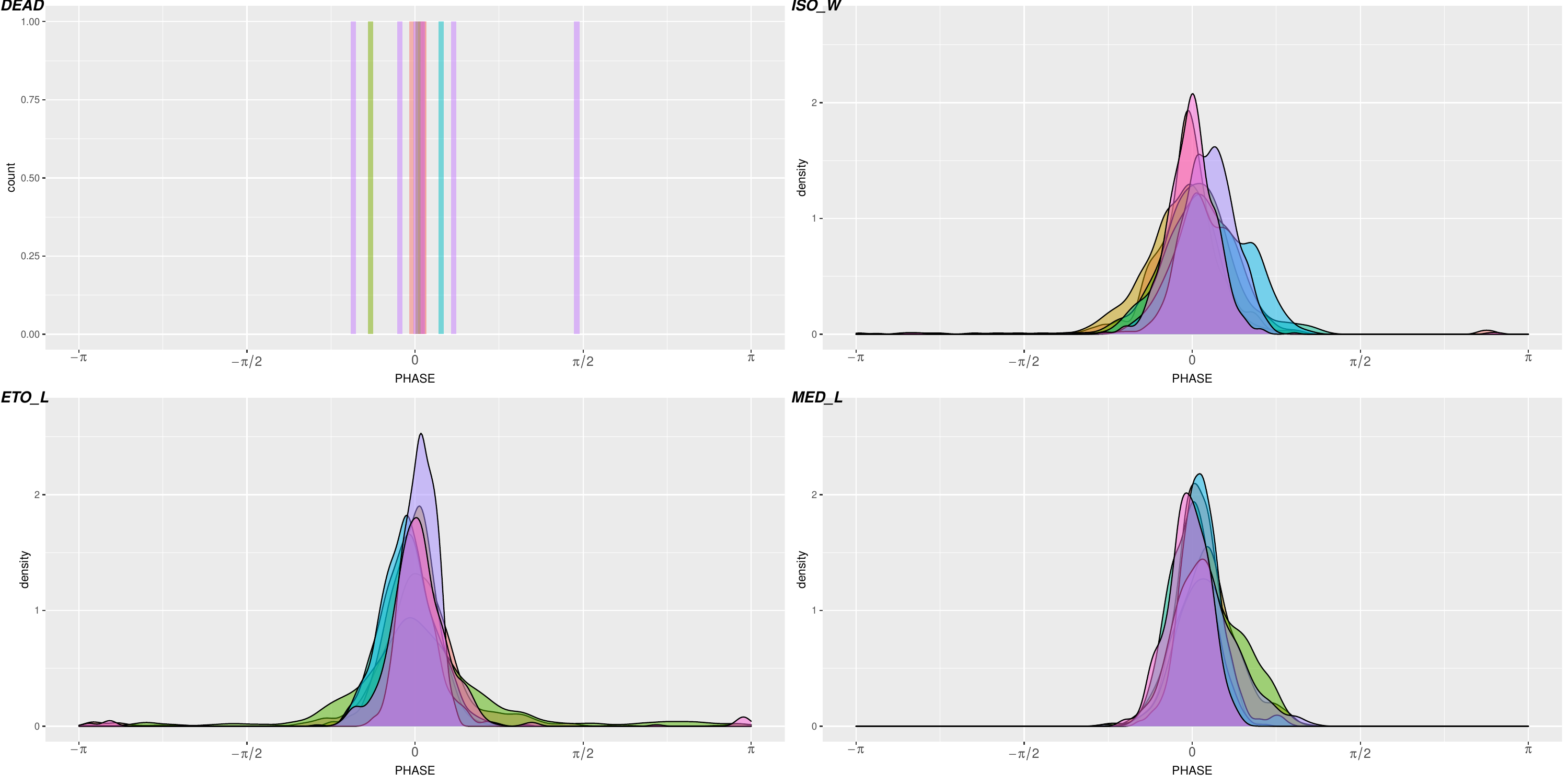}}
\caption{Plot of the empirical distribution of the phases obtained for the 4 groups of rats after first thresholding the correlations. Only phases associated to correlations with a magnitude higher than 0.3 are considered.  Each color corresponds to a rat. For dead rats, a bar plot is provided rather than a density plot due to the low number of values.}
\label{fig:fmri.corr_phase}
\end{figure}

\subsection{Graphs with correlations and phases}

We first compute the adjacency matrix obtained for each rat within each group. Edges correspond to a magnitude higher than 0.3. The value of the threshold is motivated by the observation of the supports obtained for dead rats. We then select the edges which are present in all the graphs of the rats of the group. One graph is then obtained per group. For each group, we then compute the mean of the estimated phase for each detected edge. Figure \ref{fig:fmri.graphs2} illustrates the graphs obtained for the 4 different groups.

We have colored each edge according to the mean phase when it satisfies $|\phi_{\ell,m}|>1.1\lvert\phi_{\ell,m}^*\rvert$ where $\phi_{\ell,m}^*=-\frac{\pi}{2}(d_\ell-d_m)$, $({\ell,m})\in\{1, \dots, p\}^2$. The value $\phi_{\ell,m}^*$ corresponds to the phase of causal linear  representations with power-law coefficients \citep{KechagiasPipiras} and to the ARFIMA modeling used in \cite{AchardGannaz} with similar data. The more the edges are colored, the more the behavior of the phase differs from the preceding modeling.

The DEAD group has indeed no edges. The MED\_L group has fewer edges than the two other groups of anesthetic. It hence seems that MED\_L anesthetic inhibits more the activity. Next ETO\_L group and ISO\_W group have a similar number of edges (respectively 133 and 145), but the phases differ. More than half of the mean phases are outside the interval $[-1.1\lvert\phi^*\rvert, 1.1\lvert\phi^*\rvert]$ for ETO\_L and ISO\_W groups, with similar proportions. This observation is interesting because it illustrates that the modeling of these data is complex. The introduction of a general phase enables to take this complexity into account. Concerning the physical interpretation, no easy conclusion can be given. As it was mentioned in \cite{buxton2013physics}, the time scale of BOLD (Blood oxygenation level dependent) response is very small in comparison with the neuronal activity. The observed delay is equal to a few seconds. Considering the different time scales involved in the production of the BOLD response, we may hypothesize that lags are not the underlying phenomenon that produces phase differences in fMRI signals. However, as stated in \cite{buxton2013physics}, the time scale can vary in the same subject depending on the physiological baseline state, which is known to be modified under anesthesia.

\begin{figure}[p]
  \centering
    \includegraphics[width=0.9\textwidth]{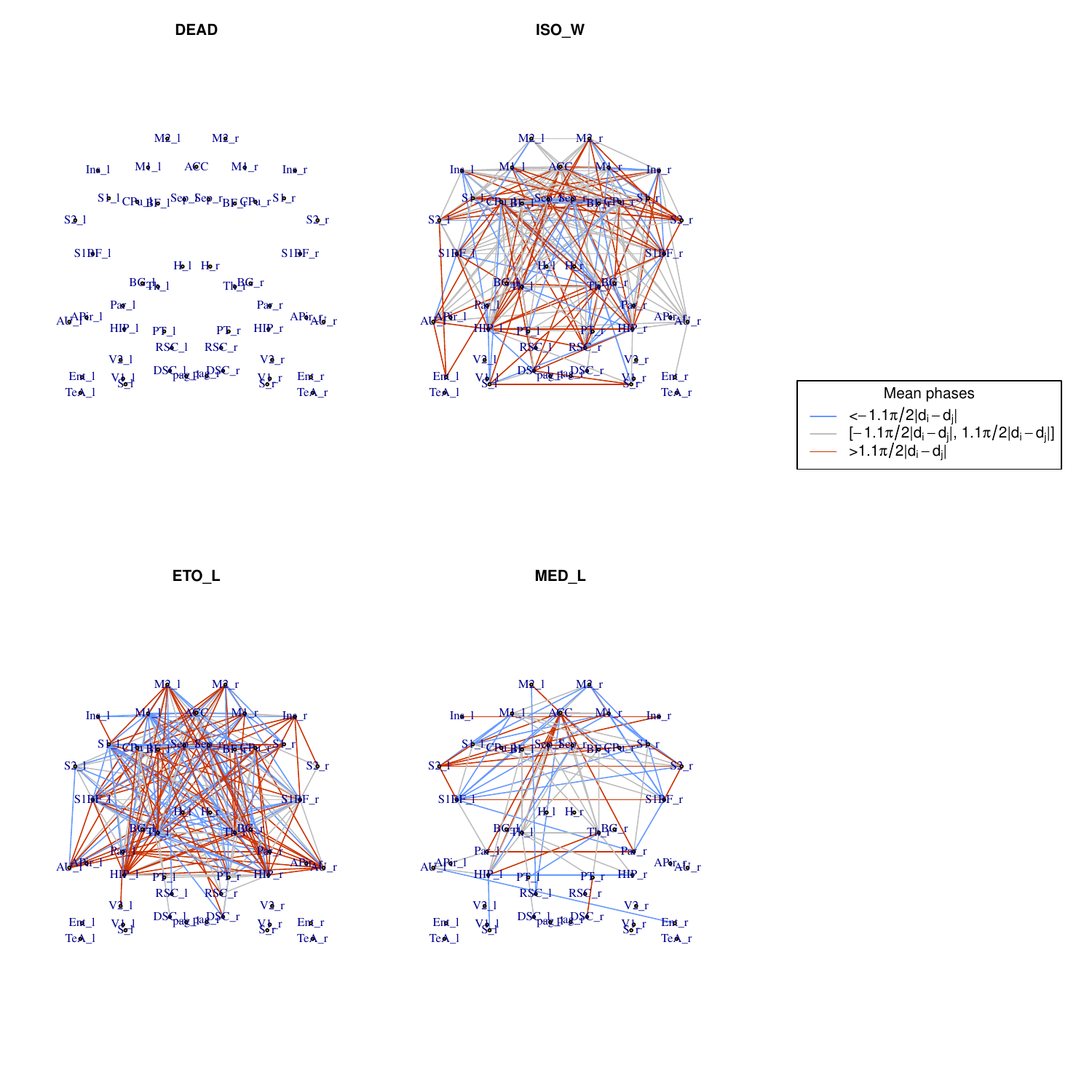} 
\caption{Plot of the average graphs with correlations and phases obtained for 4 groups of rats: DEAD, ISO\_W, ETO\_L and MED\_L. Only edges corresponding to a mean correlation's magnitude higher than 0.3 are displayed. Red edges correspond to positive mean phases higher than 1.1$\lvert\phi^*\rvert$, blue edges correspond to negative mean phases lower than -1.1$\lvert\phi^*\rvert$, and grey edges to mean phases between -1.1$\lvert\phi^*\rvert$ and 1.1$\lvert\phi^*\rvert$. The quantities $\phi^*$ are equal to $\phi_{\ell,m}^*=-\frac{\pi}{2}(d_\ell-d_m)$.} 
\label{fig:fmri.graphs2}
\end{figure}

\section{Conclusion}

This work was motivated by an application in neuroscience, namely the inference of fractal connectivity from fMRI recordings. We have studied the local Whittle estimators for multivariate time series presenting long-memory. Our modeling allows for a complex covariance structure with phase components that can be interpreted as shifts in the coupling between time series. We have introduced quasi-analytic wavelet filters to handle the possible non-stationarity in the real data application. The resulting procedures offer a consistent estimation of the main parameters of the model. Indeed, we have established that so called \emph{Common-Factor wavelets} are an efficient tool for recovering the long-memory structure as well as the covariance structure, including magnitude and phase. A simulation study on linear processes and on multivariate Brownian motions illustrates the good performance of the proposed procedure. The real data application highlights the ability of the procedure to distinguish dead rats from live rats. We also show the differences between three anesthetics and the fact that one of them slows down brain activity more intensively.

\subsection*{Ackowledgments} The authors are indebted to Marianne Clausel and Fran\c ois Roueff who contributed substantially to this work. We are grateful to the reviewers and to the co-editor who help to significantly improve the paper. We also thank Jean-Fran\c cois Coeurjolly for providing the code for the simulations of multivariate fractional Brownian motions. We also would like to thank Emmanuel Barbier and Guillaume Becq for providing us the data of the resting state fMRI on the rats.

\appendix

\section{Expression of the CFW-C(M,L) filter}
This section aims at giving the expression of the CFW-C(M,L) filters which are used to compute the wavelet coefficients. Let first recalls the expression of common-factor wavelets.

\subsection{Expression of the CFW-C(M,L) pair}

Let us recall the expression of the low-pass filter $\widehat h^{(L)}(.)$ and the high-pass filter $\widehat{h}^{(H)}(.)$:
\begin{equation} \label{eqn:h}
\widehat{h}^{(L)}(\lambda)= \sqrt{2}\,
\left(\frac{1+e^{-i\lambda}}{2}\right)^M\hat q_{L,M}(\lambda)\,\widehat{d}_L(\lambda)
\text{   ~~and~~   }
\widehat{h}^{(H)}(\lambda)=\overline{\widehat{h}^{(L)}(\lambda+\pi)}e^{-i\lambda}\,,
\end{equation}
for all $\lambda\in\R$.
All the same, for all $\lambda\in\R$,
\begin{equation} \label{eqn:g}
\widehat{g}^{(L)}(\lambda)= \sqrt{2}\,
\left(\frac{1+e^{-i\lambda}}{2}\right)^M\hat q_{L,M}(\lambda)\,\overline{\widehat{d}_L(\lambda)} e^{-i\lambda L}
\text{   ~~and~~   }
\widehat{g}^{(H)}(\lambda)=\overline{\widehat{g}^{(L)}(\lambda+\pi)}e^{-i\lambda}.
\end{equation}
Let us now explicit $\hat\varphi(\cdot)$. We have $\hat\varphi(\cdot)=\hat\varphi_h(\cdot)+\rmi\hat\varphi_g(\cdot)$ with
\begin{align*}  
\hat \varphi_h(\lambda)
&=  2^{-1/2}\prod_{j=1}^\infty 2^{-1/2} \hat h^{(L)}(2^{-j}\lambda)=2^{-1/2}\prod_{j=1}^\infty \bigl(\frac{1+\rme^{-\rmi 2^{-j}\lambda}}{2}\bigr)^M{\hat q_{L,M}(2^{-j}\lambda)\hat d_L(2^{-j}\lambda)},\\
\text{and~~}
\hat \varphi_g(\lambda)
&= 2^{-1/2}\prod_{j=1}^\infty \bigl(\frac{1+\rme^{-\rmi2^{-j}\lambda}}{2}\bigr)^M\hat q_{L,M}(2^{-j}\lambda)\overline{\hat d_L(2^{-j}\lambda)}\rme^{-\rmi 2^{-j}\lambda L}.
\end{align*}
When $\hat q(.)=1$, the expressions above become
\begin{align}
\label{eqn:phih1}\hat \varphi_h(\lambda)&=\prod_{j=1}^\infty \bigl(\frac{1+\rme^{-\rmi2^{-j}\lambda}}{2}\bigr)^M\overline{\hat d_L(2^{-j}\lambda)},\\
\label{eqn:phig1}\hat \varphi_g(t)&=\prod_{j=1}^\infty \bigl(\frac{1+\rme^{-\rmi2^{-j}\lambda}}{2}\bigr)^M\overline{\hat d_L(2^{-j}\lambda)}\rme^{-\rmi 2^{-j}\lambda L}.
\end{align}

Next, we can explicit $\hat\psi(\cdot)$. We have $\hat\psi(\cdot)=\hat\psi_h(\cdot)+\rmi\hat\psi_g(\cdot)$ with
\begin{align*}
\lefteqn{\hat \psi_h(\lambda)}\\
&= 2^{-3/2}\hat h^{(H)}(\lambda/2)\prod_{j=2}^\infty 2^{-1/2} \hat h^{(L)}(2^{-j}\lambda)\\
&= 2^{-1}\Bigl(\frac{1-\rme^{\rmi\lambda/2}}{2}\Bigr)^M\,\overline{\hat q_{L,M}(\lambda/2+\pi)}\overline{\hat d_L(\lambda/2+\pi)}\rme^{-\rmi\lambda/2}{\prod_{j=2}^\infty \bigl(\frac{1+\rme^{-\rmi 2^{-j}\lambda}}{2}\bigr)^M{\hat q_{L,M}(2^{-j}\lambda)\hat d_L(2^{-j}\lambda)},}
\end{align*}
and
\begin{multline*}
\hat \psi_g(\lambda)
=- 2^{-1/2}\Bigl(\frac{1-\rme^{-\rmi\lambda/2}}{2}\Bigr)^M\,\overline{\hat q_{L,M}(\lambda/2+\pi)}\,\hat d_L(\lambda/2+\pi)\rme^{+\rmi(\lambda/2 +\pi)L}\rme^{-\rmi\lambda/2}\\
{\prod_{j=2}^\infty \bigl(\frac{1+\rme^{-\rmi2^{-j}\lambda}}{2}\bigr)^M\hat q_{L,M}(2^{-j}\lambda)\overline{\hat d_L(2^{-j}\lambda)}\rme^{-\rmi 2^{-j}\lambda L}.}
\end{multline*}
When $\hat q(.)=1$, the expressions above become
\begin{align*}
\hat \psi_h(\lambda)&= 2^{-1/2}\Bigl(\frac{1-\rme^{-\rmi\lambda/2}}{2}\Bigr)^M\overline{\hat d_L(\lambda/2+\pi)}\rme^{-\rmi\lambda/2}\prod_{j=2}^\infty \Bigl(\frac{1+\rme^{-\rmi2^{-j}\lambda}}{2}\Bigr)^M{\hat d_L(2^{-j}\lambda)}\\
\hat \psi_g(\lambda)&= -2^{-1/2}\Bigl(\frac{1-\rme^{-\rmi\lambda/2}}{2}\Bigr)^M\hat d_L(\lambda/2+\pi)\rme^{+\rmi(\lambda/2 +\pi)L}\rme^{-\rmi\lambda/2}\prod_{j=2}^\infty \Bigl(\frac{1+\rme^{-\rmi2^{-j}\lambda}}{2}\Bigr)^M\overline{\hat d_L(2^{-j}\lambda)}\rme^{-\rmi 2^{-j}\lambda L}.
\end{align*}

We shall use the following equality 
\begin{equation*}
\prod_{\ell=1}^\infty\left(\frac{1+\rme^{-\rmi2^{-\ell}\lambda}}{2}\right)=\frac{1-e^{-\rmi\lambda}}{-\rmi\lambda}=\rme^{-\rmi\lambda/2}\frac{\sin(\lambda/2)}{\lambda/2}.
\end{equation*}
See \emph{e.g.} \cite[page 245]{Mallat99}.
It yields 
\begin{align}
\label{eqn:phih}\hat \varphi_h(\lambda)&= 2^{-1/2}\Bigl(\frac{\sin(\lambda/2)}{\lambda/2}\Bigr)^M\prod_{j=1}^\infty \overline{\hat d_L(2^{-j}\lambda)},\\
\label{eqn:phig}\hat \varphi_g(t)&= 2^{-1/2}\Bigl(\frac{\sin(\lambda/2)}{\lambda/2}\Bigr)^M\prod_{j=1}^\infty \overline{\hat d_L(2^{-j}\lambda)}\rme^{-\rmi 2^{-j}\lambda L},
\end{align}
and 
\begin{align}
\label{eqn:psih}\hat \psi_h(\lambda)&= 2^{-1/2}\sin(\lambda/4)^M \Bigl(\frac{\sin(\lambda/4)}{\lambda/4}\Bigr)^M\overline{\hat d_L(\lambda/2+\pi)}\rme^{-\rmi\lambda}\prod_{j=2}^\infty {\hat d_L(2^{-j}\lambda)},\\
\label{eqn:psig}\hat \psi_g(\lambda)&= 2^{-1/2}\sin(\lambda/4)^M \Bigl(\frac{\sin(\lambda/4)}{\lambda/4}\Bigr)^M \hat d_L(\lambda/2+\pi)\rme^{+\rmi(\lambda +\pi)L}\rme^{-\rmi\lambda}\prod_{j=2}^\infty \overline{\hat d_L(2^{-j}\lambda)}\rme^{-\rmi 2^{-j}\lambda L}.
\end{align}

\subsection{A first property}

The function $\widehat{d}_L(.)$ satisfies, for all $\lambda\in\R$,
\begin{equation} \label{eqn:d}
\widehat{d}_L(\lambda)=\rme^{\rmi\,\lambda(-L/2+1/4)}\left[\cos(\lambda/4)^{2L+1}+\rmi\,(-1)^{L+1} \sin(\lambda/4)^{2L+1}\right].
\end{equation}
We deduce the following lemma.
\begin{lemma}
\label{lem:supd}
\[
\sup_{\lambda\in\R}\abs{\hat d_L(\lambda)}=1.
\]
\end{lemma}
The proof is straightforward and it is thus omitted. 

Using the fact that, for all $\lambda\in\R$, $\abs{\frac{\sin(\lambda)}{\lambda}}\leq 1$, a direct consequence is the following result.
\begin{lemma}
\label{lem:fini}
\begin{equation*}
\sup_{\lambda\in\R}\abs{\hat\varphi(\lambda)}\leq C_\infty,
\end{equation*}
with $C_\infty=1$.
\end{lemma}

\subsection{Expression of the wavelet filters}
\label{sec:tauj}

Recall that
\[
W_{j,k}(\ell)=\int_{-\infty}^\infty \sum_{s\in\Z} X_\ell(s)\varphi(t+s)\psi_{j,k}(t)\rmd t. 
\]
Hence,
\begin{gather*}
W_{j,k}(\ell)=\sum_{s\in\Z} \tau_{j}(2^jk-s)X_\ell(s), \quad j\geq 0, k\in\Z\,,\\
\text{with  }\hat\tau_j(\lambda)=\sum_{s\in\Z} \tau_{j}(2^jk-s)\rme^{\rmi\lambda\,s}=\int_{-\infty}^\infty \sum_{\ell\in\Z}\varphi(t+\ell)\,\rme^{-\rmi\,\lambda\,\ell}2^{-j/2} \psi(2^{-j}t)\rmd t.
\end{gather*}

\section[Properties of CFW-C(M,L) filters]{Properties of CFW-C(M,L) filters}

Let us introduce the following properties.
\begin{enumerate}[label=(W\arabic*)]
\item\label{ass:finite}
 \emph{Finite support}. $\varphi$ and $\psi$ have finite support.
  \item\label{ass:moments}
   \emph{Vanishing moments}. There exist $M\geq 0$ and $C_m>0$ such that for all $j\geq0$ and
  $\lambda\in \R$,
    \begin{equation*}
    \abs{\widehat{\psi}(\lambda)}\leq {C_m\,\abs{\lambda}^M}\;,
    \end{equation*}
with $C_m$ positive constant possibly depending on $M$.
\item\label{ass:unif-smooth}
 \emph{Uniform smoothness}. There exist  $\alpha>1$ and $C_s>0$ such that for all  $\lambda\in \R$,
    \begin{equation*}
    \abs{\widehat \psi(\lambda)}\leq \frac{C_s}{(1+\abs{\lambda})^{\alpha}},
    \end{equation*}
with $\alpha$ and $C_s$ depending on $L$ and $M$.
\item\label{ass:phi_unif} 
\emph{Scaling function}.
There exist $C_\varphi$ depending on $M$ such that, for all $\lambda\in(-\pi,\pi)$, for all $k\in\Z\setminus\{0\}$,
\[
\abs{\hat\varphi(\lambda+2k\,\pi)}\leq C_\varphi\,\abs{\lambda}^{M}\,.
\]
\end{enumerate}

Properties \ref{ass:finite}, \ref{ass:moments}, \ref{ass:unif-smooth} and \ref{ass:phi_unif} correspond respectively to (W1), (W2), (W3) and (W4) of \cite{Moulines07SpectralDensity,Moulines08Whittle, AchardGannaz} in the  context of real-wavelets. We can establish that they are satisfied by CFW-C(L,M) wavelets. Assumption \ref{ass:finite} is given in Lemma \ref{lem:nj}. The wavelets supports are finite when $L$ and $M$ are finite.

\begin{proposition}\label{prop:filters}
When $M\geq 2$, and $L\geq 1$, CFW-C(M,L) wavelets satisfy 
 \ref{ass:moments}, \ref{ass:unif-smooth} and \ref{ass:phi_unif}, with $\alpha=M$ 
 and constants $C_m=1$, $C_{s}=2\cdot 5^{M}$ and $C_\varphi=2$.
\end{proposition}
The proof is given in Section~\ref{sec:proof_W}.

A remarkable property of $CFW-C(M,L)$ filters is that the regularity of the wavelets is only determined by the parameter $M$, since all parameters and constants in the proposition above only depend on $M$. All the same, the quasi-analyticity only depends on the parameter $L$, through Theorem~\ref{th:main-result-analyticity}.

With these assumptions, we can establish some properties about wavelet filters. 
At a given scale $j\geq 0$, for any $k\in\Z$, wavelet coefficients $W_{j,k}(\ell)$ of a process $X_\ell(\cdot)$ can be decomposed as
\[
W_{j,k}(\ell)=\sum_{s\in\Z} \tau_{j}(2^jk-s)X_\ell(s), \quad j\geq 0, k\in\Z\,.
\]
See Section \ref{sec:tauj}.
We recover \cite[Proposition 3]{Moulines07SpectralDensity}. More precisely, we can establish the following results.
\begin{proposition}\label{prop:ineq}
Suppose \ref{ass:finite}, \ref{ass:moments}, \ref{ass:unif-smooth}, and  \ref{ass:phi_unif} and Lemma \ref{lem:fini} hold. Then, for all $j\geq 0$, for all $\lambda\in\R$, 
\begin{align}
\label{eqn:tauj}
\abs{\hat \tau_j(\lambda)-2^{j/2}\hat \varphi(\lambda)\overline{\hat \psi(2^{j}\lambda)}}&\leq C_\tau 2^{j(1/2-\alpha)}\lvert\lambda\rvert^M,\\
\label{eqn:moments_phipsi}
\abs{\hat \varphi(\lambda)\overline{\hat \psi(2^{j}\lambda)}}&\leq C_\infty C_m\abs{2^j\lambda}^M,\\
\label{eqn:smoothness_phipsi}
\abs{\hat \varphi(\lambda)\overline{\hat \psi(2^{j}\lambda)}}&\leq \frac{C_\infty C_s}{(1+2^j\abs{\lambda})^{\alpha}},\\
\label{eqn:moments_tauj}
\abs{2^{-j/2}\hat\tau_j(2^{-j}\lambda)}&\leq C_{m\tau}\,\abs{\lambda}^{M},
\end{align}
and, for all $j,j'\geq 0$,  for all $|2^{-j}\lambda|\leq \pi$, 
\begin{align}
\label{eqn:smoothness_tauj}
\abs{2^{-j/2}\hat\tau_j(2^{-j}\lambda)}& \leq C_{s\tau}\,(1+\abs{\lambda})^{-\alpha},\\
\label{eqn:moments_smooth_tauj}
\abs{2^{-j/2}\hat\tau_j(2^{-j}\lambda)}&\leq C_{ms\tau}\,\abs{\lambda}^{M}(1+\abs{\lambda})^{-\alpha-M},
\end{align}
with $C_\tau$, $C_{m\tau}$,
 $C_{s\tau}$, and $C_{ms\tau}$ positive constants only depending on $\alpha, M$, $C_m$, $C_s$ and $C_\varphi$.
\end{proposition}
The proof is given in Section~\ref{sec:proof_ineq}.

The following property corresponds to (79) in \cite[Proposition 3]{Moulines07SpectralDensity}. In the real wavelets context, it is a consequence of (W1) to (W4) but we prove it separately here to explicit the constants.

\begin{proposition}\label{prop:ass_infini}
Consider CFW-C(M,L) wavelets, {with $M\geq 2$}. 
For all $j\geq1$,
\begin{align}\label{ass::infini}
\sup_{\abs{\lambda}\leq \pi}\abs{\abs{2^{-j/2}\hat\tau_j(2^{-j}\lambda)}^2-
\abs{\hat\psi(\lambda)}^2}&\leq C_a\,2^{-\gamma\,j}\abs{\lambda}^{2M}\,,\\
\label{ass::infini2}
\sup_{1\leq\abs{\lambda}\leq 2^j\pi}\abs{\abs{2^{-j/2}\hat\tau_j(2^{-j}\lambda)}^2-
\abs{\hat\psi(\lambda)}^2}&\leq C_a\,2^{-\gamma\,j}\abs{\lambda}^{2}(1+\abs{\lambda})^{-\alpha}\,,
\end{align}
with $\gamma=2$ and $C_a=2\cdot 5^{2M}\cdot (M+L+1)$.
\end{proposition}
The proof is given in Section~\ref{sec:proof_reg}.

\subsection{Proof of Proposition~\ref{prop:filters}}

\label{sec:proof_W}

We first establish that CFW-C(M,L) filters satisfy properties {(W2)}--(W4).

\subsubsection{Property~\ref{ass:moments}}

Let $\lambda\in\R$. Recall that $\hat \psi_h(\lambda)$ and $\hat \psi_g(\lambda)$ are given respectively by \eqref{eqn:psih} and \eqref{eqn:psig}.
Observe that $\abs{\sin(\lambda/4)}^M\leq 4^{-M}\abs{\lambda}^M$. Since $\sup |\hat d_L|=1$, assumption~\ref{ass:moments} follows with a constant $C_m=1$, for all $M\geq 1$.

\subsubsection{Property \ref{ass:unif-smooth}}

\label{sec:unif-smooth}

Let $\lambda\in\R$. Recall that $\hat \psi_h(\lambda)$ is given by \eqref{eqn:psih}.
Since $\sup_{\lambda\in\R}\abs{\hat d_L(\lambda)}\leq 1$, we obtain
\[
\abs{\hat \psi_h(\lambda)}\leq\abs{\sin(\lambda/4)}^M\abs{\frac{\sin(\lambda/4)}{\lambda/4}}^M.
\]
Since $\frac{\sin(x/4)}{|x/4|}(1+|x|)=\abs{\frac{\sin(x/4)}{x/4}}+4\,\abs{\sin(x/4)}\leq 5$ for any $x\in\R\setminus\{0\}$,
it follows that
\begin{equation*}
\abs{\hat \psi_h(\lambda)}\leq\abs{\sin(\lambda/4)}^M\left(\frac{5}{1+\abs{\lambda}}\right)^{M}.
\end{equation*}
Consequently, $\abs{\hat \psi_h(\lambda)}\leq\left(\frac{5}{1+\abs{\lambda}}\right)^{M}.$
A similar result can be proved for filter $\hat \psi_g$. By triangular inequality, we get
\begin{equation*}
\abs{\hat\psi(\lambda)}\leq \frac{C_{s}}{(1+\abs{\lambda})^\alpha},
\end{equation*}
with $\alpha=M$ and a constant $C_{s}$ equal to $2\cdot 5^{M}$.

\subsubsection{Property \ref{ass:phi_unif}}

Let $\lambda=\omega+2k\pi$, $\lvert\omega\rvert\leq\pi$, $k\in\Z$, $k\neq 0$.
We distinguish two cases:
\begin{itemize}
\item If $k$ is odd,\\
Using the expressions of $\hat \varphi_h(\lambda)$ and $\hat \varphi_g(\lambda)$ given respectively by \eqref{eqn:phih1} and \eqref{eqn:phig1}, 
\begin{align*}
\hat \varphi_h(\lambda)&=\bigl({\cos(\lambda/4)}\bigr)^M\overline{\hat d_L(\lambda/2)}\prod_{j=2}^\infty \bigl(\frac{1+\rme^{-\rmi2^{-j}\lambda}}{2}\bigr)^M\overline{\hat d_L(2^{-j}\lambda)},\\
\hat \varphi_g(\lambda)&=\bigl({\cos(\lambda/4)}\bigr)^M\overline{\hat d_L(\lambda/2)}\rme^{-\rmi \lambda L/2}\prod_{j=2}^\infty \bigl(\frac{1+\rme^{-\rmi2^{-j}\lambda}}{2}\bigr)^M\overline{\hat d_L(2^{-j}\lambda)}\rme^{-\rmi 2^{-j}\lambda L}.
\end{align*}
Observe that
\[
\bigl\lvert{\cos(\lambda/4)}\bigr\rvert^M=\lvert\cos(\omega/4+k\pi/2)\rvert^M=\lvert\sin(\omega/4)\rvert^M\leq \lvert\omega\rvert^M/4^M.
\]
With Lemma \ref{lem:supd}, we obtain that
\[
\lvert\hat \varphi_h(\lambda)\rvert\leq \lvert\omega\rvert^M/4^M\text{  ~and~  }
\lvert\hat \varphi_g(\lambda)\rvert\leq \lvert\omega\rvert^M/4^M.
\]
Hence, \begin{equation*}
\lvert\hat \varphi(\lambda)\rvert\leq\lvert\omega\rvert^M\cdot 2/4^M.
\end{equation*}

\item If $k$ is even, $k\geq 2$,\\
Let us use the expressions of $\hat \varphi_h(\lambda)$ and $\hat \varphi_g(\lambda)$ given respectively by \eqref{eqn:phih} and \eqref{eqn:phig}. 
We have,
\[
\Bigl\lvert\frac{\sin(\lambda/2)}{\lambda/2}\Bigr\rvert^M=\Bigl\lvert\frac{\sin(\omega/2)}{\omega/2}\Bigr\rvert\,\Bigl\lvert\frac{\omega}{\lambda}\Bigr\rvert^M\leq \Bigl\lvert\frac{\omega}{\lambda}\Bigr\rvert^M \leq \Bigl\lvert\frac{\omega}{\pi}\Bigr\rvert^M.
\]
Using Lemma \ref{lem:supd}, 
\[
\lvert\hat \varphi_h(\lambda)\rvert\leq \lvert\omega\rvert^M/\pi^M\text{  ~and~  }
\lvert\hat \varphi_g(\lambda)\rvert\leq \lvert\omega\rvert^M/\pi^M.
\]
Hence, \begin{equation*}
\lvert\hat \varphi(\lambda)\rvert\leq\lvert\omega\rvert^M\cdot 2/\pi^M.
\end{equation*}

\end{itemize}
We deduce that \ref{ass:phi_unif} follows with $C_\varphi=1$ when $M\geq 1$.

\subsection{Proof of Proposition \ref{prop:ineq}}

\label{sec:proof_ineq}

This section aims at recovering similar results than those given in \cite[Proposition 3]{Moulines07SpectralDensity} with explicit constants. 

\subsubsection[Approximation of the filters]{Proof of inequality \eqref{eqn:tauj}}

Observe that $t\mapsto 
\sum_{k\in\Z}\hat\varphi(\lambda+2k\pi)\rme^{\rmi\,t\,(\lambda+2k\pi)}$ is $2\pi$-periodic, integrable, and that its $\ell$-th Fourier coefficient is  $2\pi\varphi(t-\ell)$,
\begin{align*}
{\frac{1}{2\pi}\int_{-\infty}^\infty \bigl(\sum_{k\in\Z}\hat\varphi(\lambda+2k\pi)\rme^{\rmi\, t\,(\lambda+2k\pi)}\bigr)\rme^{-\rmi \ell\lambda}\rmd \lambda}
= \frac{1}{2\pi}\int_{-\infty}^\infty \hat\varphi(\lambda)\rme^{\rmi (t-\ell)\lambda}\rmd \lambda
= \varphi(t-\ell).
\end{align*}
 It follows that
\[
\sum_{\ell\in\Z}\varphi(t-\ell)\rme^{-\rmi\,t\,\lambda}=\sum_{k\in\Z}\hat\varphi(\lambda+2k\pi)\rme^{\rmi\,t\,(\lambda+2k\pi)}.
\]
 Hence, as in \cite[p180]{Moulines07SpectralDensity},
\begin{align*}
\hat\tau_j(\lambda)&=\int_{-\infty}^\infty \bigl(\sum_{k\in\Z}\hat\varphi(\lambda+2k\pi)\rme^{\rmi\,t\,(\lambda+2k\pi)}\bigr)2^{-j/2} \psi(2^{-j}t)\rmd t\\
 &=\sum_{k\in\Z}\hat\varphi(\lambda+2k\pi)\int_{-\infty}^\infty 2^{-j/2} \psi(2^{-j}t)\rme^{\rmi\,t\,(\lambda+2k\pi)}\rmd t\\
 &=2^{j/2}\sum_{k\in\Z}\hat\varphi(\lambda+2k\pi)\overline{\hat\psi(2^{j}(\lambda+2k\pi))}.
\end{align*}
We obtain
\begin{equation} \label{eqn:taujapprox}
\abs{\hat \tau_j(\lambda)-2^{j/2}\hat \varphi(\lambda)\overline{\hat \psi(2^{-j}\lambda)}}=2^{j/2}\abs{\sum_{k\in\Z,k\neq 0} \hat\varphi(\lambda+2k\pi)\overline{\hat\psi(2^j(\lambda+2k\pi))}}.
\end{equation}

Property \ref{ass:unif-smooth} yields, for all $\lambda\in(-\pi,\pi)$,
\[
\abs{\hat\psi(2^j(\lambda+2k\pi))}\leq \frac{C_s}{\abs{2^j(\lambda+2k\pi)}^\alpha}\leq \frac{C_s}{(2^j\pi\,(2\,|k|-1))^\alpha}.
\]
Using the inequality above and \ref{ass:phi_unif} in \eqref{eqn:taujapprox}, we get
\begin{equation*} 
\abs{\hat \tau_j(\lambda)-2^{j/2}\hat \varphi(\lambda)\overline{\hat \psi(2^{j}\lambda)}}\leq C_\tau 2^{j(1/2-\alpha)}\lvert\lambda\rvert^M,
\end{equation*}
with $C_\tau=C_s C_\varphi 2\zeta(\alpha)$, with $\zeta(\cdot)$ the Riemann zeta-function. Since $\alpha>1$, $C_\tau<\infty$. If $\alpha\geq 2$, $\zeta(\alpha)\leq \pi^2/6<2$, and, hence, one can take $C_\tau=4\,C_s C_\varphi$.

\subsubsection[Vanishing moments and regularity of the approximation of the filter]{Proof of  inequality \eqref{eqn:moments_phipsi} and  inequality  \eqref{eqn:smoothness_phipsi}}

{Using \ref{ass:moments}, \ref{ass:unif-smooth} and Lemma \ref{lem:fini},}   inequalities \eqref{eqn:moments_phipsi} and \eqref{eqn:smoothness_phipsi} are straightforward.

\subsubsection[Vanishing moments at a given scale]{Proof of  inequality \eqref{eqn:moments_tauj} and  inequality \eqref{eqn:smoothness_tauj}}

{With inequalities \eqref{eqn:tauj}, \eqref{eqn:moments_phipsi} and \eqref{eqn:smoothness_phipsi},} we get
\begin{align*}
\abs{\hat \tau_j(\lambda)}&\leq C_\infty C_m\,\abs{2^j\lambda}^M+ C_\tau 2^{j(1/2-\alpha)}\lvert\lambda\rvert^M\\
\text{and  }\quad \abs{\hat \tau_j(\lambda)} &\leq \frac{2^{j/2}}{(1+\abs{2^j\lambda})^{\alpha}}\Bigl(C_\infty C_s2^{-j/2}+ C_\tau 2^{-j\alpha}(1+2^j\lvert\lambda\rvert)^{\alpha}\lvert\lambda\rvert^M\bigr).
\end{align*}
It follows that 
\begin{align*}
\abs{\hat \tau_j(\lambda)}&\leq C_{m\tau}\,\abs{2^j\lambda}^M\\
\text{and  }\quad 2^{-j/2}\abs{\hat \tau_j(2^{-j}\lambda)} &\leq C_{s\tau}\frac{1}{(1+\abs{\lambda})^{\alpha}} \quad \text{   when  } \lvert\lambda\rvert\leq \pi,
\end{align*}
with $C_{m\tau}= C_\infty\,C_m+ C_\tau$, and $ C_{s\tau}=C_\infty\,C_s+C_\tau(1+\pi)^{\alpha+M}$.

\subsubsection[]{Proof of inequality \eqref{eqn:moments_smooth_tauj}}

With \eqref{eqn:moments_tauj}, inequality \eqref{eqn:moments_smooth_tauj} is straightforward with $C_{ms\tau}=C_{m\tau}(1+\pi)^{\alpha+M}$.

\subsection[Approximation of the filter]{Proof of Proposition \ref{prop:ass_infini}}

\label{sec:proof_reg}

Let $j\geq 1$ and $|\lambda|\leq 2^j\pi$. Inequalities \eqref{eqn:tauj} and \eqref{eqn:moments_phipsi} imply that
\begin{align*}
\lefteqn{\abs{\abs{2^{-j/2}\hat\tau(2^{-j}\lambda)}^2-
\abs{\hat\varphi(2^{-j}\lambda)\overline{\hat\psi(\lambda)}}^2}}\\
& \leq \abs{2^{-j/2}\hat\tau(2^{-j}\lambda)-
\hat\varphi(2^{-j}\lambda)\overline{\hat\psi(\lambda)}}\Bigl(\abs{2^{-j/2}\hat\tau(2^{-j}\lambda)}+
\abs{\hat\varphi(2^{-j}\lambda)\overline{\hat\psi(\lambda)}}\Bigr)\\
&\leq C_\tau 2^{-j\alpha}\abs{2^{-j}\lambda}^M\,(C_{m\tau}+C_\infty C_m)\abs{\lambda}^M\\
&\leq C_{a1}\,2^{-\gamma_1\,j}\abs{\lambda}^{2M}\,,
\end{align*}
with $C_{a1}=C_\tau\,(C_{m\tau}+C_m C_\infty)$ and $\gamma_1=M+\alpha$.
Similarly, using rather inequalities \eqref{eqn:tauj}, \eqref{eqn:smoothness_tauj} and \eqref{eqn:smoothness_phipsi}, 
\begin{equation*}
\abs{\abs{2^{-j/2}\hat\tau(2^{-j}\lambda)}^2-
\abs{\hat\varphi(2^{-j}\lambda)\overline{\hat\psi(\lambda)}}^2}\leq C_{b1}\,2^{-\gamma_1\,j}\abs{\lambda}^{M}(1+\abs{\lambda})^{-\alpha}\,,
\end{equation*}
with $C_{b1}=C_\tau\,(C_{s\tau}+C_s C_\infty)$.

Next,
\begin{align*}
\abs{\abs{\hat\varphi(2^{-j}\lambda)\overline{\hat\psi(\lambda)}}^2-\abs{\hat\psi(\lambda)}^2}& \leq \abs{\hat\psi(\lambda)}^2\abs{\abs{\hat\varphi(2^{-j}\lambda)}^2-
1}
\end{align*}
The control of the right-hand side is obtained with the following result.

\begin{lemma}
\label{lem:phi}
There exists a constant $C_Z=(M+L+1)$ 
 such that for all $j\in\N$, for all $\abs{2^{-j}\lambda}<\pi$, 
\begin{align*} 
\abs{2\abs{\hat \varphi_h(2^{-j}\lambda)}^2-1}&\leq C_Z\,\abs{2^{-j}\lambda}^2,\\
\abs{2\abs{\hat \varphi_g(2^{-j}\lambda)}^2-1}&\leq C_Z\,\abs{2^{-j}\lambda}^2.
\end{align*}
\end{lemma}
\begin{proof}

The proof is only derived for $\hat\varphi_h(\cdot)$. It is similar for $\hat\varphi_g(\cdot)$. Recall that 
\[
2^{1/2}\hat \varphi_h(2^{-j}\lambda)
=\Bigl(\frac{\sin(\lambda/2^{j+1})}{\lambda/2^{j+1}}\Bigr)^M\prod_{\ell=j+1}^\infty \overline{\hat d_L(2^{-\ell}\lambda)}.
\]

The triangular inequality implies that 
\begin{align*} \abs{2\abs{\hat \varphi_h(2^{-j}\lambda)}^2-1}&\leq \abs{\Bigl(\frac{\sin(\lambda/2^{j+1})}{\lambda/2^{j+1}}\Bigr)^{2M}-1}\prod_{\ell=j+1}^\infty \abs{\hat d_L(2^{-\ell}\lambda)}^2 + \abs{\prod_{\ell=j+1}^\infty \abs{\hat d_L(2^{-\ell}\lambda)}^2-1}\\
&\leq 2M\,\abs{\abs{\frac{\sin(\lambda/2^{j+1})}{\lambda/2^{j+1}}}-1} + \sum_{\ell=j+1}^\infty\abs{\abs{\hat d_L(2^{-\ell}\lambda)}^2-1},
\end{align*}
where we have used the equality $(x^K-1)=(x-1)\sum_{m=0}^{K-1}x^m$ for all $x\in\R$, $K\in\N$, and the fact that for all $x\in\R\setminus\{0\}$, $\abs{\sin(x)/x}\leq 1$ and $\abs{\hat d_L(x)}\leq 1$.

Taylor inequality states that for all $x\in\R\setminus\{0\}$, $\abs{\sin(x)-x}\leq \abs{x}^3/6$. Additionally, for all $\abs{2^{-\ell}\lambda/4}\leq 1$, 
\begin{align*}
\abs{\abs{\hat d_L(2^{-\ell}\lambda)}^2-1}&=\abs{\cos(2^{-\ell}\lambda/4)^{2(2L+1)}-1+\sin(2^{-\ell}\lambda/4)^{2(2L+1)}}\\
&\leq (2L+1)\sin\bigl(2^{-\ell}\lambda/4\bigr)^{2}+ \bigl(2^{-\ell}\lambda/4\bigr)^{2(2L+1)}\\
&\leq (2L+1)\bigl(2^{-\ell}\lambda/4\bigr)^{2}+ \bigl(2^{-\ell}\lambda/4\bigr)^{2}.
\end{align*}
We get
\begin{align*}
\abs{2\abs{\hat \varphi_h(2^{-j}\lambda)}^2-1} &\leq \frac{M}{3}\,\abs{2^{-j-1}\lambda}^2+ (2L+1)\abs{2^{-j}\lambda}^2\sum_{\ell=3}^\infty 2^{-2\ell}+ \abs{2^{-j}\lambda}^{2}\sum_{\ell=3}^\infty 2^{-2\ell}\\
&\leq \bigl(\frac{M}{6}+\frac{L}{2}+\frac{1}{2})\,\abs{2^{-j}\lambda}^2.
\end{align*}
Hence, \[ \abs{2\abs{\hat \varphi_h(2^{-j}\lambda)}^2-1}\leq C_Z\,\abs{2^{-j}\lambda}^2,
\]
with $C_Z=(M+L+1)/2$.
\end{proof}

We deduce that
\begin{align*}
\abs{\abs{\hat\varphi(2^{-j}\lambda)\overline{\hat\psi(\lambda)}}^2-\abs{\hat\psi(\lambda)}^2}& \leq C_Z\abs{2^{-j}\lambda}^{2}\abs{\hat\psi(\lambda)}^2.
\end{align*}
Using respectively \ref{ass:moments} and \ref{ass:unif-smooth}, we get
\begin{align*}
\abs{\abs{\hat\varphi(2^{-j}\lambda)\overline{\hat\psi(\lambda)}}^2-\abs{\hat\psi(\lambda)}^2}& 
\leq C_{a2}\,2^{-\gamma_2\,j}\abs{\lambda}^{2M}\,, \text{when }\abs{\lambda}\leq \pi,\\
\abs{\abs{\hat\varphi(2^{-j}\lambda)\overline{\hat\psi(\lambda)}}^2-\abs{\hat\psi(\lambda)}^2}& 
\leq C_{b2}\,2^{-\gamma_2\,j}\abs{\lambda}^{2}(1+\abs{\lambda})^{-2\alpha}\,, \text{when } 1\leq \abs{\lambda}\leq 2^j\pi.
\end{align*}
with $C_{a2}=C_m^2\,C_Z\pi^2$, $C_{b2}=C_s^2\,C_Z$ and $\gamma_2=2$.

We obtain inequality \eqref{ass::infini} and \eqref{ass::infini2} with $C_a=\max\{C_{a1}, C_{a2}, C_{b1}, C_{b2}\}$, $\gamma=\min\{\gamma_1,\gamma_2\}$. Observe that for all $M\geq 2$ $\gamma_1>2$, and, hence, $\gamma=2$ and we can take $C_{a}=2.5^{2M}\,(M+L+1)$.

\section{Proof of Proposition \ref{prop:support} and of Lemma \ref{lem:nj}}

Proposition 7.2 of \cite{Mallat99} states that $\varphi_g(.)$ and $\varphi_h(.)$ have the same support of the conjugate mirror filters associated to $h^{(L)}(.)$ and $g^{(L)}(.)$. When $\hat q_{L,M}(.)=1$, \eqref{eqn:h} and \eqref{eqn:g} yield
\[ \hat h^{(L)}(\lambda) =2^{-M+1/2}\Bigl(\sum_{k=0}^M(\rme^{-\rmi\lambda})^k\Bigr)\hat d_L(\lambda)\text{\quad and \quad} \hat g^{(L)}(\lambda) =2^{-M+1/2}\Bigl(\sum_{k=0}^M(\rme^{-\rmi\lambda})^k\Bigr)\,\rme^{\rmi\lambda}\overline{\hat d_L(\lambda)},\]
for all $\lambda\in\R$.
Moreover \eqref{eqn:d} writes as \[\hat d_L(\lambda)=1+\sum_{\ell=1}^L d(\ell) (\rme^{-\rmi\lambda})^{\ell}, \text{ with   } d(\ell)=(-1)^n \binom{L}{\ell}\prod_{k=0}^{\ell-1} \frac{1/2-L+k}{3/2+k}, \quad \ell=1,\dots,L\] (see {\it e.g.} Section 2.2 of \cite{common_factor}).
We deduce that $\hat h^{(L)}{(.)}$ and $\hat g^{(L)}{(.)}$ are polynomials of $\rme^{\rmi\lambda}$ with coefficients varying on. They are, hence, associated with conjugate mirror filters defined on $\{0,\ldots, M+L\}$ and on $\{-1-L,\ldots, M-1\}$. Hence, the support of $\varphi_g(.)$ and $\varphi_h(.)$ are respectively $[0,M+L]$ and on $[-1-L,\ldots, M-1]$. Consequently, the support of $\varphi(.)$ is $[-1-L,\ldots M+L]$, and has length $M+2L+1$.

Using again Proposition 7.2 of \cite{Mallat99}, we deduce that the supports of $\psi_g(.)$ and $\psi_h(.)$ are respectively $[-(M+L-1)/2,(M+L+1)/2]$ and $[-(M+L)/2,(M+L)/2]$, and thus that the support of $\psi(.)$ has length $M+L+1/2$. Proposition \ref{prop:support} follows.

{\it Remark.} For CFW-PR(M,L) and CFW-C(M,L) filters, the presence of the filter $\hat q_{L,M}(.)$ in $\hat h^{(L)}{(.)}$ and $\hat g^{(L)}(.)$ changes the supports. No theoretical statement provides the degree of $\hat q_{L,M}(.)$, but in practice, $\hat q_{L,M}(.)$ is a polynomial of $\rme^{-\rmi\lambda}$ of degree $M+L-1$. Then, the resulting supports of $\varphi_h(.)$ and $\varphi_g(.)$ are $[0, 2M+2L-1]$ and $[-1-L, 2M+L-1]$. Similarly, we deduce that the supports of $\psi_h(.)$ and $\psi_g(.)$ are $[-M-L, M+L]$ and $[-M-L+1/2,M+L+1/2]$.

Consider now the wavelet coefficients $\{\hat \bW_{j,k},\, j\geq 0,\,k\in\Z\}$, as defined in Lemma \ref{lem:nj}. Denote $[t_\varphi,T_\varphi]$ the support of $\varphi(.)$, and suppose $T_\varphi-t_\varphi\geq 1$. The functions $\hat \bX(t)=\sum_{k=1}^N \bX(k)\varphi(t-k)$ and $\tilde \bX(t)=\sum_{k\in\Z} \bX(k)\varphi(t-k)$ coincide for $t\in\bigl[T_\varphi, N+t_\varphi+1\bigr]$. Recall that $n_j$ is the number of coefficients at a scale $j\geq 0$ such that $\int_\R\hat \bX(t)\psi_{j,k}(t)\rmd t=\int_\R\tilde \bX(t)\psi_{j,k}(t)\rmd t$. Easy calculation then yields $n_j=\max\left\{0,\; \lfloor 2^{-j}(N-L_\phi+1)-L_\psi+1\rfloor\right\},$ where $L_\phi, L_\psi$ are the respective length of the supports of functions $\phi(.)$ and $\psi(.)$. As a consequence, $n_j=\max\left\{0,\; \lfloor 2^{-j}(N-2L-M-1)-L-M-1/2\rfloor\right\}$, for all $M\geq 1$, $L\geq 0$.

Now, suppose that $N^{-1}L\to 0$ and that $2^j\leq NL^{-1}$. Then, for all $j$ such that $2^jN^{-1}L\to 0$ when $N$ goes to infinity, $n_j\,2^{j}N^{-1} \to 1$ when $N$ goes to infinity. This concludes Lemma \ref{lem:nj}.

\section{Asymptotic behavior of the wavelet covariance}

This section deals with the proofs of the results of Section~\ref{sec:cov}. 
We will prove stronger results than Proposition~\ref{prop:covY2} which are stated below. To better {highlight} the role of the number of vanishing moments $M$ and the regularity $\alpha$, we keep these parameters, even if, for CFW-PR(M,L) and CFW-C(M,L) filters, we have $\alpha=M$ by Proposition~\ref{prop:filters}. Hence, we formulate here the assumption on the parameters on both $\alpha$ and $M$,
\begin{enumerate}[label=(C-\alph*),start=2]
\item\label{ass:parameters2} 
$-\alpha/2+\beta/2+1/2< d_\ell< M/2\quad\text{for all}\quad \ell=1,\dots,\dim$, $M\geq 2$ and $0<\beta<2$.
\end{enumerate}
Assumption \ref{ass:parameters2} is equivalent to assumption \ref{ass:parameters}.

Proposition \ref{prop:covY2} follows from the following proposition.
\begin{proposition}
\label{prop:covY1bis}
Let $\bX$ be a $\dim$-multivariate long range dependent process with long
memory parameters $d_1,\dots,d_{\dim}$ with generalized
spectral density $\bff(\cdot)$ satisfying~\ref{ass:zero-frequency} with short-range behavior \ref{ass:beta}.
Consider $\{W_{j,k}(\ell), \ell=1,\dots,p, j\geq 0, k\in\Z\}$ the wavelet coefficients of $\bX$ obtained with CFW-C(M,L) filters, $M, L\geq 2$.
Then we have, for all $j\geq0$, $k\in\Z$, 
\begin{multline*}
\Bigl\lvert 2^{-j\,(d_\ell+d_m)}\cov(W_{j,k}(\ell),W_{j,k}(m))
 -
\Omega_{\ell,m}\int_{-\infty}^\infty 
\abs{\lambda}^{-d_\ell-d_m}
\rme^{\sign(\lambda)\,\phi_{\ell,m}}\,\abs{\hat\psi(\lambda)}^2\;\rmd \lambda\Bigr\rvert\\
\leq C_1'\max\{2^{-j\beta},\,L\,2^{-2\,j}\}.
\end{multline*}
where 
 $C_1'$ is a constant only depending on $M$ and $C_f,\beta, \norm{\bOmega}, \{d_\ell,\ell=1,\dots,\dim\}$.
\end{proposition}

\subsection{Proof of Proposition~\ref{prop:covY1bis}}

Let $j\geq0$, $k\in\Z$. The quantity $\cov(\bW_{j,k})$ can be decomposed as
\begin{align*}
{\cov(\bW_{j,k})=\bA_j^{(+)}+  \bA_j^{(-)} \;,}
\text{~~with~~}
\bA_j^{(+)}&=\int_{0}^{\pi\,2^j}
\bff(2^{-j}\lambda)\;
2^{-j}\abs{\hat\tau_j(2^{-j}\lambda)}^2\;\rmd
\lambda\,,\\
\bA_j^{(-)}&= \int_{-\pi\,2^j}^{0} \bff(2^{-j}\lambda)\;
2^{-j}\abs{\hat\tau_j(2^{-j}\lambda)}^2\;\rmd
\lambda\,.
\end{align*}
Also recall that $\bGamma_j(\bd)$ is the diagonal matrix with diagonal entries $2^{-j\,d_1},\dots,2^{-j\,d_{\dim}}$.

We now sum up the main points for the convergence of $\cov(\bW_{j,k})$.
\begin{enumerate}
\item Behavior of $\bA_j^{(+)}$.\\
We introduce
\begin{align*}
\bB_j^{(+)}&=\int_{0}^{\pi\,2^j}
\bGamma_j(\bd)^{-1}\bLambda(\lambda) \bTheta \bLambda(\lambda)\bGamma_j(\bd)^{-1} \;
2^{-j}\abs{\hat\tau_j(2^{-j}\lambda)}^2\;\rmd
\lambda,\\
\bI_j^{(+)\mbox{inf}} &= \int_{0}^{\pi\,2^j}
\bGamma_j(\bd)^{-1}\bLambda(\lambda) \bTheta \bLambda(\lambda)\bGamma_j(\bd)^{-1} \;
\abs{\hat\psi(\lambda)}^2\;\rmd
\lambda,\\
\bI_j^{(+)\mbox{sup}} & = \int_{\pi\,2^j}^\infty
\bGamma_j(\bd)^{-1}\bLambda(\lambda) \bTheta \bLambda(\lambda)\bGamma_j(\bd)^{-1} \;
\abs{\hat\psi(\lambda)}^2\;\rmd
\lambda,\\
\bI^{(+)}&= \bI_j^{(+)\mbox{inf}}+\bI_j^{(+)\mbox{sup}}=\int_0^\infty 
\bGamma_j(\bd)^{-1}\bLambda(\lambda) \bTheta \bLambda(\lambda)\bGamma_j(\bd)^{-1}\;\abs{\hat\psi(\lambda)}^2 \;\rmd
\lambda.
\end{align*}
The steps of the convergence are:
\begin{enumerate}
\item $2^{-j(d_\ell+d_m-\beta)} \abs{\bA_j^{(+)}-\bB_j^{(+)}}$ is bounded using the regularity of the spectral density $\bff^S(\cdot)$ at the origin, that is,~\ref{ass:beta}, together with inequality \eqref{eqn:moments_tauj}.
\item $2^{-j(d_\ell+d_m-\beta)} \abs{\bB_j^{(+)}-\bI_j^{(+)\mbox{inf}}}$ is bounded using the convergence of the filter $\tau_j$ to $\hat\psi(.)$, through Proposition~\ref{prop:ass_infini}. We shall need inequality \eqref{eqn:moments_tauj} to control the integral around zero and inequality~\eqref{eqn:smoothness_tauj} to control the upper part.
\item $2^{-j(d_\ell+d_m-\beta)} \abs{\bI_j^{(+)\mbox{sup}}}$ is bounded using the regularity of $\hat\psi(.)$, that is, using~\ref{ass:unif-smooth}.
\end{enumerate}
All together, we shall obtain the convergence of $\bA_j^{(+)}$ to  $\bI^{(+)}$, which gives the property.
\item Behavior of $\bA_j^{(-)}$\\
We can apply the same arguments as for $\bA_j^{(+)}$ and obtain the convergence of $\bA_j^{(-)}$  to  $\bI^{(-)}$, with
\[
\bI^{(-)} = \int_{-\infty}^0
\bGamma_j(\bd)^{-1}\bLambda(\lambda) \bTheta \bLambda(\lambda)\bGamma_j(\bd)^{-1}\;\abs{\hat\psi(\lambda)}^2\;\rmd
\lambda.
\]
\end{enumerate}

In the following, $(\ell,m)\in\{1,\dots,\dim\}^2$ will denote two arbitrary indexes.

\subsubsection[Spectral approximation]{Spectral approximation, $\abs{\bA_j^{(+)}-\bB_j^{(+)}}$}
\label{sec:proof_spectral}

First notice that $\bGamma_j(\bd)^{-1}\bLambda(2^j\lambda) = \bLambda(\lambda) $. Hence,
\begin{align*}
\abs{\bA_{j}^{(+)}-\bB_j^{(+)}}&\leq \int_{0}^{\pi}
\abs{ \bff(\lambda)\;
- \bLambda(\lambda) \bTheta \bLambda(\lambda)}\abs{\hat\tau_j(\lambda)}^2\;\rmd
\lambda\\
&\leq \int_{0}^{\pi}
\abs{\bLambda(\lambda) \bTheta \bLambda(\lambda)}\circ\abs{\bff^S(\lambda)-1}\abs{\hat\tau_j(\lambda)}^2\;\rmd
\lambda.
\end{align*}
Property~\ref{ass:beta} gives
\[ \bigl(\abs{\bLambda(\lambda) \bTheta \bLambda(\lambda)}\circ\abs{\bff^S(\lambda)-1}\bigr)_{\ell,m} \leq C_f\,\norm{\bOmega}\, \abs{\lambda}^{-d_\ell-d_m+\beta}. \]
With a change of variable,
\[\abs{\bA_{j}^{(+)}-\bB_j^{(+)}}_{\ell,m}\leq C_f\,\norm{\bOmega}\,2^{j(d_\ell+d_m-\beta)} \,\int_0^{2^j\pi} {\abs{\lambda}^{-d_\ell-d_m+\beta}}\abs{2^{-j/2}\hat\tau_j(2^{-j}\lambda)}^2\,\rmd\lambda.
\]
We split the integral in two parts. First, with \eqref{eqn:moments_tauj},
\[\int_0^1 {\abs{\lambda}^{-d_\ell-d_m+\beta}}\abs{2^{-j/2}\hat\tau_j(2^{-j}\lambda)}^2\,\rmd\lambda \leq C_{m\tau}^2\,\int_0^1 {\abs{\lambda}^{-d_\ell-d_m+\beta+2M}}\,\rmd\lambda.
\]
As the parameters satisfy \ref{ass:parameters2}, the integral is bounded by a constant depending on $(d_\ell,d_m,\beta,M)$. The bound is independent on $L$ since the constant $C_{m\tau}$ does not depend on $L$.

Next, using the regularity given by \eqref{eqn:smoothness_tauj},
\[\int_1^{2^j\pi} {\abs{\lambda}^{-d_\ell-d_m+\beta}}\abs{2^{-j/2}\hat\tau_j(2^{-j}\lambda)}^2\,\rmd\lambda \leq C_{s\tau}^2\,\int_1^\infty \frac{\abs{\lambda}^{-d_\ell-d_m+\beta}}{(1+\abs{\lambda})^{2\alpha}}\,\rmd\lambda.
\]
Property \ref{ass:parameters2} ensures that the right hand side is bounded by a constant depending on $d_\ell, d_m, \beta$, $\alpha$, $M$, and not depending on $L$.

\subsubsection[Asymptotic of the filters]{Asymptotic of the filters, $\abs{\bB_j^{(+)}-\bI_j^{(+)\mbox{inf}}}$}
This step uses the convergence of the filter $\hat\tau_j$ to $\hat\psi(.)$, through Proposition~\ref{prop:ass_infini}.
First,
\[
\bGamma_j(\bd)\,2^{j\beta}\abs{\bB_j^{(+)}-\bI_j^{(+)\mbox{inf}}}\bGamma_j(\bd)\;\leq\;2^{j\beta}\int_{0}^{2^j\pi}
\bLambda(\lambda) \bTheta \bLambda(\lambda) \;
\abs{\abs{2^{-j/2}\hat\tau_j(2^{-j}\lambda)}^2-
\abs{\hat\psi(\lambda)}^2}\;\rmd
\lambda.
\]

Using \eqref{ass::infini2}, for $1\leq|\lambda|\leq 2^j\pi$,
$
\abs{\abs{2^{-j/2}\hat\tau_j(2^{-j}\lambda)}^2-
\abs{\hat\psi(\lambda)}^2}\leq C_{a}\,2^{-\gamma\,j}\abs{\lambda}^{2-\alpha}.
$
Thus,
\begin{multline*}
{\left(\int_{1}^{2^j\pi}
\bLambda(\lambda) \bTheta \bLambda(\lambda) \;
2^{j\beta}\abs{\abs{2^{-j/2}\hat\tau_j(2^{-j}\lambda)}^2-
\abs{\hat\psi(\lambda)}^2}\;\rmd
\lambda\right)_{(\ell,m)} }\\
\leq \norm{\bOmega}\,C_a\,2^{j(\beta-\gamma)}\int_{1}^{2^j\pi}\abs{\lambda}^{-d_\ell-d_m+2-\alpha}\rmd\lambda.
\end{multline*}
Depending on $-d_\ell-d_m+2-\alpha+1$ being negative, equal to zero or positive, the integral on the right-hand side is bounded up to a constant by 1, by $j$ or by $2^{j(-d_\ell-d_m+2-\alpha+1)}$. In the two first cases the right hand side goes to zero when $j$ goes to infinity since $\gamma>\beta$. In the last case, using \ref{ass:parameters2}, we obtain the bound $2^{j(\beta-d_\ell-d_m-\alpha+1)}$ which goes to zero when $j$ goes to infinity.

When $M\geq 2$, $\gamma=2$, $-d_\ell-d_m+2-2\alpha+1<0$ and $C_a=2\,(M+L+1)$. With a fixed $M$, this term is, hence, bounded up to a constant by  $L\,2^{j(\beta-2)}$.

It remains to consider the integral on $(0,1)$.
Property \eqref{ass::infini} states that
\begin{multline*}
{\left(\int_{0}^{1}
\bLambda(\lambda) \bTheta \bLambda(\lambda) \;
2^{j\beta}\abs{\abs{2^{-j/2}\hat\tau_j(2^{-j}\lambda)}^2-
\abs{\hat\psi(\lambda)}^2}\;\rmd
\lambda\right)_{(\ell,m)} }\\
\leq \norm{\bOmega}\,C_a\,2^{j(\beta-\gamma)}\int_0^{1}\abs{\lambda}^{-d_\ell-d_m+2M}\rmd\lambda.
\end{multline*}
The right-hand side tends to 0 when $j$ goes to infinity since $\beta<\gamma$ and $\max\{d_\ell, \ell=1,\dots,\dim\}< M+1/2$.

When $M\geq 2$, $\gamma=2$ and $C_a=2\,(M+L+1)$. With a fixed $M$, this term is, hence, bounded up to a constant by  $\max\{1,L\,2^{j(\beta-\gamma)}\}=\max\{1,L\,2^{j(\beta-2)}\}$.

\label{sec:proof_Ca}

\subsubsection[Regularity of the filters]{Regularity of the filters, $\abs{\bI_j^{(+)\mbox{sup}}}$}
\label{sec:Isup}
This step uses the regularity of $\hat\psi(.)$. Indeed, property~\ref{ass:unif-smooth} entails that
\begin{align*}
\abs{\bI_{j,\ell m}^{(+)\mbox{sup}}} & \leq \norm{\bOmega}2^{j(d_\ell+d_m)}\int_{2^j\pi}^\infty \abs{\lambda}^{-d_\ell-d_m}\;
\abs{\hat\psi(\lambda)}^2\;\rmd \lambda\\
& \leq C_s^2\norm{\bOmega}2^{j(d_\ell+d_m)}\int_{2^j\pi}^\infty \frac{\abs{\lambda}^{-d_\ell-d_m}}{(1+\abs{\lambda})^{2\alpha}}\;\rmd \lambda\\
& \leq C_s^2\norm{\bOmega}\,\pi^{-\beta}\,2^{j(d_\ell+d_m-\beta )}\int_{2^j\pi}^\infty \abs{\lambda}^{-d_\ell-d_m+\beta-2\alpha}\;\rmd \lambda,
\end{align*}
where last inequality results from the fact that when $|\lambda|\geq 2^j\pi$, then $1\leq |\lambda|^\beta 2^{-j\beta}\pi^\beta$.
Property \ref{ass:parameters2} thus implies that 
$2^{-j(d_\ell+d_m-\beta)} \abs{\bI_{j,\ell m}^{(+)\mbox{sup}}}$ is bounded by a constant depending of $d_\ell,d_m,\beta,\norm{\bOmega}$, and $M$.

\subsection{Proof of Proposition~\ref{prop:covY3}}

Recall that \begin{align*}
\hat\psi(\lambda)&=\hat \psi_h(\lambda)+\rmi\,\hat \psi_g(\lambda)=\left(1-\rme^{\rmi \eta_L(\lambda)}\right)\hat \psi_h(\lambda)\;,\\
\text{with~~  }  \alpha_L(\lambda)&= 2(-1)^L\,\atan\left(\tan^{2L+1}(\lambda/4)\right)\;,\\
  \eta_L(\lambda)&=-a_L(\lambda/2+\pi)+\sum_{j=1}^{\infty} a_L(2^{-j-1}\lambda)\;.
\end{align*}
Theorem~\ref{th:main-result-analyticity} states that, for
  all $\lambda\in\R$,
  \[
  \left|\hat\psi(\lambda) -
    2\1_{\R_+}(\lambda)\,\hat\psi_h(\lambda)\right|= U_L(\lambda)
  \left|\hat\psi_h(\lambda)\right|\;,
  \]
  with
\begin{equation}\label{eqn:Ul}
U_L(\lambda) \leq
    2\sqrt{2}
\left(\log_2\left(\frac{\max(4\pi,\abs{\lambda})}{2\pi}\right)+2\right)\,\left(1-\frac{\delta(\lambda,4\pi\Z)}{\max(4\pi,\abs{\lambda})}\right)^{2L+1}
 \;.
\end{equation}
We deduce from Theorem~\ref{th:main-result-analyticity} the following results, which gives inequalities in a form that can be more useful in future developments.
\begin{corollary}
\label{cor::quasi-analytic}
For all $\hat q_{L,M}(.)$ real polynomial of $(e^{-i\lambda})$, for
  all $\abs{\lambda}\leq 2\pi$,
\begin{equation*}
 \abs{\abs{\hat \psi(\lambda)}^2 -
    4\,\1_{\R_+}(\lambda)\,\abs{\hat \psi_h(\lambda)}^2}\leq  12\sqrt{2}\left(1-\frac{\abs{\lambda}}{2\pi}\right)^{2L+1} 
  \abs{\hat \psi_h(\lambda)}^2\;.
\end{equation*}
For $2\pi\leq\abs{\lambda}\leq 4\pi$,
\begin{equation*}
 \abs{\abs{\hat \psi(\lambda)}^2 -
    4\,\1_{\R_+}(\lambda)\,\abs{\hat \psi_h(\lambda)}^2}\leq  12\sqrt{2}\left(\frac{\abs{\lambda}}{4\pi}\right)^{2L+1} 
  \abs{\hat \psi_h(\lambda)}^2\;.
\end{equation*}
\end{corollary}
The proof is straightforward and it thus omitted.

Let us introduce also
\begin{align*}
\bI_j^{\mbox{inf}} & = \int_{-\pi\,2^j}^{\pi\,2^j}
\bGamma_j(\bd)^{-1}\bLambda(\lambda) \bTheta \bLambda(\lambda)\bGamma_j(\bd)^{-1} \;
\abs{\hat\psi(\lambda)}^2\;\rmd
\lambda, \\
\bI_j^{\mbox{inf},\mbox{analytic}} &= \int_{0}^{\pi\,2^j}
\bGamma_j(\bd)^{-1}\bLambda(\lambda) \bTheta \bLambda(\lambda)\bGamma_j(\bd)^{-1} \;
4\,\abs{\hat\psi_h(\lambda)}^2\;\rmd
\lambda,\\
\bI_j^{\mbox{sup},\mbox{analytic}} & = \int_{\pi\,2^j}^\infty
\bGamma_j(\bd)^{-1}\bLambda(\lambda) \bTheta \bLambda(\lambda)\bGamma_j(\bd)^{-1} \;
4\,\abs{\hat\psi_h(\lambda)}^2\;\rmd
\lambda.
\end{align*}

Following the proof of Proposition~\ref{prop:covY1bis}, the steps of the proof are the following: 
\begin{enumerate}
\item $2^{-j(d_\ell+d_m-\beta)} \abs{\cov(\bW_{j,k})-\bI_j^{\mbox{inf}}}$ is bounded, up to a constant, by $L2^{-2j}$. This result was already obtained in the proof of Proposition~\ref{prop:covY1bis}.
\item $2^{-j(d_\ell+d_m-\beta)} \abs{\bI_j^{\mbox{inf}}-\bI_j^{\mbox{inf},\mbox{analytic}}}$ is bounded up to a constant by $\max\{1,L^{-(2M-d_\ell-d_m+1)}2^{j\beta}\}$, using the quasi-analyticity property, stated in Corollary \ref{cor::quasi-analytic}.
\item $2^{-j(d_\ell+d_m-\beta)} \abs{\bI_j^{\mbox{sup},\mbox{analytic}}}$ is bounded. This result is straightforward with step 1.(c) in the proof of Proposition~\ref{prop:covY1bis}, since {$\lvert\hat\psi_h(\lambda) \rvert\leq \lvert\hat\psi(\lambda) \rvert$}.
\end{enumerate}
Hence, it only remains to prove step 2.
That is, we want to establish that the quantity \begin{equation*}
\int_{-2^j\pi}^{2^j\pi} \abs{\lambda}^{-d_\ell-d_m}\abs{\abs{\hat \psi(\lambda)}^2 -
    4\,\1_{\R_+}(\lambda)\,\abs{\hat \psi_h(\lambda)}^2}\rmd\lambda
    \end{equation*} 
is bounded up to a constant by $\max\{2^{-j\beta},L^{-M-1}\}$.

To this aim, we will use \eqref{eqn:Ul} (and Corollary~\ref{cor::quasi-analytic}) and the inequality
\begin{equation}\label{eqn:ineqPsi}
\lvert\hat\psi_h(\lambda)\rvert\leq 5^M\,|\sin(\lambda/4)|^M\,(1+|\lambda|)^{-M},
\end{equation}
for all $\lambda\in\R\setminus\{0\}$ (see Section \ref{sec:unif-smooth}).

We decompose the integral on the sub-intervals $(-2\pi,2\pi)$, $(-4\pi,-2\pi)$, $(2\pi,4\pi)$, $(-2^j\pi,-4\pi)$, $(-2^j\pi,-4\pi)$ and $(4\pi,2^j\pi)$.

\subsubsection*{On $(-2\pi,2\pi)$.}
With assumption~\ref{ass:moments}, Corollary~\ref{cor::quasi-analytic} leads to
\begin{align*}
\lefteqn{\int_{-2\pi}^{2\pi} \abs{\lambda}^{-d_\ell-d_m}\abs{\abs{\hat \psi(\lambda)}^2 -
    4\,\1_{\R_+}(\lambda)\,\abs{\hat \psi_h(\lambda)}^2}\rmd\lambda}\\
    &\leq 24\sqrt{2}\,\int_0^{2\pi} \left(1-\frac{\abs{\lambda}}{4\pi}\right)^{2L+1}{\abs{\lambda}^{-d_\ell-d_m+2M}}\rmd\lambda,\\
    &\leq 24\sqrt{2}\,{(4\pi)^{-d_\ell-d_m+2M+1}\,}B(2M-d_\ell-d_m+1,2L+2),
\end{align*}
where $B(.,.)$ is the Beta function. Using Stirling's approximation, for fixed $M$, $d_\ell$, $d_m$, when $L$ goes to infinity the right-hand side is equivalent to \[ 36\sqrt{2}\; \Gamma(2M-d_\ell-d_m+1)(4\pi)^{-d_\ell-d_m+2M+1}\,\left(2L+2\right)^{-(2M-d_\ell-d_m+1)}.\]
This bound is negligible with respect to $L^{-M-1}$, since $d_\ell+d_m< M$ under \ref{ass:parameters2}.

\subsubsection*{On $(-4\pi,-2\pi)$ and $(2\pi,4\pi)$.}

Observe that \eqref{eqn:ineqPsi} yields $\lvert \hat\psi_h(\lambda)|\leq 5^M (4\pi-\lambda)^M$.
With Corollary~\ref{cor::quasi-analytic}, we get
\begin{align*}
\lefteqn{\int_{2\pi}^{4\pi} \abs{\lambda}^{-d_\ell-d_m}\abs{\abs{\hat \psi(\lambda)}^2 -
    4\,\1_{\R_+}(\lambda)\,\abs{\hat \psi_h(\lambda)}^2}\rmd\lambda}\\
    &\leq 12\sqrt{2}\,5^{2M}\,(4\pi)\int_{2\pi}^{4\pi} \left(1-\frac{\lambda}{4\pi}\right)^{2L+1}(4\pi-\lambda)^{2M}\rmd\lambda,\\
    &\leq 12\sqrt{2}\,5^{2M}\,(4\pi)^{2M+2}\,B(2M+1,2L-1).
\end{align*}
Using Stirling's approximation, for fixed $M$ and $L$ going to infinity, the right-hand side is equivalent, up to a multiplicative constant to $L^{-2M-1}$. It is therefore lower than $L^{-M-1}$ for sufficiently large $L$.

A similar result is obtained on $(-4\pi,-2\pi)$.

\subsubsection*{On $(-2^j\pi,-4\pi)$ and $(4\pi,2^j\pi)$.}

Let us first consider the integral on an interval $(4k\pi, 4k\pi+2\pi)$, with $k\in\N$, $k\geq 2$.
Inequality \eqref{eqn:ineqPsi} implies that $\lvert \hat\psi_h(\lambda)|\leq 5^M \Bigl\lvert\frac{4k\pi-\lambda}{\lambda}\Bigr\rvert^{M/2}|\lambda|^{-M/2}$.
With \eqref{eqn:Ul}, we get
\begin{align*}
\lefteqn{\int_{4k\pi}^{4k\pi+2\pi} \abs{\lambda}^{-d_\ell-d_m}\abs{\abs{\hat \psi(\lambda)}^2 -
    4\,\1_{\R_+}(\lambda)\,\abs{\hat \psi_h(\lambda)}^2}\rmd\lambda}\\
    &\leq 12\sqrt{2}\,5^{2M}\,\log_2(k)\int_{4k\pi}^{4k\pi+2\pi} \abs{\lambda}^{-d_\ell-d_m} \left(\frac{4k\pi}{\lambda}\right)^{2L+1}\,\left(1-\frac{4k\pi}{\lambda}\right)^{M}|\lambda|^{-M}\rmd\lambda.
  \end{align*}
Since $-M-d_\ell-d_m<0$, $|\lambda|^{-M-d_\ell-d_m}<(4k\pi)^{-M-d_\ell-d_m}$ for all $\lambda > 4k\pi$. With the change of variable $\mu=\frac{4k\pi}{\lambda}$, we obtain the upper bound
\begin{align}
\nonumber & 12\sqrt{2}\,5^{2M}\,\log_2(k)(4k\pi)^{-M-d_\ell-d_m+1}\int_{0}^{1} \left(\mu\right)^{2L-1}\,\left(1-\mu\right)^{M}|\rmd\mu,\\        
\label{eqn:bound1}    &\leq 12\sqrt{2}\,5^{2M}\,(4k\pi)^{-M-d_\ell-d_m+1}\,\log_2(k)\,B(M+1,2L),
\end{align}
with $B(.,.)$ the Beta function.

Second, we consider the integral on an interval $(4k\pi-2\pi, 4k\pi)$, with $k\in\N$, $k\geq 2$.
Inequality \eqref{eqn:ineqPsi} implies that $\lvert \hat\psi_h(\lambda)|\leq 5^M \Bigl\lvert\frac{4k\pi-\lambda}{\lambda}\Bigr\rvert^{M/2}|\lambda|^{-M/2}$.
With \eqref{eqn:Ul}, we get
\begin{align*}
\lefteqn{\int_{4k\pi-2\pi}^{4k\pi} \abs{\lambda}^{-d_\ell-d_m}\abs{\abs{\hat \psi(\lambda)}^2 -
    4\,\1_{\R_+}(\lambda)\,\abs{\hat \psi_h(\lambda)}^2}\rmd\lambda}\\
    &\leq 12\sqrt{2}\,5^{2M}\,\log_2(k)\int_{4k\pi-2\pi}^{4k\pi} \abs{\lambda}^{-d_\ell-d_m} \left(1-\frac{4k\pi-\lambda}{\lambda}\right)^{2L+1}\,\left(\frac{4k\pi-\lambda}{\lambda}\right)^{M}|\lambda|^{-M}\rmd\lambda\\
        &\leq 12\sqrt{2}\,5^{2M}\,\log_2(k)(4k\pi-2\pi)^{-M-d_\ell-d_m}(4k\pi)\int_{0}^{1} \left(1-\mu\right)^{2L-1}\,\left(\mu\right)^{M}\rmd\lambda,
  \end{align*}
where we have done the change of variable $\mu=\frac{4k\pi}{\lambda}-1$.
We obtain the bound
\begin{equation}\label{eqn:bound2}
12\sqrt{2}\,5^{2M}\,(4(k-1)\pi)^{-M-d_\ell-d_m+1}\,\log_2(k)\,B(M+1,2L),
\end{equation}
with $B(.,.)$ the Beta function.

Next, \begin{multline}\label{eqn:bound_somme}
{\int_{4\pi}^{2^j\pi}\abs{\lambda}^{-d_\ell-d_m}\abs{\abs{\hat \psi(\lambda)}^2 -
    4\,\1_{\R_+}(\lambda)\,\abs{\hat \psi_h(\lambda)}^2}\rmd\lambda}\\
       \leq \sum_{k=2}^{2^{j-2}} \int_{4k\pi-2\pi}^{4k\pi} \abs{\lambda}^{-d_\ell-d_m}\abs{\abs{\hat \psi(\lambda)}^2 -
    4\,\1_{\R_+}(\lambda)\,\abs{\hat \psi_h(\lambda)}^2}\rmd\lambda \\
    +\sum_{k=1}^{2^{j-2}} \int_{4k\pi}^{4k\pi+2\pi} \abs{\lambda}^{-d_\ell-d_m}\abs{\abs{\hat \psi(\lambda)}^2 -
    4\,\1_{\R_+}(\lambda)\,\abs{\hat \psi_h(\lambda)}^2}\rmd\lambda.
\end{multline}
The two terms on the right hand side are bounded by \eqref{eqn:bound1} and \eqref{eqn:bound2}. Since $M\geq 2$ and $d_\ell+d_m>-1$, $-M-d_\ell-d_m<-1$. Hence, $\sum_{k=1}^{2^j} \log_2(k)\,k^{-M-d_\ell-d_m}$ can be bounded by a constant, not depending on $j$. We deduce that, for fixed $M$, up to a constant, the right-hand side of \eqref{eqn:bound_somme} is bounded by $B(2M-2,2L)$. Stirling's approximation states that $B(M+1,2L)$ is equivalent to $L^{-M-1}$, for fixed $M$ and $L$ going to infinity.

The bound on $(-2^j\pi,4\pi)$ is similar.

\section{Asymptotic behavior of the estimators}

We detail some points that are changed with the complex wavelet setting in the proofs of consistency and of asymptotic normality, with respect to the real wavelets setting.

First, recall that, for all $j\geq 0$, $n_j$ denotes the number of non zero wavelet coefficients $\{\bW_{j,k},\, k\in\Z\}$. Under the assumptions that $2^{-j_0}L$ is bounded and that $2^{-j_0}N\to \infty$, the sequence $n_j$ is equivalent to $2^{-j}N$ when $j$ goes to infinity. These assumptions are made in both Theorem \ref{prop:convergence} and Theorem \ref{thm:normality}. Hence, $n_j$ behaves similarly to in \cite{AchardGannaz} and in \cite{normality}.

\subsection{Proof of Theorem \ref{prop:convergence} }

For complex wavelets the approximation of the wavelet covariance does not admit the same bound  as for real wavelets. Hence, the study of the term \[
S^{(1)}_{\ell,m}(\mu)=\sum_{j=j_0}^{j_1} n_j \mu_j\left[\frac{\cov(W_{j,k}(\ell),W_{j,k}(m)) }{2^{j(d_\ell^0+d_m^0)}}-G_{\ell,m}^0 \right],
\]
defined on page 499 of \cite{AchardGannaz}, is modified. Consequently, Proposition 8 and Proposition 9 of \cite{AchardGannaz} do not hold anymore. They are replaced respectively by Proposition \ref{prop:Smu} and Proposition \ref{prop:Smu2} below.

Let us take $\ell$ and $m$ in $1,\dots,p$, and define, for any sequence $\mu=\{\mu_{j},\,j\geq 0\}$, 
\begin{equation*}
S_{\ell,m}(\mu)=\sum_{j,k} \mu_{j} \left(\frac{W_{j,k}(\ell)W_{j,k}(m)}{ 2^{j(d_\ell^0+d_m^0)}}-G_{\ell,m}^0\right)=\sum_{j=j_0}^{j_1} \mu_{j} \left(\frac{I_{\ell,m}(j)}{2^{j(d_\ell^0+d_m^0)}}-n_j G_{\ell,m}^0\right).
\end{equation*}

$S_{\ell,m}(\mu)$ is decomposed in two terms $S^{(0)}_{\ell,m}(\mu)$ and $S^{(1)}_{\ell,m}(\mu)$,
\begin{align*}
S^{(0)}_{\ell,m}(\mu)&= \sum_{j=j_0}^{j_1} \mu_j \frac{1}{2^{j(d_\ell^0+d_m^0)}}\sum_k\left({W_{j,k}(\ell)W_{j,k}(m)}-\cov(W_{j,k}(\ell),W_{j,k}(m)) \right),\\\
S^{(1)}_{\ell,m}(\mu)&=\sum_{j=j_0}^{j_1} n_j \mu_j\left[\frac{\cov(W_{j,k}(\ell),W_{j,k}(m)) }{2^{j(d_\ell^0+d_m^0)}}-G_{\ell,m}^0 \right].
\end{align*}

\begin{proposition}
\label{prop:Smu}

Assume that the sequences $\mu$ belong to the set $\{ \{\mu_j\}_{j \geq 0},\; |\mu_j|\leq\frac{1}{n_j}\}$. Suppose that \ref{ass:parameters} holds.  Under condition (C), $\sup_{\{\mu,\; |\mu_j|\leq \frac{1}{n_j}\}}S_{\ell,m}(\mu)$ is uniformly bounded by $2^{-j_0\beta}+L2^{-2j_0}+j_1 L^{-M-1}+N^{-1/2} 2^{j_1/2}$ up to a multiplicative constant, that is, 
  $$\sup_{\mu\in\{ (\mu_j)_{j \geq 0},\; |\mu_j|\leq\frac{1}{n_j}\}} \{ S_{\ell,m}(\mu) \} = \BigO_\P (2^{-j_0\beta}+L2^{-2j_0}+j_1 L^{-M-1}+N^{-1/2} 2^{j_1/2}).$$
\end{proposition}

\begin{proof}
From Proposition~\ref{prop:covY3}, there exists $C>0$ such that
\begin{equation}
|S^{(1)}_{\ell,m}(\mu)|\leq  C \sum_{j=j_0}^{j_1} \bigl(2^{-\beta j} + L2^{-j}+L^{-M-1}\bigr)n_j |\mu_j|.\label{eqn:S1}
\end{equation} 
Under the assumption $|\mu_j|\leq \frac{1}{n_j}$, we have the inequality $|S^{(1)}_{\ell,m}(\mu)|\leq  C \sum_{j=j_0}^{j_1} \bigl(2^{-\beta j} + L2^{-j}+L^{-M-1}\bigr)$. The right-hand bound is equivalent to $2^{-j_0\beta}+L2^{-2j_0}+j_1 L^{-M-1}$ up to a constant.

The term $S^{(0)}$ is unchanged and the proposition follows as in the proof of Proposition 8 of \cite{AchardGannaz}.
\end{proof}

\begin{proposition}
\label{prop:Smu2}
Let $0<j_0 \leq j_1 \leq j_N$. Suppose that \ref{ass:parameters} holds. 
Assume that the sequences $\mu$ belong to the set 
$$\mathcal{S}(q,\gamma,c)=\{ \{\mu_j\}_{j \geq 0}, |\mu_j|\leq\frac c n |j-j_0+1|^q 2^{(j-j_0)\gamma},~ \forall j=j_0,\ldots j_1\}$$ with $0\leq\gamma<1$. 
Under condition (C), $\sup_{\mu\in\mathcal{S}(q,\gamma,c)}S_{\ell,m}(\mu)$ is uniformly bounded by $2^{-j_0\beta}+L2^{-2j_0}+\log(N) L^{-M-1}+H_\gamma(N^{-1/2} 2^{j_0/2})$ up to a constant,
  $$\sup_{\mu\in\mathcal{S}(q,\gamma,c)}\{ S_{\ell,m}(\mu) \} =\BigO_\P (2^{-j_0\beta}+L2^{-2j_0}+j_1 L^{-M-1}+H_\gamma(N^{-1/2} 2^{j_0/2}))$$
  with $H_\gamma(u)=\begin{cases} u & \text{~if~~} 0\leq\gamma<1/2,\\
   \log(1+u^{-2})^{q+1}\,u & \text{~if~~} \gamma=1/2,\\
   \log(1+u^{-2})^q\, u^{2(1-\gamma)} & \text{~if~~} 1/2<\gamma<1.
  \end{cases}$
\end{proposition}
In particular, for any $0\leq\gamma<1$, under the assumptions $2^{-j_0\beta}+N^{-1/2} 2^{j_0/2}\to 0$, and $L2^{-2j_0}+\log(N) L^{-M-1}\to 0$,  we have $\sup_{\mu\in\mathcal{S}(q,\gamma,c)}\{ S_{\ell,m}(\mu) \} =o_\P(1)$.
\begin{proof}

Under the assumptions of the proposition, one deduce from inequality~\eqref{eqn:S1} that,  
\begin{align*}
\sup_{\mu\in\mathcal S(q,\gamma,c)}|S^{(1)}_{\ell,m}(\mu)| & \leq  cC \frac 1 n \sum_{j=j_0}^{j_1} n_j \bigl(2^{-\beta j} + L2^{-j}+L^{-M-1}\bigr)2^{\gamma(j-j_0))}(j-j_0+1)^q\\
&\leq cC\, 2^{-\beta j_0}\sum_{i=0}^{j_1-j_0} 2^{-(1+\beta -\gamma)i}(i+1)^q +cC\, L2^{-j_0} \sum_{i=0}^{j_1-j_0} 2^{-(2 -\gamma)i}(i+1)^q\\
& \qquad +cC \,j_1 L^{-M-1}\sum_{i=0}^{j_1-j_0} 2^{-(1 -\gamma)i}(i+1)^q.
\end{align*}
The sums on the right-hand side of the inequality tend to 0 under the assumptions of the proposition, since  $1-\gamma>0$.

The term $S^{(0)}$ is unchanged and the proposition follows as in the proof of Proposition 9 of \cite{AchardGannaz}.
\end{proof}

The rest of the proof is very similar to the real case and it is omitted. Remark that a key of the proof is Oppenheim's inequality, which holds for complex matrices, see \cite{MatrixBook}.

\subsection{Proof of Theorem \ref{thm:normality} }

\paragraph{Expressions of the asymptotic variances.~\\ }

For $u\geq 0$, $(\delta_1,\delta_2)\in(-\alpha,M)^2$, define $\tilde I_u(\delta_1,\delta_2)$ \begin{equation*}
\tilde I_u(\delta_1,\delta_2)=\frac{2\pi}{K(\delta_1)K(\delta_2)}\int_{-\pi}^\pi\overline{\tilde D_{u,\infty}(\lambda;\delta_1)}\tilde D_{u,\infty}(\lambda;\delta_2)\, \rmd\lambda\,,
\end{equation*} where $D_{u,\infty}(\lambda;\delta)$ is an approximation of the cross-spectral density between wavelet coefficients $\{\bW({j,k}), \,k\in\Z\}$ and  $\{\bW({j+u,2^uk+\tau}),\; \tau=0,\dots, 2^u-1,\, k\in\Z\}$, 
\begin{align*}
 D_{u,\tau}(\lambda;\delta)&= \sum_{t\in\Z}\lvert\lambda+2t\pi\rvert^{-\delta} \overline{\hat\psi(\lambda+2t\pi)}\,2^{u/2}\hat\psi(2^u(\lambda+2t\pi))\,\rme^{-\rmi 2^{u}\tau(\lambda+2t\pi)}\,,\\
\tilde D_{u,\infty}(\lambda;\delta) &= \sum_{\tau=0}^{2^{-u}-1} D_{u,\tau}(\lambda;\delta)\,.
\end{align*}
We introduce 
\begin{align*}
&\mathcal I^d_\Delta(\delta_1,\delta_2)= \frac{2}{\kappa_\Delta}\tilde I_0(\delta_1,\delta_2)\\
\nonumber & \pushright{+\frac{2}{\kappa_\Delta^2}\sum_{u=1}^{\Delta}(2^{u\delta_1}+2^{u\delta_2})\,2^{-u} \frac{2-2^{-\Delta+u}}{2-2^{-\Delta}}((u+\eta_{\Delta-u}-\eta_\Delta)(\eta_{\Delta-u}-\eta_\Delta)+\kappa_{\Delta-u})\,\tilde I_u(\delta_1,\delta_2)}\\
 &\pushright{\text{ if } \Delta<\infty,} \\
&\mathcal I^d_\infty(\delta_1,\delta_2)= \tilde I_0(\delta_1,\delta_2)+ \sum_{u=1}^{\infty}(2^{u\delta_1}+2^{u\delta_2})\,2^{-u}\,\tilde I_u(\delta_1,\delta_2)\;,\quad {\text{if } \Delta=\infty.}
\end{align*}
 Define also
\begin{multline}
\label{eqn:Itilded_mat}
\GIGd(\Delta)=\diag{\vect{\bG^0}}
\bigl(\mathcal I^d_\Delta(d_a^0+d_b^0, d_{a'}^0+d_{b'}^0)_{(a,b),(a',b')\in\{1,\dots,p^2\}}\bigr)\diag{\vect{\mathbf{G}^0}}\,.
\end{multline}
Additionally, let us denote
\begin{align*}
\mathcal I_\Delta^G(\delta_1,\delta_2) & = \tilde I_{0}(\delta_1,\delta_2)+\sum_{u=1}^\Delta (2^{u\delta_1}+2^{u\delta_2})2^{-u}\,\frac{2-2^{-\Delta+u}}{2-2^{-\Delta}}\,\tilde I_{u}(\delta_1,\delta_2)
 &\text{if } \Delta<\infty, \\
\mathcal I_\infty^G(\delta_1,\delta_2) &=\tilde I_{0}(\delta_1,\delta_2)+\sum_{u=1}^\infty (2^{u\delta_1}+2^{u\delta_2})2^{-u}\,\tilde I_{u}(\delta_1,\delta_2) &\text{if } \Delta=\infty. 
\end{align*}
Let us also define
\begin{multline}
\label{eqn:ItildeG_mat}
\GIGG(\Delta)=\diag{\vect{\bG^0}}
\bigl(\mathcal I_\Delta^G(d_a^0+d_b^0, d_{a'}^0+d_{b'}^0)_{(a,b),(a',b')\in\{1,\dots,p^2\}}\bigr)\diag{\vect{\bG^0}}.
\end{multline}

Let us reformulate Theorem \ref{thm:normality} with the exact expression of the asymptotic variance.

\begin{theorem}
\label{thm:normality2}
Suppose that conditions of Theorem~\ref{thm:cvgce} are satisfied and that assumption \ref{ass:linear} hold.
Let $j_0 < j_1 \leq j_N$ with $j_N=\max\{j,n_j\geq 1\}$ such that \[
j_1-j_0\to \Delta\in\{1,\dots,\infty\},\; \log(N)^2(N 2^{-j_0(1+2\beta)}+ N^{-1/2} 2^{j_0/2})\to 0. 
\] Define $n=\sum_{j=j_0}^{j_1} n_j$.

Consider CFW-C(M,L) filters with $M\geq 2$ and $\log(N)^2\,N^{1/2} 2^{-j_0/2}(L2^{-j_0}+ L^{-M-1})\to 0$.

Then, \begin{itemize}
\item 
$\sqrt{n}(\hat\bd -\bd^0)$ converges in distribution to a centered Gaussian distribution with a variance equal to \begin{equation}
\label{eqn:vard}
\mathbf{V}^{(\bdexp)}(\Delta)=\frac{1}{2\,\log(2)^{2}}(\bG^{0-1}\circ \bG^0+\bI_p)^{-1}\, \mathbf{\Upsilon}(\Delta)\, (\bG^{0-1}\circ \bG^0+\bI_p)^{-1},
\end{equation}
where $\bI_p$ is the identity matrix in $\R^{p\times p}$ and with entry $(a,a')$ of $\mathbf{\Upsilon}^{(\Delta)}$, for $(a,a')\in\{1,\dots,p\}^2$, given by
\begin{equation*}
\Upsilon_{a,a'}{(\Delta)}=\sum_{b,b'=1,\dots,p} (G^{0-1})_{a,b}(G^{0-1})_{a',b'} \bigl(\GIGd_{(a,a'),(b,b')}(\Delta)+\GIGd_{(a,b'),(a',b)}(\Delta)\bigr)
\end{equation*}
where quantities $\GIGd(\Delta)$ are defined by \eqref{eqn:Itilded_mat}.

\item $\vect{\sqrt{n}\left(\hat \bG(\hat\bd)-\bG^0\right)}$ converges in distribution to a centered Gaussian distribution with a variance equal to $\mathbf{V}^{\bGexp(\Delta)}$, with
\begin{equation}
\label{eqn:varG}
V_{(a,b),(a',b')}^{(\bGexp)}(\Delta)=\GIGG_{(a,a'),(b,b')}(\Delta)+\GIGG_{(a,b'),(a',b)}(\Delta)
\end{equation}
where quantities $\GIGG(\Delta)$ are defined by \eqref{eqn:ItildeG_mat}.
\end{itemize}

\end{theorem}

\paragraph{Proof.~\\ }

The properties of wavelet filters in the proofs of \cite{normality} are used through Proposition 31 of \cite{normality}. The inequalities (I1) and (I2) in \cite{normality} correspond respectively to \eqref{eqn:moments_smooth_tauj} and \eqref{eqn:tauj} of Proposition \ref{prop:ineq}. Inequality (I3) in \cite{normality} follows with the proof of Proposition 31 of \cite{normality}. The constants in \eqref{eqn:moments_smooth_tauj} and \eqref{eqn:tauj} do not depend of $L$, which allows to use these inequalities as in the proofs of \cite{normality}.

The other property of the wavelets used in the proofs of \cite{normality}  is the convergence of $\lvert\hat\varphi(2^{-j}\lambda)\rvert$ to 1 when $j$ goes to infinity (page 29 of \cite{normality}). When $L2^{-2j}$ goes to zero, Lemma \ref{lem:phi} yields 
\[\lim_{j\to \infty}\abs{\lvert\hat\varphi(2^{-j}\lambda)\rvert^2-1}=0,\]
which is the desired result.

Finally, the approximation of the sample wavelet covariance is changed. It is sufficient, to use results of Proposition \ref{prop:covY3} (instead of Proposition 1 of \cite{normality}), to check that, for all $(\ell,m)\in\{1,\dots,p\}^2$, \[\lim_{j\to \infty}\sqrt{n_{j}}\abs{2^{-j\,(d_\ell+d_m-\beta)}\cov(W_{j,k}(\ell),W_{j,k}(m)) -
G_{\ell,m}}=0.
\]
Hence, based on Proposition \ref{prop:covY3} and Lemma \ref{lem:nj}, when $N^{-1}2^{j_1}L\to 0$,  it is sufficient to have \[ \lim_{j\to \infty} N^{1/2}2^{-j/2}\bigl(L2^{-2j}+L^{-M-1}\bigr)=0 , \]
since when $j$ goes to infinity, $j_0\leq j\leq j_1$, $n_j$ is equivalent to $N2^{-j}$.

The rest of the proof does not present major changes. It is thus omitted.

\bibliographystyle{abbrvnat}
\bibliography{lrd}

\end{document}